\documentclass[twoside, 12pt]{amsart}

\usepackage[T1]{fontenc}
\usepackage[lf]{Baskervaldx} 
\usepackage[bigdelims,vvarbb]{newtxmath} 

\usepackage[cal=boondoxo]{mathalfa} 

\calclayout

\usepackage{latexsym}
\usepackage{amssymb}
\usepackage{amsmath}
\usepackage{amsthm} 
\usepackage{amscd}
\usepackage{graphicx}
\usepackage[all]{xy}
\usepackage{a4wide}
\usepackage{multirow}
\usepackage{multicol}
\usepackage{color}
\usepackage[dvipsnames]{xcolor}
\usepackage[colorlinks,final,hyperindex]{hyperref}
\usepackage[noabbrev,capitalize]{cleveref}

\usepackage{centernot} 

\makeatletter
\def\@tocline#1#2#3#4#5#6#7{\relax
  \ifnum #1>\c@tocdepth 
  \else
    \par \addpenalty\@secpenalty\addvspace{#2}%
    \begingroup \hyphenpenalty\@M
    \@ifempty{#4}{%
      \@tempdima\csname r@tocindent\number#1\endcsname\relax
    }{%
      \@tempdima#4\relax
    }%
    \parindent\z@ \leftskip#3\relax \advance\leftskip\@tempdima\relax
    \rightskip\@pnumwidth plus4em \parfillskip-\@pnumwidth
    #5\leavevmode\hskip-\@tempdima
      \ifcase #1
       \or\or \hskip 1em \or \hskip 2em \else \hskip 3em \fi%
      #6\nobreak\relax
    \dotfill\hbox to\@pnumwidth{\@tocpagenum{#7}}\par
    \nobreak
    \endgroup
  \fi}
\makeatother

\usepackage{tikz}
\usepackage{tikz-cd}
\usepackage{pgfplots}
\usepackage{pgfplotstable}
\tikzset{math3d/.style=
    {x= {(-0.353cm,-0.353cm)}, z={(0cm,1cm)},y={(1cm,0cm)}}}
\tikzset{JLL3d/.style=
    {x= {(0.4cm,-0.2cm)}, z={(0cm,1cm)},y={(-1cm,0cm)}}}
\usetikzlibrary{calc}
\usetikzlibrary{shapes,shapes.geometric,fit,positioning,calc,matrix}
\tikzset{
  optree/.style={scale=.5,thick,grow'=up,level distance=10mm,inner sep=1pt},
  comp/.style={draw=none,circle,fill,line width=0,inner sep=0pt},
  dot/.style={draw,circle,fill,inner sep=0pt,minimum width=3pt},
  circ/.style={draw,circle,inner sep=1pt,minimum width=4mm},
  emptycirc/.style={draw,circle,inner sep=1pt,minimum width=2mm},
  root/.style={level distance=10mm,inner sep=1pt},
  leaf/.style={draw=none,circle,fill,line width=0,inner sep=0pt},
  nodot/.style={draw,circle,inner sep=1pt},
}

\usepackage{graphicx,calc}
\newlength\myheight
\newlength\mydepth
\settototalheight\myheight{Xygp}
\newcommand*\inlinegraphics[1]{%
  \settototalheight\myheight{Xygp}%
  \settodepth\mydepth{Xygp}%
  \raisebox{-\mydepth}{\includegraphics[height=\myheight]{#1}}%
}

\definecolor{Chocolat}{rgb}{0.09, 0.09, 0.7}
\definecolor{BleuTresFonce}{rgb}{0.215, 0.215, 0.36}
\hypersetup{citecolor=BleuTresFonce, linkcolor=Chocolat}

\newtheorem{definition}{Definition}[section]
\newtheorem{proposition}[definition]{Proposition}
\newtheorem{lemma}[definition]{Lemma}
\newtheorem{theorem}[definition]{Theorem}
\newtheorem{corollary}[definition]{Corollary}
\newtheorem*{problem}{Problem}

\theoremstyle{remark}
\newtheorem{example}[definition]{\sc Example}
\newtheorem{remark}[definition]{\sc Remark}

\newcommand{\RR}{\mathbb{R}}
\newcommand{\ZZ}{\mathbb{Z}}
\newcommand{\NN}{\mathbb{N}}

\DeclareMathOperator{\Ima}{Im} 
\DeclareMathOperator{\cone}{Cone} 

\newcommand{\ZP}{\mathbb{Z}_{>0}}

\newcommand{\La}{\mathcal{L}}

\newcommand{\PT}[1]{\mathrm{PT}_{#1}} 

\newcommand{\PolySub}{\mathsf{Poly}}

\newcommand{\sF}{\mathcal{F}}
\DeclareMathOperator{\tp}{top}
\DeclareMathOperator{\bm}{bot}

\DeclareMathOperator{\conv}{conv}

\newcommand{\tr}{\mathrm{tr}}
\newcommand{\id}{\mathrm{id}}

\DeclareMathOperator{\codim}{codim}

\newcommand{\as}{{\scriptstyle \text{\rm !`}}}

\title{The diagonal of the operahedra}

\author{Guillaume Laplante-Anfossi}
\address{Universit\'e Sorbonne Paris Nord, Laboratoire Analyse, G\'eom\'etrie et Applications, CNRS, UMR 7539, F-93430 Villetaneuse, France.}
\email{laplante-anfossi@math.univ-paris13.fr}

\date{\today}



\thanks{The author was supported by the European Union's Horizon 2020 research and innovation program under the Marie Sklodowska-Curie grant agreement No 754362 \inlinegraphics{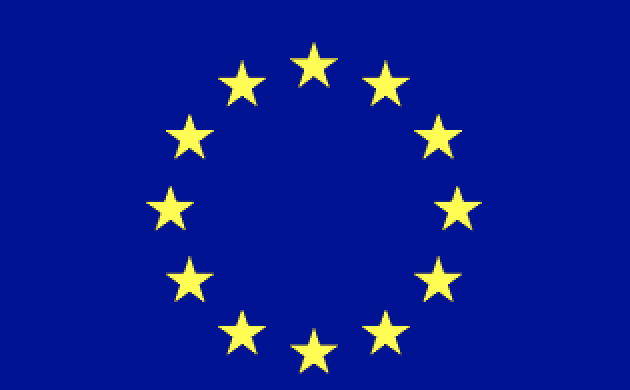}, by the Natural Sciences and Engineering Research Council of Canada (NSERC) and by the ANR-20-CE40-0016 Higher Algebra, Geometry and Topology.}

\begin{document}

\begin{abstract}
The primary goal of this article is to set up a general theory of coherent cellular approximations of the diagonal for families of polytopes by developing the method introduced by N. Masuda, A. Tonks, H. Thomas and B. Vallette. We apply this theory to the study of the operahedra, a family of polytopes ranging from the associahedra to the permutahedra, and which encodes homotopy operads. After defining Loday realizations of the operahedra, we make a coherent choice of cellular approximations of the diagonal, which leads to a compatible topological cellular operad structure on them. This gives a model for topological and algebraic homotopy operads and an explicit functorial formula for their tensor product. 
\end{abstract}

\maketitle

\setcounter{tocdepth}{1}
\tableofcontents

\section*{Introduction} 

\subsection*{State of the art} The present work lies at the intersection of the theory of polytopes and the operadic calculus. The starting point is the following observation: for a non-trivial polytope $P$, the image of the set-theoretic diagonal $\triangle_P:P\to P\times P, x\mapsto (x,x)$ is not a union of faces of $P\times P$. One is led to the problem of finding a cellular approximation to $\triangle_P$, that is finding a cellular map $\triangle_P^{\textrm{cell}} : P \to P\times P$ which is homotopic to $\triangle_P$ and which agrees with $\triangle_P$ on the vertices of $P$, see \cref{figure:cellularapproximation}.

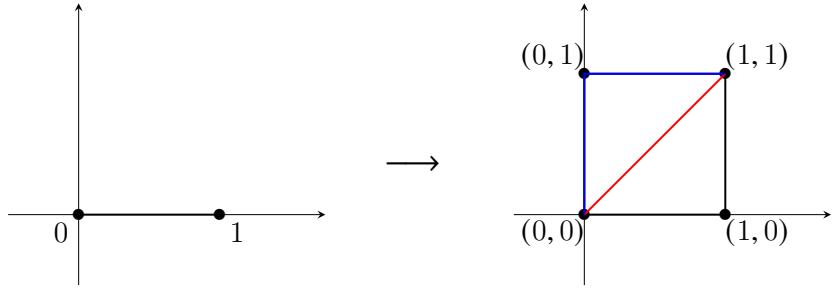
\begin{figure}[h!]
\centering
\resizebox{0.7\linewidth}{!}{
\begin{tikzpicture}[>=stealth]

\path[<-] (3.5,0)node[right]{$ $} edge(-1,0) (0,3)node[above]{$ $} edge(0,-1);

\node at(0,-0.01){$\bullet$};

\node at(-0.25,-0.25){$0$};

\node at(2,-0.01){$\bullet$};

\draw[thick] (0,0)--(2,0);

\node at(2.25,-0.25){$1$};

\end{tikzpicture}
\quad \resizebox{0.05\linewidth}{!}{\raisebox{3em}{$\longrightarrow$}} \quad \quad
\begin{tikzpicture}[>=stealth]

\path[<-] (3.5,0)node[right]{$ $} edge(-1,0) (0,3)node[above]{$ $} edge(0,-1);

\node at(0,-0.01){$\bullet$};

\node at(-0.45,-0.25){$(0,0)$};

\node at(2,-0.01){$\bullet$};

\node at(2.45,-0.25){$(1,0)$};

\node at(2,2-0.01){$\bullet$};

\node at(2.45,2.25){$(1,1)$};

\node at(0,2-0.01){$\bullet$};

\node at(-0.45,2.25){$(0,1)$};

\draw[thick] (0,0)--(2,0)--(2,2)--(0,2)--cycle;

\draw[thick, red] (0,0)--(2,2);

\draw[thick, blue] (0,0)--(0,2)--(2,2);

\end{tikzpicture}}
\caption{The set-theoretic diagonal of the unit interval (in red) is not cellular. One needs to find a cellular approximation (in blue).}
\label{figure:cellularapproximation}
\end{figure}

One can always find such an approximation, however, a problem of fundamental importance in algebraic topology is to find coherent cellular approximations of the diagonal for families of polytopes. 
For instance, the cup product on singular and cubical cohomology comes from coherent cellular approximations of the diagonal of the standard simplices and cubes, the Alexander--Whitney \cite{EilenbergZilber53, EilenbergMacLane54} and Serre \cite{Serre51} maps, respectively. 

More recently, families of greater combinatorial complexity have appeared in operad theory. 
The first seminal example is the family of associahedra. In contrast with the standard simplices and cubes, each face of an associahedron is not itself an associahedron, but a product of lower-dimensional associahedra, which underlies the algebraic structure of an operad. 
More precisely, the cellular chains on the associahedra are naturally endowed with an operad structure which encodes associative algebras up to homotopy \cite{Stasheff63}. 
In light of this fact, finding a family of coherent cellular approximations of the diagonal of the associahedra becomes a very desirable objective, as it defines a functorial tensor product of $A_\infty$-algebras \cite{SaneblidzeUmble04,MarklShnider06}. 
Such a universal formula has applications in different fields of mathematics, for instance the homology of fibered spaces \cite{Proute86}, string field theories \cite{GaberdielZwiebach97} and Fukaya categories \cite{Seidel08}. 

N. Masuda, A. Tonks, H. Thomas and B. Vallette introduced in \cite{MTTV19} a method for finding coherent cellular approximations of the diagonal for families of polytopes, using the theory of fiber polytopes of L. J. Billera and B. Sturmfels \cite{BilleraSturmfels92}. They applied it to the associahedra and obtained a coherent family of approximations, which led to a topological cellular operad structure on them. 
This provided a model for topological $A_\infty$-algebras and an explicit functorial formula for their tensor product. 
Applying the cellular chains functor, it is possible to recover the formula of M. Markl and S. Shnider \cite{MarklShnider06}, which should coincide with the one of \cite{SaneblidzeUmble04}. 
The key feature, which makes this problem highly constrained, is requiring the operadic composition maps to be compatible with the approximation of the diagonal. Such composition maps are in fact unique \cite[Proposition 7]{MTTV19}, and this uniqueness property is precisely the one allowing for the operad structure \cite[Theorem 1]{MTTV19}. 

The diagonal of the associahedra admits a particularly simple description of its cellular image in terms of the Tamari order, so unexpectedly simple that J.-L. Loday was led to the name "magical formula". 
However, one cannot expect a similar formula for other families of polytopes, and an explicit combinatorial description for the cellular image of the approximation of the diagonal of an arbitrary polytope is missing in the work of \cite{MTTV19}.

\subsection*{Present results} 
The first contribution of the present paper is to give such a universal formula, which can be applied to any polytope (\cref{thm:universalformula}), and which is expressed in terms of a new conceptual object: its \emph{fundamental hyperplane arrangement} (\cref{def:fundamentalhyperplane}). 
In the case of the simplices, the theory developed here allows one to recover conceptually a perturbative formula due to M. Abouzaid \cite{Abouzaid09} for the intersection pairing on cellular chains on a manifold, see \cref{rem:Abouzaid}. 
This suggests deeper connections with the intersection theory on toric varieties \cite{FultonSturmfels97} (see \cref{rem:FultonSturmfels}), combinatorial algebraic topology \cite{RomeroSergeraert19,KaufmannMardones21}, higher category theory \cite{KapranovVoevodsky91,MedinaMardones20}, discrete and continuous Morse theory \cite{Forman98,FriedmanMardonesSinha21} and physics \cite{Thorngren18,Tata20}. 

As already mentioned the associahedra encode homotopy associative algebras, and homotopy operads were defined by P. Van der Laan in \cite{VanDerLaan03} as a multi-linear generalization of homotopy associative algebras. 
This leads to thinking about a multi-linear generalization of the associahedra, which encodes homotopy operads. 
Such a generalization was provided by J. Obradovi\'c in \cite{Obradovic19}. 

\begin{center}
\begin{tikzcd}
\text{Associative algebras} \arrow[rr, "\textit{Multi-linear operations}", dashed] \arrow[d] &  & \text{Operads} \arrow[d]                        \\
\text{Associative alg. up to homotopy} \arrow[rr, dashed] \arrow[<->,d, dotted]   &  & \text{Operads up to homotopy} \arrow[<->,d, dotted] \\
\text{Associahedra}    &  & \textit{Operahedra}                            
\end{tikzcd}
\end{center}

The second contribution of this paper is to define Loday realizations of the operahedra, the family of polytopes encoding homotopy operads, and to apply to it the general theory developed in the first part. 
In contrast with existing realizations of the operahedra, these integer-coordinates realizations, which generalize J.-L. Loday's realizations of the associahedra \cite{Loday04a}, present simple geometric properties that ease calculations.  
They allow for the definition of a coherent family of cellular approximations of the diagonal, which lead to a compatible topological cellular operad structure on the family of Loday realizations of the operahedra.
This is the first topological cellular operad structure for this family of polytopes, which provides a model of topological and algebraic operads up to homotopy (\cref{thm:MainOperad}) and an explicit functorial formula for their tensor product (\cref{coroll:functorialtensor}). This formula presents interesting combinatorial properties, and agrees with the magical formula for the associahedra \cite[Theorem 2]{MTTV19}.  

In addition to the associahedra, the operahedra contain yet another important family of polytopes: the permutahedra. The $(n-1)$-dimensional standard permutahedron is defined as the convex hull of all the permutations of $\{1,\ldots,n\}$. 
Closely related to various properties of the symmetric group, it has important applications in algebraic topology, appearing in the study of iterated loop spaces \cite{Milgram66}, $E_n$-operads \cite{Berger97} and topological Hochschild cohomology \cite{McClureSmith03,KaufmannZhang17}.  

In order to define a cellular approximation of the diagonal of the permutahedron, we first have to compute its fundamental hyperplane arrangement (\cref{thm:permutohyperplane}). 
This new hyperplane arrangement refines the braid arrangement and deserves further study. In contrast with the cases of the simplices, the cubes and the associahedra, there are many distinct diagonals that agree with the natural order on the vertices, in this case the weak Bruhat order. So, for the first time, one has to make a choice of approximation. In the case of the operahedra, this choice is further restricted, but not completely determined, by requiring coherence with operadic composition, see \cref{prop:samechambers}.

General geometric arguments show that a choice of approximation of the diagonal for a polytope $P$ gives a choice of approximation for any polytope $Q$ whose normal fan coarsens the one of $P$ (\cref{coroll:coarseningoriented}). 
Moreover, the universal formula for the diagonal of $P$ applies \emph{mutatis mutandis} to $Q$ (\cref{coroll:coarseninguniversal}). 
Since the normal fan of any operahedron is refined by the normal fan of the permutahedron, we restrict our attention to the latter. 
In fact, the preceding argument shows that the formula obtained here applies immediately to all generalized permutahedra \cite{Postnikov09}, which precisely are the polytopes whose normal fans coarsen the one of the permutahedron. 

\subsection*{Future directions} The family of generalized permutahedra include an example of fundamental importance in symplectic topology: the multiplihedra \cite{Stasheff70,Forcey08,Mazuir21}. 
Further, a cellular approximation of the diagonal of the multiplihedra would allow one to define the tensor product of $A_\infty$-categories \cite{LOT20}, which is the subject of an ongoing work with T. Mazuir. Generalized permutahedra also include all graph-associahedra, to which our formula applies immediately. This comprises
\begin{itemize}
    \item the family encoding homotopy modular operads \cite{Ward19},
    \item all the families encoding operadic-like structures described in \cite{BMO20}.
\end{itemize}
Other applications of the theory presented here include 
\begin{itemize}
\item the 2-associahedra which is of great interest in symplectic topology \cite{Bottman18}, 
\item the freehedra encoding representations of a derived algebraic group up to homotopy \cite{AbadCrainicDherin11,AbadCrainic13,Poliakova20}, and
\item the assocoipihedra that intervene in string topology \cite{DrummondColePoirierRounds15,PoirierTradler17,PoirierTradler19}, and which were recently realized as convex polytopes \cite{PilaudPeeble22}.
\end{itemize}

There are already several important examples of operads up to homotopy in the literature. One of them is given by the singular chains of configuration spaces of points in the plane \cite[Section 5]{VanDerLaan03}, which are quasi-isomorphic to the singular chains on the little discs operad. Another closely related structure is the operad of normalized cacti \cite{WahlCacti20}, arising in the study of moduli spaces of Riemann surfaces. The present tensor product applies to both of them.

On the combinatorial side, the operahedra lie at the intersection of many interesting families of polytopes, they can thus be studied from these different perspectives. 
They correspond to
\begin{itemize}
    \item graph-associahedra, where the underlying graph is the line graph of a tree, that is a clawfree block graph \cite[Theorem 8.5]{Harary69},
    \item nestohedra \cite{FeichtnerSturmfels05} and generalized permutahedra \cite{Postnikov09} which can be obtained by removing facet-defining inequalities in the description of the permutahedra \cite{Pilaud14},
    \item a subfamily of hypergraph polytopes \cite{DP11,CIO18,Obradovic19},
    \item a subfamily of poset associahedra \cite{Galashin21}.
\end{itemize}
Standard weight Loday realizations of the operahedra were already defined in a different manner by V. Pilaud in \cite{PilaudSignedTree13}, as part of a broader family which generalizes C. Hohlweg and C. Lange's realizations of the associahedra \cite{HL07}. One can naturally wonder if the techniques of \cite{PilaudSignedTree13} can be extended to all block graph associahedra, which is the subject of ongoing work with V. Pilaud. 

Finally, the present work sheds light on the substitution operation on graph-associahedra defined by S. Forcey and M. Ronco \cite{ForceyRonco19} and prompts applications to Hopf algebra structures on generalized permutahedra \cite{AguiarArdila17}.  

\subsection*{Conventions} 
In the following we use the conventions and notations of \cite{Ziegler95} for convex polytopes and the ones of \cite{LodayVallette12} for operads. Throughout the paper we will consider only planar trees.

\subsection*{Aknowledgements} 
I would like to warmly thank my advisors Eric Hoffbeck and Bruno Vallette for introducing me to the subject, for many invaluable discussions and for their careful reading of the manuscript. 
I am also indebted to Thibaut Mazuir, Arnau Padrol, Vincent Pilaud and Hugh Thomas for numerous discussions and insights. 
The observation leading to \cref{prop:inversionnumber} is due to Christian Gaetz, and the argument of \cref{prop:topbot} was suggested by Arnau Padrol. 
I would like to thank the anonymous referee for their attentive reading, pointing out some errors and providing useful suggestions that have greatly improved the paper.

\section{Cellular approximation of the diagonal of a polytope} \label{section:cellularappoximation}

In this section, we study the method introduced in \cite{MTTV19} for finding a cellular approximation of the diagonal of a polytope and establish its general properties. We associate to any polytope $P$ its \emph{fundamental hyperplane arrangement} $\mathcal{H}_P$, where each chamber defines an approximation of the diagonal. Two different chambers can define the same approximation, and bringing down the walls between them leads to a new notion of "quasi-positively oriented" polytope. 

An approximation of the diagonal of $P$ always exists and it depends only on the normal fan of the polytope. Its image admits a description in terms of the poset structure on the vertices of $P$ induced by the choice of a chamber in $\mathcal{H}_P$. In the case of the associahedra, one recovers the Tamari order. The condition $\tp(F)\leq\bm(G)$ in M. Markl and S. Schnider's "magical formula" for the associahedra \cite{MarklShnider06} turns out to be present in any approximation of the diagonal, but it is not sufficient to characterize its image in general, as shows the case of the permutahedra treated in the next section. 

A careful study of the fundamental hyperplane arrangements leads to a universal formula describing combinatorially the cellular image of the approximation of the diagonal for any polytope $P$. Once one has established the universal formula for $P$, one has in fact established the formula for any polytope $Q$ whose normal fan coarsens the one of $P$.

\subsection{General method} \label{section:generalmethod} 
In the following, we adopt notations and conventions of the monograph of G. M. Ziegler \cite{Ziegler95} on the theory of polytopes. 
Let $P\subset \RR^n$ be a polytope. Except for the case where $P$ is the trivial polytope, the diagonal map 
\begin{equation*}
\begin{matrix}
    \triangle_P & \text{:} & P& \to & P\times P   \\  
    & & z & \mapsto & (z ,z)
\end{matrix}
\end{equation*}
is not cellular, that is, its image is not a union of cells of $P\times P$. 

\begin{problem} Find a \emph{cellular approximation of the diagonal of $P$}, that is, a cellular map which is homotopic to $\triangle_P$ and which coincides with $\triangle_P$ on the vertices of $P$.
\end{problem}

We consider a special case of the fiber polytope construction of L. J. Billera and B. Sturmfels \cite{BilleraSturmfels92}, see also \cite[Chapter 9]{Ziegler95} for more details. Let $\mathcal{L}(P)$ denote the lattice of faces of $P$ and let $(e_i)_{1\leq i\leq n+1}$ denote the standard basis of $\RR^{n+1}$. For a polytope $P\subset \RR^n$ and a vector $\vec v \in \RR^n$, we consider the projection $\pi$ and the linear form $\phi$ defined respectively by 
\setcounter{MaxMatrixCols}{20}
\begin{equation*}
\begin{matrix}
\pi & \text{:} & P\times P & \to & P  & \text{ and } & \phi & \text{:}  & \RR^n \times \RR^n & \to & \RR \\  
& & (x, y) & \mapsto & \tfrac{1}{2}(x+ y) & &  & &  (x, y) & \mapsto & \langle x- y, \vec v \rangle \ .
\end{matrix}
\end{equation*}
The linear map 
\begin{equation*}
\begin{matrix}
\pi^\phi & \text{:} & P\times P & \to & P\times \mathbb{R}  \\  
& & (x, y) & \mapsto & (\pi(x,y),\phi(x,y))  
\end{matrix}
\end{equation*}
defines a polytope $P^\phi\coloneqq \Ima(\pi^\phi) \subset \RR^{n+1}$. Let \[\mathcal{L}^{\downarrow}(P^\phi)\coloneqq \{ F \in \mathcal{L}(P^\phi) \ | \ \forall x \in F, \lambda>0, x - \lambda e_{n+1} \notin P^\phi \} \subset \mathcal{L}(P^\phi)\] be the family of lower faces of $P^\phi$. Then, the set of faces \[ \mathcal{F}^\phi \coloneqq (\pi^\phi)^{-1}\mathcal{L}^{\downarrow}(P^\phi)\subset \mathcal{L}(P\times P)\cong\mathcal{L}(P)\times\mathcal{L}(P) \] induces a subdivision $\pi(\mathcal{F}^\phi)$ of $P$ that is called \emph{coherent}. As indicated by the last isomorphism, the faces of $P\times P$ are pairs of faces of $P$. Depending on the context, we will denote such a pair by $F\times G$ or $(F,G)$. 

One always has $\dim (F\times G)\geq\dim(\pi(F\times G))$ for all $F\times G\in\mathcal{F}^\phi$. The coherent subdivision $\pi(\mathcal{F}^\phi)$ is said to be \emph{tight} if $\dim (F\times G)=\dim(\pi(F\times G))$ for all $F\times G\in\mathcal{F}^\phi$. 

To any tight coherent subdivision $\pi(\mathcal{F}^\phi)$ of $P$ one can associate the unique section $\triangle_{(P,\vec v)}: P\to P\times P$ of $\pi$ which minimizes $\phi$ in each fiber, see \cite[Lemma 9.5]{Ziegler95}. 

\begin{proposition} Let $P\subset\RR^n$ be a polytope. Suppose that $\vec v \in \RR^n$ induces a tight coherent subdivision of $P$. Then, the associated section $\triangle_{(P,\vec v)}$ is a cellular approximation of the diagonal of~$P$. 
\end{proposition}
\begin{proof} If $z$ is a vertex of $P$, then the fiber $\pi^{-1}(z)$ is the point $(z,z)$, so $\triangle_{(P,\vec v)}$ agrees with the set-theoretic diagonal on vertices. An explicit homotopy between the two maps is given by 
\begin{equation*}
\begin{matrix}
        H & \text{:} & P\times [0,1] & \longrightarrow & P\times P   \\  
        & & (z, t) & \longmapsto & (1-t)(z,z)+ t(x,y) 
\end{matrix}
\end{equation*}
where $(x,y)$ is such that $\langle x-y, \vec v\rangle=\min\left\{\phi|_{\pi^{-1}(z)}\right\}$.
\end{proof}

\subsection{Cellular description of the diagonal} Given a cellular approximation $\triangle_{(P,\vec v)}$ of the diagonal of a polytope $P$, one key problem is to describe combinatorially its image.
For more clarity, let us first recall some standard notations.
Let $P\subset \RR^n$ be a polytope. 
Codimension $1$ faces of $P$ are called \emph{facets}. 
For a face $F \in \mathcal{L}(P)$, the \emph{normal cone} of $F$ is the cone 
\[\mathcal{N}_P(F)\coloneqq \left\{ c \in (\RR^n)^{*} \ \bigg | \ F \subseteq \{ x \in P \ | \ c x =\max_{y \in P} c y \}\right\}  \ . \]  
The codimension of $\mathcal{N}_P(F)$ is equal to the dimension of $F$. 
The \emph{normal fan} of $P$ is the collection of the normal cones $\mathcal{N}_P \coloneqq \{\mathcal{N}_P(F) \ | \ F \in \mathcal{L}(P)\setminus\emptyset \}$. 
This fan is complete, i.e. it is a partition of $(\RR^n)^{*}$.
From now on we see $\mathcal{N}_P$ as a subset of $\RR^n$ via the canonical identification $(\RR^n)^{*}\cong\RR^n$.

\begin{remark}
    When $P \subset \RR^n$ is full-dimensional, i.e. when $\dim P = n$, one can alternatively define the normal fan of $P$ via the lines generated by a family of normal vectors to the facets of $P$, called rays. 
    Given a polytope $P$ which is not full-dimensional, one can then consider the restriction to the affine hull of $P$, and define the normal fan in this space.  
    We have decided not to follow this approach, and consider always the normal fan to be full-dimensional.
\end{remark}

For $X$ a subset of $\RR^n$, the \emph{cone} of $X$ is defined by $\cone (X)\coloneqq \{\lambda_1 x_1 + \cdots + \lambda_n x_n \ | \ \{x_1,\ldots,x_n\}\subseteq X, \lambda_i\geq 0 \}$ and its \emph{polar cone} is defined by $X^{*}\coloneqq \{y \in \RR^n \ | \ \forall x \in X, \langle x, y \rangle \leq 0 \}$. 
The following result, which will be at the heart of further developments, applies to \emph{any} coherent subdivision of $P$.

\begin{proposition} \label{prop:coeurcoeur} Let $P$ be a polytope in $\RR^n$, let $\vec v \in \RR^n$ and let $F,G\in \mathcal{L}(P)$ be two faces of $P$. Then, \[ (F,G) \in \mathcal{F}^\phi \iff  \vec v \in \cone ( -\mathcal{N}_P(F)\cup\mathcal{N}_P(G)) \ . \]
\end{proposition}

\begin{proof} Elaborating on the proof of \cite[Proposition 6]{MTTV19}, we have that 
\begin{eqnarray*}
    (F,G)\in \mathcal{F}^{\phi} & \iff & \nexists x\in F, y \in G, \lambda>0 \text{ such that } \left(\tfrac{1}{2}(x + y),\langle x- y, \vec v \rangle - \lambda \right) \in P^\phi \\
    & \iff & \nexists x \in F, y \in G, \vec w \in \RR^n, \varepsilon>0 \text{ such that } \langle \vec v , \vec w \rangle >0 \text{ and } (x-\varepsilon\vec w, y+\varepsilon\vec w) \in P\times P \\
    & \iff & \nexists \vec w \in \RR^n \text{ such that } \langle \vec v , \vec w \rangle >0 \text{ and } \vec w \in -\mathcal{N}_P(F)^{*} \cap \mathcal{N}_P(G)^{*} \\
    & \iff & \forall \vec w \in \cone (-\mathcal{N}_P(F)\cup\mathcal{N}_P(G))^{*} \text{ we have } \langle \vec v , \vec w \rangle \leq 0 \\
    &\iff& \cone (-\mathcal{N}_P(F)\cup\mathcal{N}_P(G))^{*} \subset \cone (\vec v)^{*} \\
    & \iff & \cone (\vec v) \subset  \cone (-\mathcal{N}_P(F)\cup\mathcal{N}_P(G)) \\ 
    & \iff & \vec v \in  \cone (-\mathcal{N}_P(F)\cup\mathcal{N}_P(G)) \ ,
\end{eqnarray*}
where we used that for $X,Y$ two subsets of $\RR^n$, we have $\cone(X)^{*} \cap \cone(Y)^{*}=\cone(X\cup Y)^{*}$ and $\cone(X)^{*} \subset \cone(Y)^{*} \iff \cone(Y)\subset \cone (X)$. 
\end{proof}

\begin{corollary} 
    \label{corollary:perturbation} 
    For all $\varepsilon>0$, we have
\begin{eqnarray*}
    (F,G)\in \mathcal{F}^\phi &\iff& (\mathcal{N}_P(F)+\varepsilon\vec v)\cap \mathcal{N}_P(G)\neq\emptyset  \\
    &\iff& \mathcal{N}_P(F)\cap (\mathcal{N}_P(G)-\varepsilon\vec v) \neq\emptyset  \ . 
\end{eqnarray*}
Moreover, if $P$ is full-dimensional and if the coherent subdivision $\pi(\mathcal{F}^\phi)$ is tight, then the pairs $(F,G)\in\mathcal{F}^\phi$ which satisfy $\dim F + \dim G=\dim P$ are in bijection with the dimension zero cells of $(\mathcal{N}_P\pm\varepsilon\vec v) \cap \mathcal{N}_P$.
\end{corollary}

\begin{proof} The first part of the statement follows directly from \cref{prop:coeurcoeur} : for $\varepsilon>0$, we have by definition of a cone that the inclusion $\cone (\vec v) \subset \cone (-\mathcal{N}_P(F)\cup\mathcal{N}_P(G))$ holds if and only if $\varepsilon \vec v \in \cone (-\mathcal{N}_P(F)\cup\mathcal{N}_P(G))$. 
This is equivalent to the existence of $\lambda\geq 0,\mu \geq 0, f\in \mathcal{N}_P(F)$ and $g \in \mathcal{N}_P(G)$ such that $-\lambda f + \mu g = \varepsilon \vec v$, which proves the claim. For the second part of the statement, if a pair of faces $(F,G)\in \mathcal{F}^\phi$ verifies $\dim((\mathcal{N}_P(F)+\varepsilon\vec v)\cap \mathcal{N}_P(G))=0$, then we have $\dim \mathcal{N}_P(F) + \dim \mathcal{N}_P(G) \leq \dim P$ since $P$ is full-dimensional, so we have $\dim F + \dim G \geq \dim P$. In the case where the subdivision is tight, we must have $\dim F + \dim G =\dim P$, otherwise we would have $\dim(\pi(F\times G))=\dim(F\times G)=\dim F+\dim G> \dim P$, which is impossible since $\Ima(\pi)=P$.
\end{proof}
\cref{corollary:perturbation} is a "perturbative" way of seeing \cref{prop:coeurcoeur} : the pairs of $\mathcal{F}^\phi$ arise as intersections of the normal fan of $P$ with a translated copy of itself in the direction of $\vec v$, see \cref{figure:perturbation}.

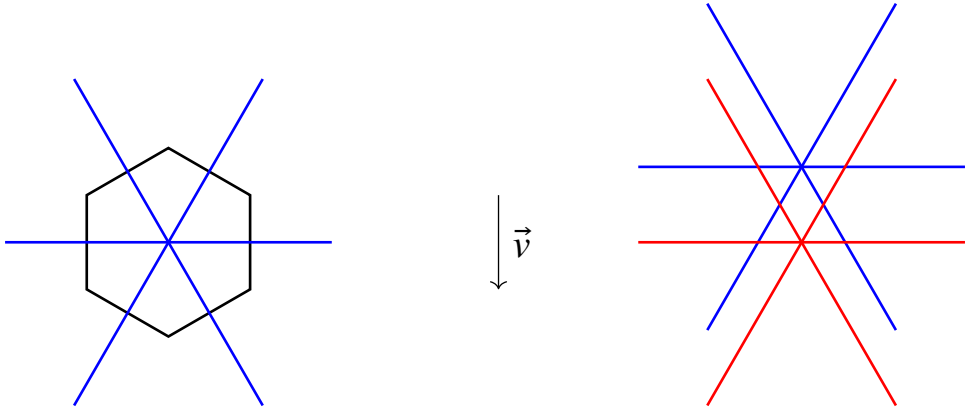
\begin{figure}[h!]
\centering
\resizebox{0.8\linewidth}{!}{
\begin{tikzpicture}
        \draw[thick] (0.866,0.5)--(0,1)--(-0.866,0.5)--(-0.866,-0.5)--(0,-1)--(0.866,-0.5)--cycle;
        \draw[thick,blue] (-1,-1.732)--(1,1.732);
        \draw[thick,blue] (1,-1.732)--(-1,1.732);
        \draw[thick,blue] (-1.732,0)--(1.732,0);
        \draw[->] (3.5,0.5)--(3.5,-0.5);
        \draw (3.5,0) node[right] {$\vec v$};
\end{tikzpicture}
\quad \quad 
\begin{tikzpicture}
    \draw[thick,blue] (-1,-1.732)--(1,1.732);
    \draw[thick,blue] (1,-1.732)--(-1,1.732);
    \draw[thick,blue] (-1.732,0)--(1.732,0);
    \draw[thick,red] (-1,-1.732-0.8)--(1,1.732-0.8);
    \draw[thick,red] (1,-1.732-0.8)--(-1,1.732-0.8);
    \draw[thick,red] (-1.732,0-0.8)--(1.732,0-0.8);
\end{tikzpicture}}
\caption{The normal fan $\mathcal{N}_P$ of the 2-dimensional permutahedron (in blue), and its perturbed copy $\mathcal{N}_P+\varepsilon\vec v$ (in red).}
\label{figure:perturbation}
\end{figure}

By definition, the coherent subdivision $\pi(\mathcal{F}^\phi)$ of $P$ is given by union of the polytopes $(F+G)/2$, for all the pairs of faces $(F,G)\in\mathcal{F}^\phi$. 
In the case where $P$ is full-dimensional, the dual cell decomposition of $\pi(\mathcal{F}^\phi)$ is then isomorphic to $(\mathcal{N}_P+\varepsilon\vec v) \cap \mathcal{N}_P$, see \cref{fig:dualcell}.

\begin{figure}[h!]
\centering
\resizebox{0.5\linewidth}{!}{
\begin{tikzpicture}[xscale=0.5, yscale=0.4]
        
\draw[thick,black] (0,3)--(-2,1)--(-2,-1)--(0,-3)--(2,-1)--(2,1)--cycle;

\draw[thick,gray] (0,-3)--(-2,-5)--(-2,-7)--(0,-9)--(2,-7)--(2,-5)--cycle;
        
\draw[thick,gray] (2,-1)--(4,-3);
\draw[thick,gray] (2,1)--(4,-1);
\draw[thick,black] (2,-5)--(4,-3);
\draw[thick,black] (4,-1)--(4,-3);
\draw[thick,black] (2,-7)--(4,-5);
\draw[thick,gray] (4,-5)--(4,-3);

\draw[thick,gray] (-2,-1)--(-4,-3);
\draw[thick,gray] (-2,1)--(-4,-1);
\draw[thick,black] (-2,-5)--(-4,-3);
\draw[thick,black] (-4,-1)--(-4,-3);
\draw[thick,black] (-2,-7)--(-4,-5);
\draw[thick,gray] (-4,-5)--(-4,-3);

\draw[thick,blue] (-5,-2)--(-4,-2)--(-2,0)--(2,0)--(4,-2)--(5,-2);
\draw[thick,blue] (-3,-8)--(-3,-4)--(-2,-3)--(0,0)--(3,4);
\draw[thick,blue] (3,-8)--(3,-4)--(2,-3)--(0,0)--(-3,4);

\draw[thick,red] (-5,-4)--(-4,-4)--(-2,-6)--(2,-6)--(4,-4)--(5,-4);
\draw[thick,red] (-3,2)--(-3,-2)--(-2,-3)--(0,-6)--(3,-10);
\draw[thick,red] (3,2)--(3,-2)--(2,-3)--(0,-6)--(-3,-10);

\node at (0,-0.1) {\textbullet} ;
\node at (0,-6.1) {\textbullet} ;
\node at (2,-3.1) {\textbullet} ;
\node at (3,-1.1) {\textbullet} ;
\node at (3,-5.1) {\textbullet} ;
\node at (-2,-3.1) {\textbullet} ;
\node at (-3,-1.1) {\textbullet} ;
\node at (-3,-5.1) {\textbullet} ;

\draw[->] (6.5,-1.5)--(6.5,-4.5);
\draw (7.5,-3) node[right] {$\vec v$};
\end{tikzpicture}}
\caption{A tight coherent subdivision of the 2-dimensional permutahedron and its dual cell decomposition.}
\label{fig:dualcell}
\end{figure}
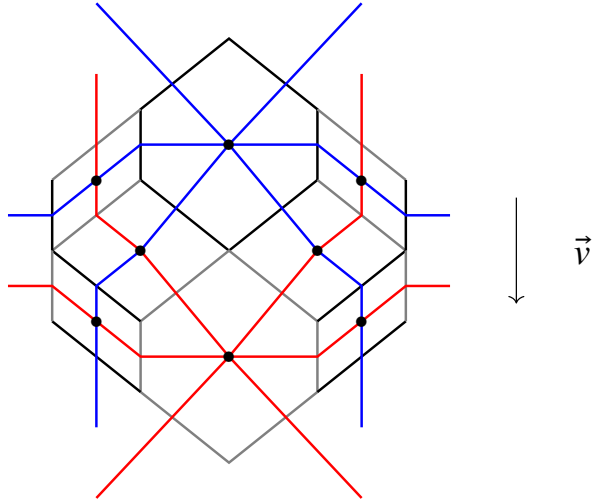

\begin{remark} \label{rem:Abouzaid}
In the case of the simplices, one recovers via \cref{corollary:perturbation} the classical equivalence between the cup product on the simplicial cochains of a triangulation of a manifold and the intersection pairing on cellular chains, as described by M. Abouzaid in \cite[Appendix E]{Abouzaid09}. 
We denote by $(e_i)_{1\leq i\leq n}$ the standard basis of $\RR^n$ and we set $e_0\coloneqq 0$. 
Let us consider the following full-dimensional realization of the $n$-simplex \[\Delta^n \coloneqq\conv (\{e_i \ | \ 0\leq i\leq n \}) \ . \] The rays of the normal fan $\mathcal{N}_{\triangle^n}$ are generated by the vectors $\{-e_i \ | \ 1\leq i \leq n\}$ and $(1,\ldots,1)$. We fix some $\varepsilon>0$, and define $\vec v=(\varepsilon/n,\ldots,\varepsilon/2,\varepsilon)$. Then, \cite[Lemma E.4]{Abouzaid09} shows that the non-empty intersections of $\mathcal{N}_{\Delta^n}\cap (\mathcal{N}_{\triangle^n}-\varepsilon \vec v)$ coincide with the formula for the Alexander-Whitney map. By means of \cref{corollary:perturbation}, this is exactly what we would obtain by proving that $\vec v$ induces a tight coherent subdivision of $\Delta^n$.
\end{remark}

\begin{remark}
\label{rem:FultonSturmfels}
\cref{corollary:perturbation} seems to be very closely related to the Fulton--Sturmfels formula, or "fan displacement rule" \cite[Theorem 4.2]{FultonSturmfels97}, which plays a key role in the intersection theory of toric varieties.
One of our future goals is to relate precisely the two constructions.
\end{remark}

We aim now at giving a geometric meaning to the cone $\cone(-\mathcal{N}_P(F)\cup\mathcal{N}_P(G))$ that appears in \cref{prop:coeurcoeur}.
We denote by $\mathring P$ the \emph{relative interior} of a polytope $P$, that is the interior of $P$ with respect to its embedding into its affine hull. 

\begin{lemma}\label{lemma:intersection} Let $P,Q\subset\RR^n$ be two polytopes. There is a bijection \[\mathcal{L}(P\cap Q)\cong \{(F,G)\in\mathcal{L}(P)\times\mathcal{L}(Q) \ | \ \mathring F \cap \mathring G \neq \emptyset \} \ . \] Moreover, for any face $F\cap G\in\mathcal{L}(P\cap Q)$, we have \[\mathcal{N}_{P\cap Q}(F\cap G)=\cone(\mathcal{N}_P(F)\cup\mathcal{N}_Q(G)) \ . \]
\end{lemma}
\begin{proof}
Any polytope $P\subset \RR^n$ is a bounded intersection of facet-defining closed halfspaces, one for each facet of $P$, and of the affine hull of $P$. Each halfspace has a support hyperplane. Let $x$ be a point in the interior of a face of $P\cap Q$. Then $x$ is in the support hyperplanes $H_i$ of $P$ for a certain subset $I$ and also in the support hyperplanes $H_j$ of $Q$ for a certain subset $J$. Thus $x$ is in the face $F$ of $P$ defined by the $H_i$ and in the face $G$ of $Q$ defined by the $H_j$. For the second part of the statement, we observe that the normal cone $\mathcal{N}_P(F)$ of a face $F$ is spanned by the normal vectors of the support hyperplanes defining that face and any basis of the orthogonal complement of $P$ in $\RR^n$, and the result follows.
\end{proof}

\begin{definition} Let $P\subset \RR^n$ be a polytope. For $z \in P$, we denote by $\rho_z P \coloneqq 2z-P$ the reflection of $P$ with respect to $z$, see \cref{fig:permutoreflex}.
\end{definition}

\begin{figure}[h!]
    \centering
    \includegraphics[width=0.6\linewidth]{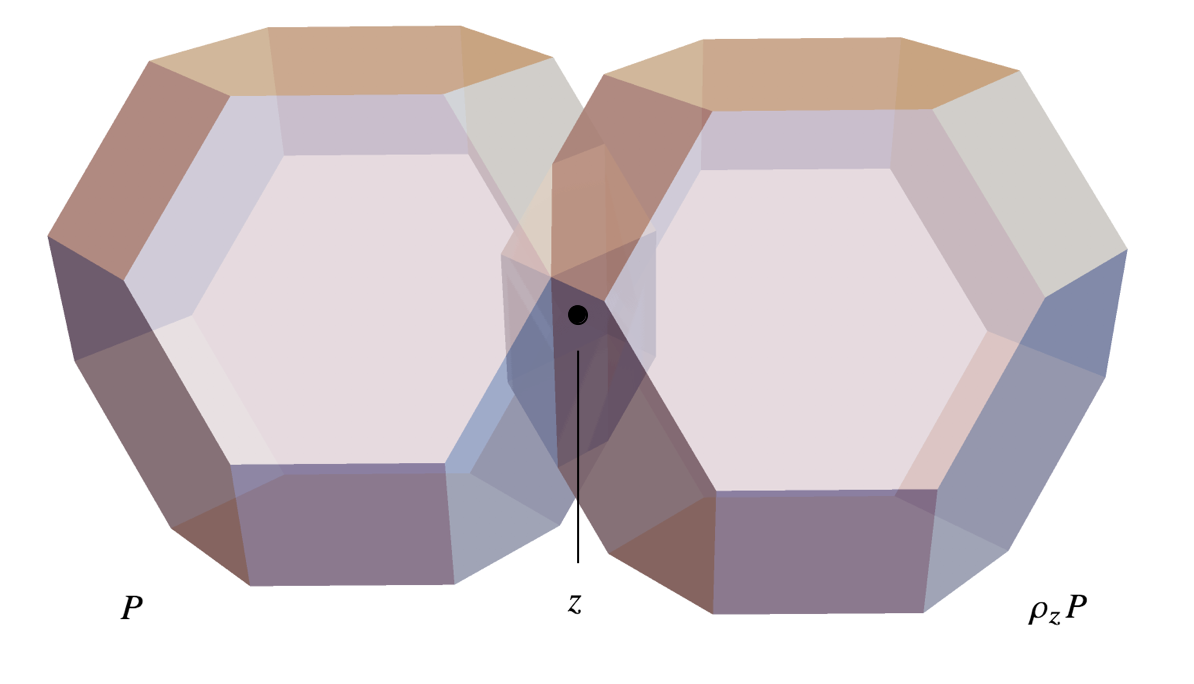} 
    \caption{A 3-dimensional permutahedron $P$, its reflection $\rho_z P$ and the intersection $P\cap \rho_z P$. }
    \label{fig:permutoreflex}
\end{figure}

\begin{proposition} \label{prop:normalintersection} Let $P\subset \RR^n$ be a polytope, and let $F,G$ be two faces of $P$. For any  $z,w\in (\mathring F + \mathring G)/2$, we have \[\mathcal{N}_{P\cap\rho_zP}(G\cap\rho_z F)=\mathcal{N}_{P\cap\rho_wP}(G\cap\rho_w F)=\cone(-\mathcal{N}_P(F)\cup\mathcal{N}_P(G)) \ . \]
\end{proposition}
\begin{proof} The result follows directly from the application of \cref{lemma:intersection} to the intersection $P\cap\rho_z P$, and the fact that for any face $F$ of $P$ and any $z\in P$ we have $\mathcal{N}_{\rho_zP}(\rho_z F)=-\mathcal{N}_P(F)$. 
\end{proof}

\begin{corollary} Let $P\subset\RR^n$ be a polytope, let $\vec v\in \RR^n$. For two faces $F,G$ of $P$, we have 
\begin{eqnarray*}
    (F,G)\in\mathcal{F}^\phi & \iff & \forall z \in \tfrac{\mathring F + \mathring G}{2}, \ \vec v \in \mathcal{N}_{P\cap\rho_zP}(G\cap\rho_z F) \\
    &\iff& \exists z \in \tfrac{\mathring F + \mathring G}{2}, \ \vec v \in \mathcal{N}_{P\cap\rho_zP}(G\cap\rho_z F)  \ . 
\end{eqnarray*}   
\end{corollary}
\begin{proof} The result is obtained by combining \cref{prop:coeurcoeur,prop:normalintersection}.
\end{proof}

\subsection{Pointwise description of the diagonal} We are interested in answering the following question: which choice of vector $\vec v$ gives a tight coherent subdivision of $P$?

\begin{definition}[Quasi-oriented polytope] A polytope $P\subset \RR^n$ is \emph{quasi-oriented} by $\vec v \in \RR^n$ if the linear form $\langle -, \vec v \rangle$ has a unique minimal element $\bm(P)$ and a unique maximal element $\tp(P)$ in $P$. 
\end{definition}

\begin{definition}[Oriented polytope] \label{def:orientedpolytope} A polytope $P\subset\RR^n$ is \emph{oriented} by $\vec v\in\RR^n$ if $\vec v$ is not perpendicular to any edge of $P$. 
\end{definition}

An orientation vector induces a poset on the vertices of $P$, for which the oriented 1-skeleton of $P$ is the Hasse diagram. Dually, it corresponds to a poset structure on the maximal cones of the normal fan $\mathcal{N}_P$. We observe that if $P$ is oriented by $\vec v$, then so is any face of $P$. 
Any oriented polytope is quasi-oriented, but the converse in not true in general. Consider the 3-dimensional cross-polytope $\lozenge_3\coloneqq\conv(e_1,-e_1,e_2,-e_2,e_3,-e_3)$, and choose $\vec v\coloneqq e_3$. Then, $(\lozenge_3,\vec v)$ is quasi-oriented but not oriented, since $\vec v$ is perpendicular to the four edges contained in the $xy$-plane. 

\begin{definition}[Positively and quasi-positively oriented polytope] 
    A polytope $P\subset\RR^n$ is \emph{positively oriented} (resp. \emph{quasi-positively oriented}) by $\vec v\in \RR^n$ if for any $z\in P$, the intersection $P\cap \rho_z P$ is oriented (resp. quasi-oriented) by $\vec v$. 
\end{definition}

Any positively oriented polytope is quasi-positively oriented, but the converse is not true in general, see \cref{example:pyramid}.
We note that any quasi-positively oriented polytope is also oriented. To see this, let $e$ be an edge from a vertex $x$ to a vertex $y$ in $P$, and set $z:= (x + y)/2$. Then $P\cap\rho_z P=e$ is quasi-oriented by $\vec v$, so $\vec v$ is not perpendicular to $e$.

\begin{proposition} \label{prop:quasiposifftight} Let $P$ be a polytope. Then, \[ (P,\vec v) \text{ is quasi-positively oriented } \iff \pi(\mathcal{F}^\phi) \text{ is tight. }  \]
\end{proposition}
\begin{proof} We read the proof of \cite[Proposition 5]{MTTV19} with a new perspective. We have that $\pi(\mathcal{F}^\phi)$ is tight if and only if for any $z \in P$, the fiber $\pi^{-1}(z)=\{(x, y)\in P\times P \ | \ x+ y=2 z \}$ admits a unique minimal element with respect to $\phi$. Since the sum of $x+y$ is constant, $\phi(x, y)$ is minimized in $\pi^{-1}(z)$ if and only if $\langle x,\vec v\rangle$ is minimized and $\langle y,\vec v\rangle$ is maximized. On both coordinates, $\pi^{-1}(z)$ projects down to the intersection $P\cap\rho_zP$. So, the fiber $\pi^{-1}(z)$ admits a unique minimal element with respect to $\phi$ if and only if $P\cap\rho_z P$ admits a unique pair of minimal and maximal elements with respect to $\langle - , \vec v \rangle$. 
\end{proof}

In summary, we have the chain of implications showed in \cref{fig:implications}.
\begin{figure}[h!]
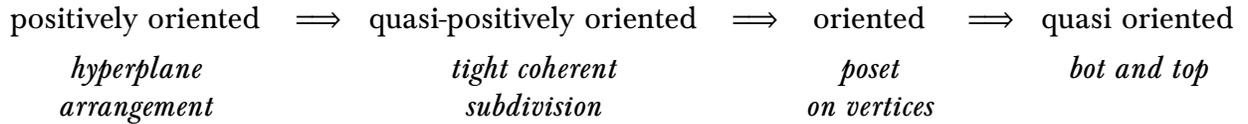

\begin{equation*}
\begin{matrix}
\medskip
\text{positively oriented} & \implies & \text{quasi-positively oriented} & \implies & \text{oriented} & \implies & \text{quasi oriented}  \\ 
\textit{hyperplane} & & \textit{tight coherent} & & \textit{poset} & & \textit{bot and top} \\
\textit{arrangement} & & \textit{subdivision} & & \textit{on vertices} & &  \\
\end{matrix}
\end{equation*}
\caption{The different notions of orientation and their associated properties.}
\label{fig:implications}
\end{figure}

In the case where $(P,\vec v)$ is quasi-positively oriented, the proof of \cref{prop:quasiposifftight} gives the following pointwise description of $\triangle_{(P,\vec v)}$.

\begin{proposition}[Bot-top diagonal]\label{def:Diag}
The map $\triangle_{(P,\vec v)}$ associated to a quasi-positively oriented polytope $(P, \vec v)$ admits the following pointwise description 
\begin{align*}
\begin{array}{rlcl}
\triangle_{(P,\vec v)}\  : & P &\to  &P\times P\\
&z & \mapsto& 
\bigl(\bm(P\cap \rho_z P),\,  \tp(P\cap \rho_z P)\bigr) \ .
\end{array}
\end{align*}
We call it the \emph{bot-top diagonal} of $(P,\vec v)$. 
\end{proposition}

\begin{example} 
\label{example:pyramid} 
We consider the pyramid \[P=\conv((0,0,1),(-1,0,0),(0,1.5,-0.5),(0,-1.5,-0.5),(3,0,-2))\subset \RR^3 \ , \] shown in \cref{fig:pyramide}, and we set $\vec v=(0,0,1)$. We claim that $\pi(\mathcal{F}^\phi)$ is tight while $(P, \vec v)$ is not positively oriented. 
For the second assertion, we let $z=0$ and we observe that four edges of $P\cap\rho_z P$ lie in the $xy$ plane and are thus perpendicular to $\vec v$. 
For the first assertion, we first observe that directions of the rays of $\mathcal{N}_P$ are given by $(1,1,1),(-1,1,1),(-1,-1,1),(1,-1,1)$ and $(-0.5,0,-1)$. 
Then, one can compute that for any pair of faces $(F,G)$ with $\dim F + \dim G > \dim P$, we have $\vec v \notin \cone(-\mathcal{N}_P(F)\cup\mathcal{N}_P(G))$. 
We conclude with \cref{prop:coeurcoeur}.
\end{example}

\begin{figure}[h!]
    \centering
    \includegraphics[width=0.4\linewidth]{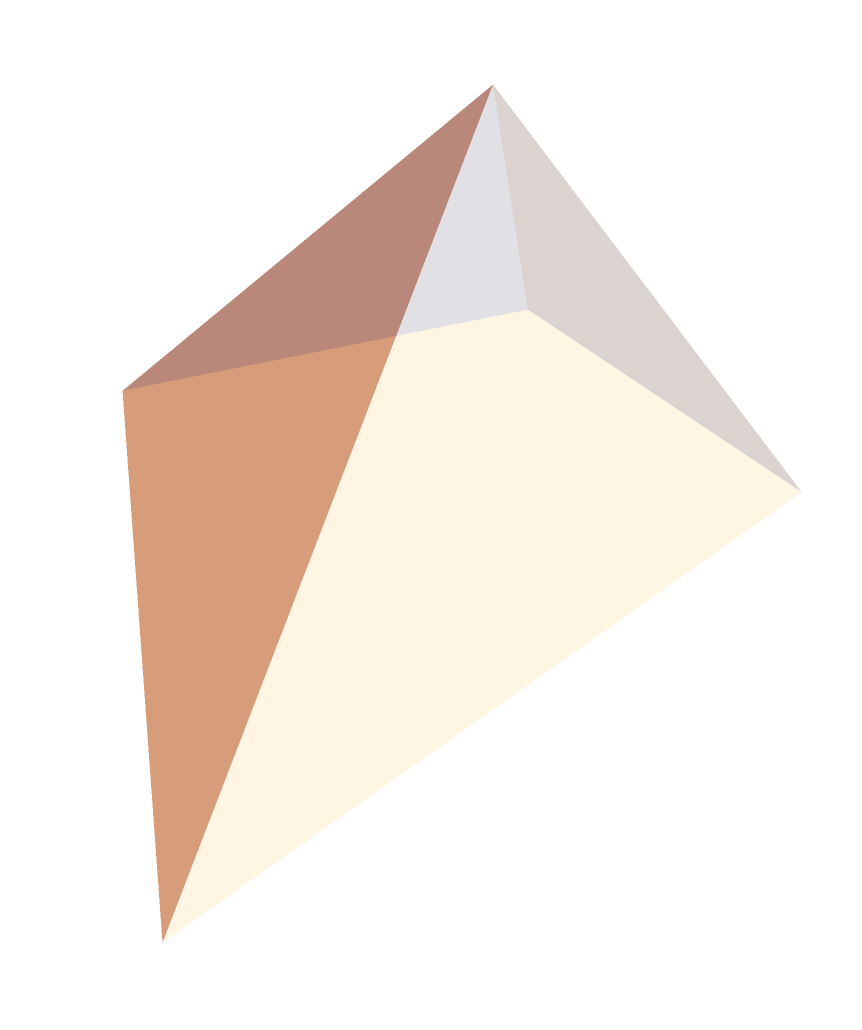} 
    \caption{The pyramid $P$ described in \cref{example:pyramid} is quasi-positively oriented by $\vec v=(0,0,1)$ but not positively oriented.}
    \label{fig:pyramide}
\end{figure}

\subsection{Poset description of the diagonal} 

\begin{proposition} \label{prop:topbot} Let $(P,\vec v)$ be an oriented polytope. Then, 
    \[(F,G)\in \mathcal{F}^\phi \implies \tp(F)\leq\bm(G) \ . \]
\end{proposition}
\begin{proof} \cref{corollary:perturbation} asserts that $(F,G) \in \mathcal{F}^\phi \iff \forall \varepsilon>0 \ \exists x \in \mathcal{N}_P(F)$ such that $x+\varepsilon\vec v \in \mathcal{N}_P(G)$. 
    \begin{enumerate}
        \item Suppose that $F$ and $G$ are vertices of $P$. Let $\varepsilon>0$ and choose $x \in \mathcal{N}_P(F)$ such that $x+\varepsilon\vec v \in \mathcal{N}_P(G)$. We consider the segment $\ell=\{x+t\vec v \ | \ t\in [0,\varepsilon]\}$ and the linearly ordered set of maximal cones of $\mathcal{N}_P$ crossed by $\ell$. They determine a sequence of vertices $F=F_1,F_2,\ldots,F_k=G$. We claim that $F_i\leq F_{i+1}$ for all $i$. Indeed, when $\ell$ goes from $\mathcal{N}_P(F_i)$ to $\mathcal{N}_P(F_{i+1})$, it intersects the interior of a cone $\mathcal{N}_P(E)$, where $E$ is a face with $\dim(E)\geq 1$. Since $\vec v$ orients $P$, this intersection is a point. So we must have $F_i=\bm(E)$ and $F_{i+1}=\tp(E)$. 
        
        \item For general faces $F$ and $G$, we have $(F,G)\in \mathcal{F}^\phi \implies (\tp(F),\bm(G))\in \mathcal{F}^\phi$ and we can apply the preceding point. 
    \end{enumerate}
\end{proof}

Applying the present method to the cubes, the standard simplices and the associahedra as in \cite[Example 1 and Theorem 2]{MTTV19}, one obtains a characterization of the form $(F,G)\in \mathcal{F}^\phi \iff \tp(F)\leq\bm(G)$. \cref{prop:topbot} shows that these are the "simplest" possible formulas.  
In \cref{section:diagonal} we will study a family of examples where these formulas are no longer sufficient to characterize the image of the diagonal.

\subsection{Fundamental hyperplane arrangement and universal formula}

We now restrict our attention to a positively oriented polytope $(P,\vec v)$. For such a polytope, the orientation vector $\vec v$ does not live in any linear space orthogonal to an edge of $P\cap\rho_z P$, for any $z \in P$. That is, $\vec v$ lives in a chamber of the following hyperplane arrangement. 

\begin{definition}[Fundamental hyperplane arrangement] \label{def:fundamentalhyperplane} Let $P\subset \RR^n$ be a polytope. The \emph{fundamental hyperplane arrangement} $\mathcal{H}_P$ of $P$ is the collection of hyperplanes in $\RR^n$ orthogonal to the directions of the edges of $P\cap\rho_z P$, for all $z \in P$. 
\end{definition}

Here, a \emph{direction} of an edge is an arbitrary choice of vector spanning its affine hull. The fundamental hyperplane arrangement is \emph{central}, i.e. every hyperplane $H\in \mathcal{H}_P$ contains the origin. We call the interior of a maximal cone in $\mathcal{H}_P$ a \emph{chamber}. We observe that $\mathcal{H}_P$, considered as a fan, refines both the normal fan $\mathcal{N}_P$ of $P$ and its opposite $-\mathcal{N}_P$.

\begin{example}[The cubes] 
    The fundamental hyperplane arrangement of the $n$-dimensional cube $C^n=[0,1]^n \subset \RR^n$ is the set of coordinate hyperplanes $\mathcal{H}_{C^n}=\{x_i=0 \ | \ 1\leq i \leq n\}$. In this case, for any $z \in C^n$, the edges of $C^n\cap\rho_zC^n$ are all parallel to some edges of $C^n$, so $\mathcal{H}_{C^n}$ is just the set of hyperplanes perpendicular to the directions of the edges of $C^n$.
\end{example}

\begin{example}[The simplices] 
    The fundamental hyperplane arrangement of the $n$-dimensional standard simplex $\Delta^n=\{(x_0,\ldots,x_n)\in \RR^{n+1} \ | \ x_0+\cdots + x_n=1\}$ (which is distinct from the realization considered in \cref{rem:Abouzaid}) is the braid arrangement $\mathcal{H}_{\Delta^n}=\{x_i=x_j \ | \ 0\leq i<j \leq n\}$. Here again, it corresponds to the set of hyperplanes perpendicular to the directions of the edges of $\Delta^n$.
\end{example}

\begin{example}[The associahedra] 
    The fundamental hyperplane arrangement of the Loday realization of the $n$-dimensional associahedron $K^n\subset \RR^{n+1}$ is the following refinement of the braid arrangement \[\left\{x_{i_1}+x_{i_3}+\cdots+x_{i_{2k-1}}=x_{i_2}+x_{i_4}+\cdots+x_{i_{2k}} \ | \ 1\leq i_1<i_2<\cdots<i_{2k}\leq n+1, \ 1\leq k\leq \lfloor \tfrac{n+1}{2} \rfloor \right\} \ . \] In contrast with the preceding examples, new directions of edges appear when considering $K^n\cap\rho_z K^n$ for some $z \in K^n$, see \cite[Proposition 2]{MTTV19}. 
\end{example}

The following proposition is useful for computing the fundamental hyperplane arrangement of a polytope. 
\begin{proposition} \label{prop:edgesPP} Let $P$ be a polytope. There is a surjection
\begin{eqnarray*}
\begin{Bmatrix}
\text{pair of faces } (F,G) \text{ of } P \\
\text{with } \codim(\cone(-\mathcal{N}_P(F)\cup\mathcal{N}_P(G)))=1  
\end{Bmatrix}
\twoheadrightarrow
\begin{Bmatrix}
\text{ direction } \vec d \text{ of an edge of } P\cap\rho_z P  \\
\text{ for some } z \in P 
\end{Bmatrix}_{\big{/}\sim} \ ,
\end{eqnarray*}
where we identify in the target two directions which are scalar multiples of each other. 
\end{proposition}
\begin{proof} Let $(F,G)$ be a pair of faces of $P$ such that $\codim(\cone(-\mathcal{N}_P(F)\cup\mathcal{N}_P(G))=1$. By \cref{prop:normalintersection}, this condition is equivalent to $\codim(\mathcal{N}_{P\cap\rho_z P}(G\cap\rho_z F))=1$, where $z \in (\mathring{F}+\mathring{G})/2$. Thus, a pair of faces satisfying this codimension 1 condition defines an edge $G\cap\rho_z F$ of $P\cap\rho_zP$ and the application above is well-defined. Now by \cref{lemma:intersection}, any edge of $P\cap\rho_zP$ arises as the intersection $G\cap\rho_z F$ for some pair of faces $(F,G) \in P\times P$, so the application is also surjective. 
\end{proof}

Now we aim at extracting a combinatorial formula for the cellular image of the diagonal from the geometry of the fundamental hyperplane arrangement.

\begin{proposition}[Chamber invariance] \label{prop:chamberinvariance} Let $P\subset \RR^n$ be a polytope. Two vectors $\vec v$ and $\vec w$ belonging to the same chamber of $\mathcal{H}_P$ define the same bot-top diagonal, that is \[ \triangle_{(P,\vec v)}=\triangle_{(P,\vec w)} \ . \]  
\end{proposition}
\begin{proof}
Suppose that $\vec v$ and $\vec w$ are such that $\triangle_{(P,\vec v)}\neq\triangle_{(P,\vec w)}$. This means that there is a point $z\in P$ for which $\bm_{\vec v}(P\cap\rho_z P)\neq \bm_{\vec w}(P\cap\rho_z P)$ or $\tp_{\vec v}(P\cap\rho_z P)\neq \tp_{\vec w}(P\cap\rho_z P)$. So, there is an edge $e$ of $P\cap\rho_z P$ such that $\vec v$ and $\vec w$ determine two different orientations of $e$. If $e$ has direction $\vec d$, this means that $\langle \vec d, \vec v \rangle$ and $\langle \vec d, \vec w \rangle$ have opposite signs. Thus $\vec v$ and $\vec w$ lie on opposite sides of the hyperplane $H=\{x \in \RR^n | \ \langle \vec d, x \rangle = 0\} \in \mathcal{H}_P$. 
\end{proof}

\begin{remark}
We note that the converse of \cref{prop:chamberinvariance} does not hold in general, that is, two distinct chambers in $\mathcal{H}_P$ can determine the same bot-top diagonal. This is due to the fact that the condition of being positively oriented is strictly stronger than being quasi-positively oriented. 
\end{remark}

For a vector $\vec d \in \RR^n$ defining an hyperplane $H=\{x \in \RR^n \ | \ \langle \vec d, x \rangle = c \in \RR \}$, we set the notation $H^{+} \coloneqq \{x \in \RR^n \ | \ \langle \vec d , x \rangle > c \}$.

\begin{definition}[Outward pointing normal vector]
\label{def:normalpointing}
Let $F$ be a facet of a polytope $P\subset \RR^n$. 
A vector $\vec d \in \RR^n$ is said to be an \emph{outward pointing normal vector} for $F$ if it defines an hyperplane $H$ such that $P \cap H = F$ and $P \cap H^{+} = \emptyset$.
\end{definition}

Recall that a face $F$ of a polytope $P$ is equal to the intersection of a family of facets $\{F_i\}$.
If we choose an outward pointing normal vector $\vec F_i$ for each facet $F_i$, then the normal cone of $F$ is spanned by these normal vectors together with a basis $\{b_k\}$ of the orthogonal complement of the affine hull of $P$ in $\RR^n$, i.e. we have $\mathcal{N}_P(F)=\cone(\{\vec F_i\} \cup \{ b_k, - b_k\})$. 

For a pair of faces $F,G$ of $P$, let us set the notation \[\mathcal{H}_P(F,G)\coloneqq\{H \in \mathcal{H}_P \ | \ H\text{ intersects the interior of a codimension 1 face of }\cone(-\mathcal{N}_P(F)\cup\mathcal{N}_P(G))\} \ . \] 

\begin{theorem}[Universal formula for the bot-top diagonal] 
\label{thm:universalformula} 
Let $(P,\vec v)$ be a positively oriented polytope in $\RR^n$. 
For each $H\in\mathcal{H}_P$, we choose a normal vector $\vec d_H$ such that $\langle \vec d_H, \vec v \rangle >0$. 
We have 
\begin{eqnarray}
(F,G) \in \Ima \triangle_{(P,\vec v)} 
&\iff& \forall H \in \mathcal{H}_P(F,G), \ \exists i , \ \langle \vec F_i, \vec d_H \rangle < 0  \text{ or } \exists j , \ \langle \vec G_j, \vec d_H \rangle > 0  \\ 
&\iff& \forall H \in \mathcal{H}_P , \ \exists i , \ \langle \vec F_i, \vec d_H \rangle < 0  \text{ or } \exists j , \ \langle \vec G_j, \vec d_H \rangle > 0 \ . 
\end{eqnarray} 
\end{theorem}
\begin{proof} 
Let us write $\cone(-F,G)\coloneqq \cone(-\mathcal{N}_P(F)\cup\mathcal{N}_P(G))=\cone(\{\vec F_i\} \cup \{\vec G_j\} \cup \{b_k,-b_k\})$ and let us denote by $C$ the chamber of $\mathcal{H}_P$ containing $\vec v$. Combining \cref{prop:coeurcoeur,prop:chamberinvariance} we have that $(F,G) \in \Ima\triangle_{(P,\vec v)} \iff \vec v \in \cone(-F,G) \iff C \subset \cone(-F,G)$. 
Moreover, we recall from \cref{prop:normalintersection} that
 we have $\cone(-F,G)=\mathcal{N}_{P\cap\rho_z P}(G\cap\rho_z F)$ for any $z \in (\mathring F + \mathring G)/2$. 
We observe that since every $\vec d_H$ lives in the affine hull of $P$, we have $\langle \vec d_H, b_k \rangle =0$ for all $H$ and for all $k$, so we can focus on the families of $\vec F_i$'s and $\vec G_j$'s and distinguish two cases. 

If $\dim(G\cap\rho_z F)\geq 1$, both sides of (1) are false and thus equivalent. Indeed, in this case $\cone(-F,G)$ is not full-dimensional, so it cannot contain $C$, which is full-dimensional. Moreover, $\cone(-F,G)$ belongs to all hyperplanes $H\in\mathcal{H}_P(F,G)$, which implies $\langle \vec F_i ,\vec d_H \rangle =\langle \vec G_j ,\vec d_H \rangle=0$ for all $i,j$. The same argument applies to (2), since $\mathcal{H}_P(F,G)\subset\mathcal{H}_P$.

Suppose now that $\dim(G\cap\rho_z F)=0$. In this case $\cone(-F,G)$ is full-dimensional and its bounding hyperplanes are precisely the hyperplanes perpendicular to the edges of $P\cap\rho_z P$ which are adjacent to $G\cap\rho_z F$, that is, the hyperplanes of $\mathcal{H}_P(F,G)$. By definition, we have $C\subset H^{+}$ for all $H\in\mathcal{H}_P$. We examine the first implication ($\implies$) of (2). Suppose that $(F,G)\in\Ima\triangle_{(P,\vec v)}$, and let $H\in\mathcal{H}_P$. Since $C$ is full-dimensional, we have $C\subset\cone(-F,G)\implies \cone(-F,G)\cap H^{+}\neq\emptyset$. In particular, there exists an outward pointing normal vector in $\cone(-F,G)$ which has a strictly positive scalar product with $\vec d_H$, hence the right hand side of (2). This implies the right hand side of (1). Now we prove the reverse implication ($\impliedby$) of (1) by contraposition. If $C \nsubset \cone(-F,G)$, then \cref{prop:chamberinvariance} implies that $C\cap\cone(-F,G)=\emptyset$. In this case, there exists an $H\in\mathcal{H}_P(F,G)$ such that $\cone(-F,G)\subset\RR^n\setminus H^{+}$. Indeed, if we had $\cone(-F,G)=\cap_{H\in\mathcal{H}_P(F,G)}\overline{H^{+}}$, where $\overline{H^{+}}\coloneqq\{x \in \RR^n | \ \langle \vec d_H, x \rangle \geq 0\}$, then we would have $C=\cap_{H\in\mathcal{H}_P}H^{+}\subset \cone(-F,G)$ which is impossible. So the scalar product of any outward pointing normal vector in $\cone(-F,G)$ with $\vec d_H$ has a nonpositive value.
\end{proof}

In practice, one uses \cref{thm:universalformula} by first computing the directions $\vec d_H$ from the normal fan of $P$, and then applying (2). The equivalence (1) is of a more conceptual nature: it says that strictly speaking, all the hyperplanes of $\mathcal{H}_P$ are not needed in the computation of $\Ima\triangle_{(P,\vec v)}$. However, computing the set of hyperplanes $\mathcal{H}_P(F,G)$ for a given pair of faces seems to be more complicated than applying (2), both from the combinatorial and the computational points of view. 

\begin{example}[The cubes] The $n$-dimensional cube $C^n=[0,1]^n \subset \RR^n$ is positively oriented by the vector $\vec v=(1,\ldots,1)$. We choose as normal vectors $\vec d_H$ the family $\{ e_i \ | \ 1\leq i \leq n\}$. Any pair of subsets $K,L\subset\{1,\ldots,n\}$ with $K\cap L=\emptyset$ defines a face $F$ such that $\mathcal{N}_{C^n}(F)=\cone(\{ e_k \ | \ k \in K\} \cup \{ - e_l \ | \ l \in L\})$. \cref{thm:universalformula} says that $(F,G) \in \Ima \triangle_{(C^n,\vec v)}$ if and only if for each $1\leq i \leq n$, either $- e_i \in \mathcal{N}_{C^n}(F)$ or $e_i \in \mathcal{N}_{C^n}(G)$. Restricting our attention to pairs with $\dim F + \dim G = n$, we obtain directly the families $\{\vec F_i=- e_i \ | \ i \in I\}$ and $\{\vec G_j= e_j \ | \ j \in J\}$ for partitions $I\cup J=\{1,\ldots,n\}$, which define J.-P. Serre's approximation of the diagonal.  
\end{example}

\begin{example}[The simplices] 
The $n$-dimensional standard simplex $\Delta^n\subset\RR^{n+1}$ is positively oriented by any vector $\vec v$ with strictly increasing coordinates. 
We write $[n+1]=\{1,\ldots,n+1\}$. 
We choose as normal vectors $\vec d_H$ the family $\{ e_j - e_i \ | \ 1\leq i<j \leq n+1\}$ and we set $\vec n =(1,\ldots,1)$. Any subset $I\subset [n+1]$ defines a face $F$ such that $\mathcal{N}_{\Delta^n}(F)=\cone(\{- e_j \ | \ j\in [n+1]\setminus I\}\cup\{\vec n, -\vec n\})$. \cref{thm:universalformula} says that $(F,G) \in \Ima \triangle_{(\Delta^n,\vec v)}$ if and only if for each pair $1\leq i<j \leq n+1$, either $- e_j\in \mathcal{N}_{\Delta^n}(F)$ or $- e_i \in \mathcal{N}_{\Delta^n}(G)$. Restricting our attention to pairs with $\dim F + \dim G = n$, we obtain directly the families $\{\vec F_i=- e_i \ | \ 1 \leq i \leq k\}$ and $\{\vec G_j=- e_j \ | \ k\leq j \leq n+1 \}$ for $k\in [n+1]$, which define the Alexander--Whitney map.
\end{example}

The case of the associahedra will be treated in the same fashion in \cref{section:diagonal}, as a special case of \cref{thm:formulaoperahedra}.

\subsection{Universal formula and refinement of normal fans} \label{sec:applications} We consider polytopes related by refinement of their normal fans. We have in mind applications to the operahedra in \cref{section:diagonal}, and to generalized permutahedra in forthcoming work. We recall that a fan $\mathcal{G}'$ \emph{refines} a fan $\mathcal{G}$, or that $\mathcal{G}$ \emph{coarsens} $\mathcal{G}'$, if every cone of $\mathcal{G}$ is the union of cones of $\mathcal{G}'$ and $\cup_{C\in\mathcal{G}}C=\cup_{C'\in\mathcal{G}'}C'$, see \cite[Lecture 7]{Ziegler95} for more details and examples.

\begin{definition}[Coarsening projection] \label{def:coarseningprojection} Let $P$ and $Q$ be two polytopes in $\RR^n$ such that the normal fan of $P$ refines the normal fan of $Q$. The \emph{coarsening projection} from $P$ to $Q$ is the application $\theta : \mathcal{L}(P)\to\mathcal{L}(Q)$ which sends a face $F$ of $P$ to the face $\theta(F)$ of $Q$ whose normal cone $\mathcal{N}_Q(\theta(F))$ is the minimal cone with respect to inclusion which contains $\mathcal{N}_P(F)$.
\end{definition}

\begin{proposition} \label{prop:coarseninghyperplanes} Let $P$ and $Q$ be two polytopes in $\RR^n$ such that the normal fan of $P$ refines the normal fan of $Q$. Then, their fundamental hyperplane arrangements satisfy $\mathcal{H}_Q\subset \mathcal{H}_P$.
\end{proposition}
\begin{proof} Let $H\in \mathcal{H}_Q$. Let $F,G$ be two faces of $Q$ such that the intersection $G \cap \rho_z F$ is an edge of $Q\cap\rho_z Q$ perpendicular to $H$, for any $z \in (\mathring{F}+\mathring{G})/2$. If we write $\mathcal{N}_Q(F)=\cone(\{\vec F_i\}_{i\in I})$ and $\mathcal{N}_Q(G)=\cone(\{\vec G_j\}_{j\in J})$, this means that a direction $\vec d$ of this edge is solution to the system of equations $\langle \vec F_i, \vec d \rangle =0$ and $\langle \vec G_j, \vec d \rangle =0$. Now we choose any $F'\in \theta^{-1}(F)$ and $G'\in \theta^{-1}(G)$ such that $\dim(F')=\dim(F)$ and $\dim(G')=\dim(G)$. We can write the normal cones of $F'$ and $G'$ as $\mathcal{N}_P(F')=\cone(\{\vec F_k'\}_{k \in K})$ and $\mathcal{N}_P(G')=\cone(\{\vec G_l'\}_{l \in L})$ where for each $k$ and $l$, we have $\vec F_k' = \sum_{i \in I} \alpha_i \vec F_i$ and $\vec G_l' = \sum_{j \in J} \beta_j \vec G_j$ with $\alpha_i, \beta_j \geq 0$ for all $i,j$. So, the direction $\vec d$ is also solution to the system of equations $\langle \vec F_k', \vec d \rangle =0$ and $\langle \vec G_l', \vec d \rangle =0$. Then, the dimension assumption shows that for any $w \in (\mathring{F}'+\mathring{G}')/2$ the intersection $G'\cap\rho_w F'$ is an edge of $P\cap\rho_w P$ with direction $\vec d$, and thus $H\in\mathcal{H}_P$. 
\end{proof}

\begin{corollary} \label{coroll:coarseningoriented} Suppose that that normal fan of $P$ refines the normal fan of $Q$. If $P$ is positively oriented by $\vec v$, then so is $Q$. 
\end{corollary}
\begin{proof} This is an immediate consequence of \cref{prop:coarseninghyperplanes}.
\end{proof}

\begin{proposition} \label{prop:coarseningcommutes} Let $P$ and $Q$ be two polytopes in $\RR^n$ such that the normal fan of $P$ refines the normal fan of $Q$, and suppose that they are both positively oriented by the same vector $\vec v \in \RR^n$. Then, the coarsening projection $\theta$ commutes with the cellular maps $\triangle_{(P,\vec v)}$ and $\triangle_{(Q,\vec v)}$. 
\end{proposition}
\begin{proof} Let $F$ be a face of $Q$. By definition of the coarsening projection $\theta$, we have that $\mathcal{N}_Q(F)=\cup_{F'\in\theta^{-1}(F)}\mathcal{N}_P(F')$. It follows that \[\bigcup_{\substack{F'\in\theta^{-1}(F)\\ G'\in\theta^{-1}(G)}}\cone(-\mathcal{N}_P(F')\cup\mathcal{N}_P(G'))=\cone(-\mathcal{N}_Q(F)\cup\mathcal{N}_Q(G)) \ ,\] from which we conclude by using \cref{prop:coeurcoeur}.
\end{proof}

\begin{proposition} \label{coroll:coarseninguniversal} Let $P$ and $Q$ be two polytopes such that $\mathcal{H}_Q\subset \mathcal{H}_P$, and suppose that they are both positively oriented by the same vector $\vec v$. For each $H\in\mathcal{H}_P$, we choose a normal vector $\vec d_H$ such that $\langle \vec d_H, \vec v \rangle >0$. We have
\begin{eqnarray*}
(F,G) \in \Ima \triangle_{(Q,\vec v)} 
&\iff& \forall H \in \mathcal{H}_P , \ \exists i , \ \langle \vec F_i, \vec d_H \rangle < 0  \text{ or } \exists j , \ \langle \vec G_j, \vec d_H \rangle > 0 \ .
\end{eqnarray*} 
\end{proposition}
\begin{proof} 
We denote by $C_P$ (resp. $C_Q$) the chamber of $\mathcal{H}_P$ (resp. $\mathcal{H}_Q$) containing $\vec v$. Since $\mathcal{H}_Q\subset \mathcal{H}_P$, we have $C_P\subset C_Q$. As in the proof of \cref{thm:universalformula}, we write $\cone(-F,G)\coloneqq \cone(-\mathcal{N}_Q(F)\cup\mathcal{N}_Q(G))$ and we have $(F,G) \in \Ima\triangle_{(Q,\vec v)} \iff C_Q \cap \cone(-F,G)\neq \emptyset \iff C_Q \subset \cone(-F,G)$. For the first implication ($\implies$), suppose that there exists an $H\in\mathcal{H}_P$ such that $\cone(-F,G)\subset \overline{H^{-}}=:\{x \in \RR^n | \ \langle \vec d_H, x \rangle \leq 0\}$. Since $C_P\subset C_Q \subset \cone(-F,G)$, we have $C_P\subset \overline{H^{-}}$, which is impossible. The reverse implication ($\impliedby$) follows immediately from \cref{thm:universalformula} since $\mathcal{H}_Q\subset\mathcal{H}_P$.
\end{proof}

One can thus compute the universal formula for a polytope $P$ and apply it \emph{mutatis mutandis} to any polytope $Q$ whose normal fans coarsens the one of $P$. Alternatively, one can apply the coarsening projection via \cref{prop:coarseningcommutes}. Depending on the polytopes under consideration, one approach or the other might give a simpler combinatorial description.

\section{Realizations of the operahedra} \label{section:realizations}

In this section we define the operahedra, the family of polytopes that will be the center of attention for the rest of the paper. These polytopes range from the associahedra to the permutohedra. Their face lattices are described by the combinatorics of planar nested trees. Via the line graph construction, these correspond to tubed clawfree block graphs, and the operahedra are thus instances of graph-associahedra \cite{CarrDevadoss06}. We define integer-coordinate realizations of the operahedra by the same procedure as for J.-L. Loday's realizations of the associahedra \cite{Loday04a}. Their fundamental geometric properties are described in \cref{prop:PropertiesKLoday}. The standard weight realizations were already studied by V. Pilaud in \cite{PilaudSignedTree13}, but the construction given here is different.

\subsection{What is an operahedron?} \label{subsec:whatis}

Let us consider the set $\PT{n}$ of reduced planar rooted trees with $n$ internal vertices, for $n\geq 1$, that is trees where each internal vertex is at least bivalent. 
We denote the set of internal vertices of a tree $t\in \PT{n}$ by $V(t)$ and its set of internal edges by $E(t)$, and we label them as pictured in \cref{fig:treeandnesting}: starting from the root, we traverse around the tree in clockwise direction, numbering a vertex (resp. an edge) only the first time we see it.

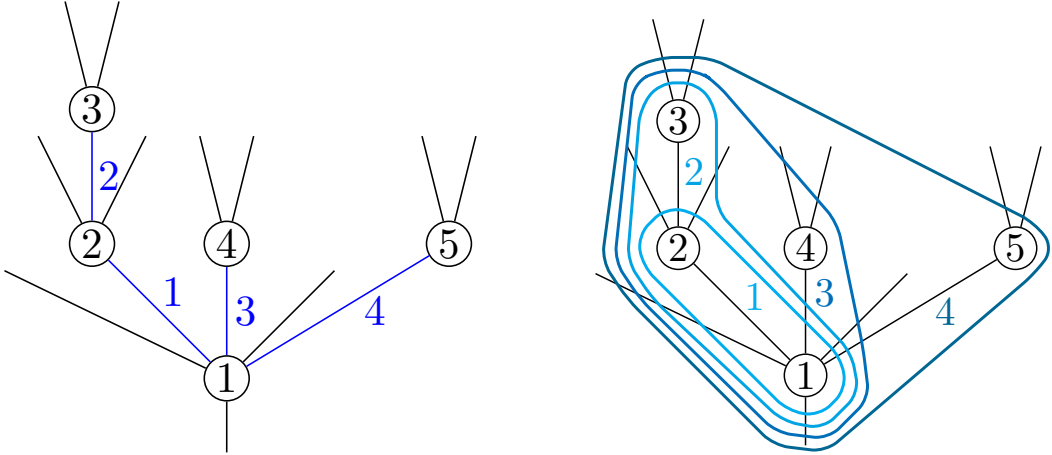
\begin{figure}[h!]
\centering
\resizebox{0.4\linewidth}{!}{
\begin{tikzpicture}[scale=1.2]
       \node (E) [circle,draw=black,minimum size=4mm,inner sep=0.1mm] at (-0,0) {\small $1$};
       \node (F) [circle,draw=black,minimum size=4mm,inner sep=0.1mm] at (-1,1) {\small $2$};
       \node (A) [circle,draw=black,minimum size=4mm,inner sep=0.1mm] at (-1,2) {\small $3$};
       \node (q) [circle,draw=black,minimum size=4mm,inner sep=0.1mm] at (0,1) {\small $4$};
       \node (r) [circle,draw=black,minimum size=4mm,inner sep=0.1mm] at (1.65,1) {\small $5$};
       \node (x) [circle,draw=none,minimum size=4mm,inner sep=0.1mm] at (-0.4,0.64) {\color{blue} \small $1$};
       \node (y) [circle,draw=none,minimum size=4mm,inner sep=0.1mm] at (-0.875,1.5) {\color{blue} \small $2$};
       \node (u) [circle,draw=none,minimum size=4mm,inner sep=0.1mm] at (0.14,0.5) {\color{blue} \small $3$};
       \node (v) [circle,draw=none,minimum size=4mm,inner sep=0.1mm] at (1.1,0.5) {\color{blue} \small $4$};
       \draw[-] (0.8,0.8) -- (E)--(-1.65,0.8); 
       \draw[-] (-1.2,2.8) -- (A)--(-0.8,2.8); 
       \draw[-] (-0.2,1.8) -- (q)--(0.2,1.8);   
       \draw[-] (1.85,1.8) -- (r)--(1.45,1.8); 
       \draw[-] (E)--(0,-0.55); 
       \draw[-] (-1.4,1.8) -- (F)--(-0.6,1.8);   
       \draw[-,draw=blue] (E)--(F) node {};
       \draw[-,draw=blue] (E)--(q) node {};
       \draw[-,draw=blue] (E)--(r) node {};
       \draw[-,draw=blue] (F)--(A) node {};
\end{tikzpicture}}
\quad \quad \quad 
\resizebox{0.4\linewidth}{!}{
\begin{tikzpicture}[scale=1.2]
       \node (E)[circle,draw=black,minimum size=4mm,inner sep=0.1mm] at (-0,0) {\small $1$};
       \node (F) [circle,draw=black,minimum size=4mm,inner sep=0.1mm] at (-1,1) {\small $2$};
       \node (A) [circle,draw=black,minimum size=4mm,inner sep=0.1mm] at (-1,2) {\small $3$};
       \node (q) [circle,draw=black,minimum size=4mm,inner sep=0.1mm] at (0,1) {\small $4$};
       \node (r) [circle,draw=black,minimum size=4mm,inner sep=0.1mm] at (1.65,1) {\small $5$};
       \node (x) [circle,draw=none,minimum size=4mm,inner sep=0.1mm] at (-0.4,0.62) {\color{Cyan} \small $1$};
       \node (y) [circle,draw=none,minimum size=4mm,inner sep=0.1mm] at (-0.875,1.6) {\color{Cerulean} \small $2$};
       \node (u) [circle,draw=none,minimum size=4mm,inner sep=0.1mm] at (0.15,0.65) {\color{NavyBlue} \small $3$};
       \node (v) [circle,draw=none,minimum size=4mm,inner sep=0.1mm] at (1.1,0.5) {\color{MidnightBlue} \small $4$};
       \draw[-] (0.8,0.8) -- (E)--(-1.65,0.8); 
       \draw[-] (-1.2,2.8) -- (A)--(-0.8,2.8); 
       \draw[-] (-0.2,1.8) -- (q)--(0.2,1.8);   
       \draw[-] (1.85,1.8) -- (r)--(1.45,1.8); 
       \draw[-] (E)--(0,-0.55); 
       \draw[-] (-1.4,1.8) -- (F)--(-0.6,1.8);   
       \draw[-] (E)--(F) node {};
       \draw[-] (E)--(q) node  {};
       \draw[-] (E)--(r) node {};
       \draw[-] (F)--(A) node {};
       \draw [Cyan,rounded corners,thick] (0.11,-0.32) -- (-0.14,-0.28) -- (-1.28,0.86) -- (-1.32,1.1) --  (-1.1,1.32) -- (-0.86,1.28) -- (0.28,0.14) -- (0.32,-0.11) -- cycle;
       \draw [Cerulean,rounded corners,thick] (0.14,-0.42) -- (-0.18,-0.36) -- (-1.2,0.6) -- (-1.4,0.9) -- (-1.3,2.1) -- (-1.15,2.3) -- (-0.85,2.3) -- (-0.7,2.1) --  (-0.7,1.3) -- (0.36,0.18) -- (0.42,-0.14) -- cycle;
       \draw [NavyBlue,rounded corners,thick] (0.18,-0.5) -- (-0.23,-0.45) -- (-1.3,0.6) -- (-1.5,0.9) --  (-1.35,2.3) --  (-1.15,2.4) --  (-0.85,2.4) --  (-0.7,2.3) --  (.25,1.2) -- (0.45,0.23) -- (0.5,-0.18) -- cycle;
       \draw [MidnightBlue,rounded corners,thick] (0.16,-0.6) -- (-0.25,-0.55) -- (-1.4,0.6) -- (-1.6,0.9) --  (-1.4,2.4) --  (-1.15,2.5) --  (-0.75,2.5) --  (1.8,1.2) --  (1.95,1) --  (1.8,0.8) -- cycle;
\end{tikzpicture} }
\caption{A tree $t\in\PT{5}$ with our labeling conventions and one of its maximal nestings.}
\label{fig:treeandnesting}
\end{figure} 

The leaves and root (resp. the leaf edges and root edge) are not considered part of the set $V(t)$ (resp. $E(t)$), and from now on, the word "vertex" (resp. "edge") will refer exclusively to internal vertices (resp. edges).
Moreover, we will abuse terminology and use the terms \emph{leaves} and \emph{root}, as well as the terms \emph{inputs} and \emph{output}, to designate the leaf edges and root edge, respectively. 

Any subset of edges $N\subset E(t)$ defines a subgraph of $t$ whose edges are $N$ and whose vertices are all the vertices adjacent to an edge in $N$. We call this graph the \emph{closure} of $N$. 

\begin{definition}[Nest] A \emph{nest} of a tree $t \in \PT{n}$ is a non-empty set of edges $N \subset E(t)$ whose closure is a connected subgraph of $t$.
\end{definition}
Every nest $N$ thus defines a subtree $t(N)$ of $t$ by adjoining to its closure all the edges, leaves or root adjacent to its vertices. We call it the \emph{induced subtree of $N$}. 

\begin{definition}[Nesting] \label{def:nesting} A \emph{nesting} of a tree $t \in \PT{n}$ is a set $\mathcal{N}=\{N_i\}_{i\in I}$ of nests such that 
\begin{enumerate}
    \item the \emph{trivial nest} $E(t)$ is in $\mathcal{N}$,
    \item for every pair of nests $N_i\neq N_j$, we have either $N_i \subsetneq N_j$, $N_j \subsetneq N_i$ or $N_i \cap N_j = \emptyset$, and
    \item if $N_i \cap N_j = \emptyset$ then no edge of $N_i$ is adjacent to an edge of $N_j$.
\end{enumerate}
\end{definition}
Two nests that satisfy Conditions (2) and (3) are said to be \textit{compatible}. We naturally represent a nesting by circling the closure of each nest as in Figure \ref{fig:treeandnesting}. We denote by $\mathcal{N}(t)$ the set of nestings of a tree $t$. We notice that for a \emph{corolla} $t \in \PT{1}$ we have $\mathcal{N}(t) = \emptyset$. We call \emph{nested tree} a pair $(t,\mathcal{N})$ made up of a tree and a nesting.

\begin{definition}[Lattice of nestings] We denote by $(\mathcal{N}(t),\subset)$ the poset of nestings of a tree $t \in \PT{n}$ ordered by inclusion, together with a maximal element. 
\end{definition}

\begin{remark} 
\label{rem:linegraph} 
As explained in detail in \cite[Section 3.4]{Ward19}, a nesting of a tree $t$ is the same as a tubing \cite[Definition 2.2]{CarrDevadoss06} of the line graph \cite[Chapter 8]{Harary69} of the closure of $E(t)$. 
Note that the leaves and the root are not taken into account, so any tree $t'$ with the same internal structure as $t$ has the same line graph. 
\end{remark}

\begin{definition}[Edge contraction] The \emph{contraction of an edge} $e$ connecting vertices $v_1$ and $v_2$ consists in deleting $e$ and collapsing $v_1$ and $v_2$ to a new vertex $v$ having as inputs the union of the inputs of $v_1$ and $v_2$. 
\end{definition}
\begin{figure}[h!]
\centering
    \resizebox{0.7\linewidth}{!}{
    \begin{tikzpicture}
    \node (E)[circle,draw=black,minimum size=4mm,inner sep=0.1mm] at (0,0) {\small $1$};
    \node (F) [circle,draw=black,minimum size=4mm,inner sep=0.1mm] at (-1,1) {\small $2$};
    \node (A) [circle,draw=black,minimum size=4mm,inner sep=0.1mm] at (-1,2) {\small $3$};
    \node (q) [circle,draw=black,minimum size=4mm,inner sep=0.1mm] at (0,1) {\small $4$};
    \node (r) [circle,draw=black,minimum size=4mm,inner sep=0.1mm] at (1.65,1) {\small $5$};
    \draw[-] (0.8,0.8) -- (E)--(-1.65,0.8); 
    \draw[-] (-1.2,2.8) -- (A)--(-0.8,2.8); 
    \draw[-] (-0.2,1.8) -- (q)--(0.2,1.8);   
    \draw[-] (1.85,1.8) -- (r)--(1.45,1.8); 
    \draw[-] (E)--(0,-0.55); 
    \draw[-] (-1.4,1.8) -- (F)--(-0.6,1.8);   
    \draw[-] (E)--(F) node {};
    \draw[-,thick] (E)--(q) node  {};
    \draw[-,thick] (E)--(r) node {};
    \draw[-] (F)--(A) node {};

    \draw [NavyBlue,rounded corners,thick] (-0.15,-0.4) -- (-0.4,-0.25) -- (-0.4,1.2) -- (-0.15,1.4) -- (1.8,1.4) -- (2,1.1) -- (2,0.85) -- (0.4,-0.25) -- (0.15,-0.4) -- cycle;
    \end{tikzpicture} 
    \quad \resizebox{0.04\linewidth}{!}{\raisebox{2em}{$\longrightarrow$}}\quad \quad
    \begin{tikzpicture}
        \node (E)[circle,draw=black,minimum size=4mm,inner sep=0.1mm] at (0,0) {\small $1$};
        \node (F) [circle,draw=black,minimum size=4mm,inner sep=0.1mm] at (-1,1) {\small $2$};
        \node (A) [circle,draw=black,minimum size=4mm,inner sep=0.1mm] at (-1,2) {\small $3$};
        \draw[-] (0.8,0.8) -- (E)--(-1.65,0.8); 
        \draw[-] (-1.2,2.8) -- (A)--(-0.8,2.8); 
        \draw[-] (-0.2,0.8) -- (E)--(0.2,0.8);   
        \draw[-] (1.85,0.8) -- (E)--(1.45,0.8); 
        \draw[-] (E)--(0,-0.55); 
        \draw[-] (-1.4,1.8) -- (F)--(-0.6,1.8);   
        \draw[-] (E)--(F) node {};
        \draw[-] (F)--(A) node {};
\end{tikzpicture} 
}
\caption{Contraction of the edges in a nest.}
\label{fig:contraction}
\end{figure}
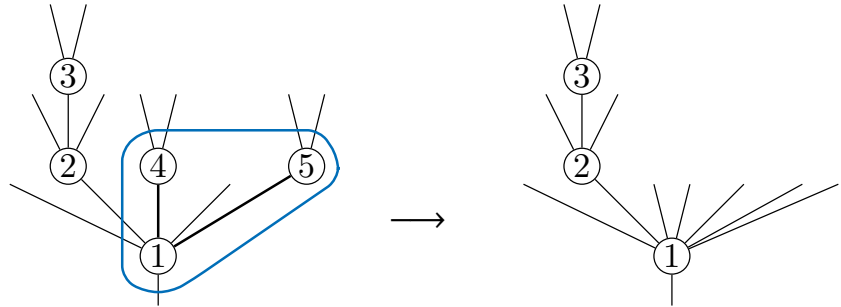 
We observe that trees are stable under edge contraction. Given a nested tree $(t,\mathcal{N})$, the \emph{contraction of a nest} $N \in \mathcal{N}$ consists in the contraction of all the edges of $t(N)$, as pictured in \cref{fig:contraction}. 

\begin{definition}[Maximal nesting] A nesting is \emph{maximal} if it has maximal cardinality. 
\end{definition}

We denote by $\mathcal{MN}(t) \subset \mathcal{N}(t)$ the set of maximal nestings. If $t \in \PT{n}$ and $\mathcal{N} \in \mathcal{MN}(t)$, we have $|\mathcal{N}|=|E(t)|=n-1$. We call \emph{fully nested tree} a nested tree $(t,\mathcal{N})$ where $\mathcal{N}$ is maximal. 

\begin{definition} \label{def:notationgeq}
For any subset of edges $S\subset E(t)$ of a nested tree $(t,\mathcal{N})$, we denote by $\mathcal{N}(S)$ the set of nests of $\mathcal{N}$ containing $S$.
\end{definition}
    
By definition of a nesting, the set $\mathcal{N}(e)$ is totally ordered by inclusion, for any edge $e\in E(t)$. For a maximal nesting $\mathcal{N} \in \mathcal{MN}(t)$, the assignment $e \mapsto \min\mathcal{N}(e)$ defines a bijection between $E(t)$ and $\mathcal{N}$.

\begin{samepage}
\begin{definition}[Poset of maximal nestings] \label{def:order2}
We denote by  $(\mathcal{MN}(t),<)$ the poset generated by the transitive closure of the covering relations
\begin{center}
\resizebox{0.25\linewidth}{!}{
\quad
\begin{tikzpicture}
    \node (t3)[circle,draw=black,minimum size=4mm,inner sep=0.1mm] at (0,0) {\small $t_1$};
    \node (t2) [circle,draw=black,minimum size=4mm,inner sep=0.1mm] at (0,1) {\small $t_2$};
    \node (t1) [circle,draw=black,minimum size=4mm,inner sep=0.1mm] at (0,2) {\small $t_3$};  
    \draw[-] (t3)--(t2) node {};
    \draw[-] (t2)--(t1) node {};
    \draw [Plum,rounded corners,thick] (-0.15+2,0.7) -- (-0.3+2,0.9) -- (-0.3+2,2.1) -- (-0.15+2,2.3) -- (0.15+2,2.3) -- (0.3+2,2.1) -- (0.3+2,0.9) -- (0.15+2,0.7) -- cycle;

    \node (B) at (1,1) {$\prec$};

    \node (b3)[circle,draw=black,minimum size=4mm,inner sep=0.1mm] at (2,0) {\small $t_1$};
    \node (b2) [circle,draw=black,minimum size=4mm,inner sep=0.1mm] at (2,1) {\small $t_2$};
    \node (b1) [circle,draw=black,minimum size=4mm,inner sep=0.1mm] at (2,2) {\small $t_3$};  
    \draw[-] (b3)--(b2) node {};
    \draw[-] (b2)--(b1) node {};
    \draw [Plum,rounded corners,thick] (1.85-2,-0.3) -- (1.7-2,-.1) -- (1.7-2,1.1) -- (1.85-2,1.3) -- (2.15-2,1.3) -- (2.3-2,1.1) -- (2.3-2,-.1) -- (2.15-2,-.3) -- cycle;
\end{tikzpicture}} \quad\quad \raisebox{2.75em}{and}\quad\quad 
\resizebox{0.4\linewidth}{!}{
\raisebox{1em}{\begin{tikzpicture}
    \node (t3)[circle,draw=black,minimum size=4mm,inner sep=0.1mm] at (0,0) {\small $t_1$};
    \node (t2) [circle,draw=black,minimum size=4mm,inner sep=0.1mm] at (0.71,0.71) {\small $t_3$};
    \node (t1) [circle,draw=black,minimum size=4mm,inner sep=0.1mm] at (-0.71,0.71) {\small $t_2$};  
    \draw[-] (t3)--(t1) node {};
    \draw[-] (t2)--(t3) node {};
    \draw [Plum,rounded corners,thick] (0.11,-0.32) -- (-0.14,-0.28) -- (-0.99,0.57) -- (-1.03,0.81) -- (-0.81,1.03) -- (-0.57,0.99)  -- (0.28,0.14) -- (0.32,-0.11) -- cycle;

    \node (B) at (1.5,0.5) {$\prec$};

    \node (b3)[circle,draw=black,minimum size=4mm,inner sep=0.1mm] at (3,0) {\small $t_1$};
    \node (b2) [circle,draw=black,minimum size=4mm,inner sep=0.1mm] at (3+0.71,0.71) {\small $t_3$};
    \node (b1) [circle,draw=black,minimum size=4mm,inner sep=0.1mm] at (3-0.71,0.71) {\small $t_2$};  
    \draw[-] (b3)--(b1) node {};
    \draw[-] (b2)--(b3) node {};
    \draw [Plum,rounded corners,thick] (3-0.32,-0.11) -- (3-0.28,0.14) -- (3+0.57,0.99) -- (3+0.81,1.03) -- (3+1.03,0.81) -- (3+0.99,0.57) -- (3+0.14,-0.28) -- (3-0.11,-0.32) -- cycle;
\end{tikzpicture}}} \raisebox{2.75em}{  ,} 
\end{center} 
where $t_1,t_2$ and $t_3$ are trees.
\end{definition}
\end{samepage}

On the set of \emph{linear trees}, i.e. trees where each vertex is connected to at most two edges, this order relation specializes to the \emph{Tamari order} \cite{Tamari51}. This can be seen via the bijection between the set of maximal nestings of a linear tree with $n$ vertices and the set of planar binary trees with $n$ leaves shown in \cref{fig:bijections}. 

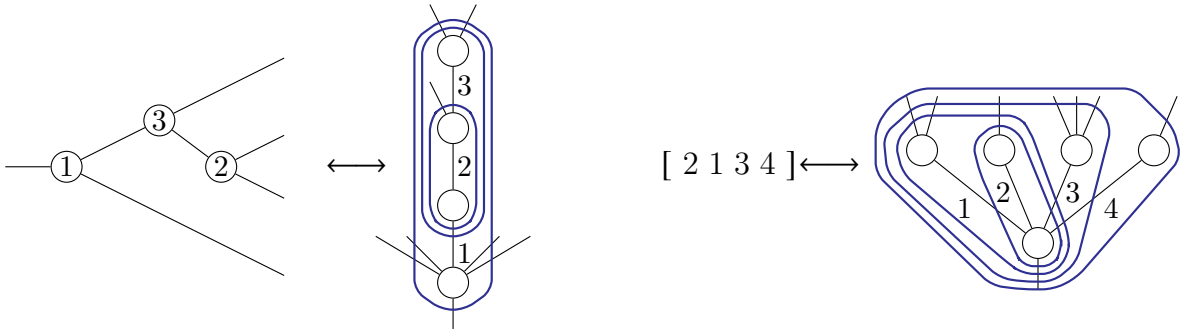
\begin{figure}[h!]
\resizebox{1\linewidth}{!}{
\begin{tikzpicture}
    \node (b1)[circle,draw=none,minimum size=4mm,inner sep=0.1mm] at (-2,3) {};
    \node (b2)[circle,draw=none,minimum size=4mm,inner sep=0.1mm] at (-2,2) {};
    \node (b3) [circle,draw=none,minimum size=4mm,inner sep=0.1mm] at (-2,1) {};
    \node (b4) [circle,draw=none,minimum size=4mm,inner sep=0.1mm] at (-2,0) {};
    \node (b5) [circle,draw=none,minimum size=4mm,inner sep=0.1mm] at (-6,1.5) {};

    \node (c2) [circle,draw=black,minimum size=4mm,inner sep=0.1mm] at (-3,1.5) {\small $2$}; 
    \node (c1) [circle,draw=black,minimum size=4mm,inner sep=0.1mm] at (-3.8,2.1) {\small $3$}; 
    \node (c3) [circle,draw=black,minimum size=4mm,inner sep=0.1mm] at (-5,1.5) {\small $1$}; 
    \draw[-] (b1)--(c1) node {};
    \draw[-] (b2)--(c2) node {};
    \draw[-] (b3)--(c2) node {};
    \draw[-] (b4)--(c3) node {};
    \draw[-] (c2)--(c1) node {};
    \draw[-] (c3)--(c1) node {};
    \draw[-] (c3)--(b5) node {};

    \node (B) at (-1.25,1.5) {$\longleftrightarrow$};

    \node (x1) [circle,draw=none,minimum size=4mm,inner sep=0.1mm] at (0.15,2.55) {\small $3$};
    \node (x2) [circle,draw=none,minimum size=4mm,inner sep=0.1mm] at (0.15,1.5) {\small $2$};
    \node (x3) [circle,draw=none,minimum size=4mm,inner sep=0.1mm] at (0.15,0.38) {\small $1$};

    \node (t4)[circle,draw=black,minimum size=4mm,inner sep=0.1mm] at (-0,0) {};
    \node (t3)[circle,draw=black,minimum size=4mm,inner sep=0.1mm] at (0,1) {};
    \node (t2) [circle,draw=black,minimum size=4mm,inner sep=0.1mm] at (0,2) {};
    \node (t1) [circle,draw=black,minimum size=4mm,inner sep=0.1mm] at (0,3) {};

    \draw[-] (t4)--(t3) node {};
    \draw[-] (t3)--(t2) node {};
    \draw[-] (t2)--(t1) node {};

    \draw[-] (-0.3,3.6) -- (t1)--(0.3,3.6); 
    \draw[-] (-0.3,2.6) -- (t2); 
    \draw[-] (-1,0.6) -- (t4)--(-0.6,0.6); 
    \draw[-] (0.6,0.6) -- (t4)--(1,0.6); 

    \draw[-] (0,-0.6) -- (t4);

    \draw [Blue,rounded corners,thick] (-0.15,0.7) -- (-0.3,0.9) -- (-0.3,2.1) -- (-0.15,2.3) -- (.15,2.3) -- (.3,2.1) -- (.3,.9) -- (.15,.7) -- cycle;
    \draw [Blue,rounded corners,thick] (-0.15,0.6) -- (-0.4,0.8) -- (-0.4,3.1) -- (-0.15,3.3) -- (.15,3.3) -- (0.4,3.1) -- (.4,.8) -- (.15,.6) -- cycle;
    \draw [Blue,rounded corners,thick] (-0.15,-0.35) -- (-0.5,-0.1) -- (-0.5,3.15) -- (-0.15,3.4) -- (.15,3.4) -- (0.5,3.15) -- (0.5,-0.1) -- (0.15,-0.35) -- cycle;
    
\end{tikzpicture} \quad \quad \quad \raisebox{1.2em}{
\begin{tikzpicture}
    \node (A) at (-4,1) {$[ \ 2 \ 1 \ 3  \ 4 \ ]$};

    \node (B) at (-2.7,1) {$\longleftrightarrow$};

    \node (x1) [circle,draw=none,minimum size=4mm,inner sep=0.1mm] at (-0.95,0.45) {\small $1$};
    \node (x2) [circle,draw=none,minimum size=4mm,inner sep=0.1mm] at (-0.45,0.65) {\small $2$};
    \node (x3) [circle,draw=none,minimum size=4mm,inner sep=0.1mm] at (0.45,0.65) {\small $3$};
    \node (x4) [circle,draw=none,minimum size=4mm,inner sep=0.1mm] at (0.95,0.45) {\small $4$};

    \node (t5)[circle,draw=black,minimum size=4mm,inner sep=0.1mm] at (-0,0) {};

    \node (t1)[circle,draw=black,minimum size=4mm,inner sep=0.1mm] at (-1.5,1.2) {};
    \node (t2) [circle,draw=black,minimum size=4mm,inner sep=0.1mm] at (-0.5,1.2) {};
    \node (t3) [circle,draw=black,minimum size=4mm,inner sep=0.1mm] at (0.5,1.2) {};  
    \node (t4) [circle,draw=black,minimum size=4mm,inner sep=0.1mm] at (1.5,1.2) {};  

    \draw[-] (t1)--(t5) node {};
    \draw[-] (t2)--(t5) node {};
    \draw[-] (t3)--(t5) node {};
    \draw[-] (t4)--(t5) node {};

    \draw[-] (-1.7,1.9) -- (t1)--(-1.3,1.9); 
    \draw[-] (-0.5,1.9) -- (t2); 
    \draw[-] (0.2,1.9) -- (t3)--(0.5,1.9); 
    \draw[-] (0.8,1.9) -- (t3); 
    \draw[-] (1.8,1.9) -- (t4);
    \draw[-] (0,-0.6) -- (t5);

    \draw [Blue,rounded corners,thick] (-0.2,-0.3) -- (-0.3,-0.1) -- (-0.75,0.9) -- (-0.85,1.3) -- (-0.55,1.55) -- (-0.25,1.45) -- (0.3,0.1) -- (0.3,-0.1) -- (0.2,-0.3) -- cycle;
    
    \draw [Blue,rounded corners,thick] (-0.2,-0.4) -- (-0.4,-0.25) -- (-1.75,0.9) -- (-1.85,1.3) -- (-1.55,1.65) -- (-0.25,1.65) -- (-0.1,1.45) -- (0.4,0.1) -- (0.4,-0.15) -- (0.2,-0.4) -- cycle;

    \draw [Blue,rounded corners,thick] (-0.2,-0.5) -- (-0.5,-0.45) -- (-1.95,0.9) -- (-2,1.4) -- (-1.55,1.8) -- (-0.25,1.8) -- (0.75,1.8) -- (0.95,1.5) -- (0.5,-0.25) -- (0.2,-0.5) -- cycle;

    \draw [Blue,rounded corners,thick] (-0.2,-0.6) -- (-0.6,-0.55) -- (-2.1,0.9) -- (-2.1,1.5) -- (-1.9,1.8) -- (-1.55,2) -- (-0.25,2) -- (1.25,2) -- (1.75,1.5) -- (1.85,1.1) -- (0.5,-0.45) -- (0.2,-0.6) -- cycle;

\end{tikzpicture}}}
\caption{Bijections between maximal nestings of a linear tree and planar binary trees and between maximal nestings of a 2-leveled tree and permutations.}
\label{fig:bijections}
\end{figure}

On the set of \emph{2-leveled trees}, i.e. trees where all the edges are adjacent to the same vertex, this order specializes to the weak Bruhat order. This can be seen via the bijection between the set of maximal nestings of a 2-leveled tree with $n+1$ vertices and the elements of the symmetric group of order $n$ shown in \cref{fig:bijections}. For a 2-leveled tree $t \in \mathrm{PT}_{n+1}$ and a maximal nesting $\mathcal{N} \in \mathcal{MN}(t)$, we construct a permutation $\sigma \in \mathbb{S}_n$ in the following way. First, for each $i \in E(t)$ we write $N_i\coloneqq \min\mathcal{N}(i) \in \mathcal{N}$. Then, the image of the permutation $\sigma$ is the unique ordered sequence $(\sigma(1),\ldots,\sigma(n))$ such that $|N_{\sigma(j)}|<|N_{\sigma(j+1)}|$ for all $j \in \{1,\ldots,n-1\}$. A covering relation in \cref{def:order2} between two maximal nestings $\mathcal{N}$ and $\mathcal{N}'$ then corresponds precisely to a covering relation of the weak Bruhat order between the associated permutations $\sigma$ and $\sigma'$. 

\begin{definition}[Operahedron]
An \emph{operahedron} is a polytope whose face lattice is isomorphic to the dual $(\mathcal{N}(t),\subset^{\emph{op}})$ of the lattice of nestings of a planar tree $t \in \PT{n}$, for any $n\geq 1$.
\end{definition}

The operahedron corresponding to a tree $t\in\PT{n}$ is of dimension $n-2$ (by convention, the empty set has dimension -1). The face corresponding to a nested tree $(t,\mathcal{N})$ has codimension $|\mathcal{N}|-1$, the number of non-trivial nests of $\mathcal{N}$. The oriented 1-skeleton of an operahedron gives the Hasse diagram of the poset of maximal nestings $(\mathcal{MN}(t),<)$. 

\begin{remark} Following \cref{rem:linegraph}, one can see that the operahedra are a special class of graph-associahedra, as defined in \cite{CarrDevadoss06}: they are associated to line graphs of trees, that is, clawfree block graphs \cite[Theorem 8.5]{Harary69}. Hence they are also a special class of hypergraph polytopes as defined in \cite{DP11}.
\end{remark}

\begin{remark} \label{rem:lattices} It would be interesting to know whether or not the posets $(\mathcal{MN}(t),<)$ are lattices. As studied in \cite{BarnardMcConville18}, the poset of maximal tubings of a graph do not form a lattice in general. For linear and 2-leveled trees, we have lattices isomorphic to the Tamari and weak Bruhat order lattices, respectively. For the other operahedra, comparison with the calculations of \cite[Section 6.1]{BarnardMcConville18} shows that we indeed have lattices up to dimension 3.
\end{remark}

\subsection{Loday realizations of the operahedra} \label{realizations} 

\begin{definition}[Weighted fully nested tree]
A \emph{weighted fully nested tree} is a triple $(t,\mathcal{N},\omega)$ made up of a fully nested tree with $n$ vertices together with a weight $\omega= (\omega_1, \ldots, \omega_n) \in \ZP^n$. We say that the weight $\omega$ has \emph{length} $n$.
\end{definition}

Let us fix a weighted fully nested tree $(t,\mathcal{N},\omega)$. For any edge $i\in E(t)$, we consider the two subtrees $t_1$ and $t_2$ of $t(\mathrm{min}\mathcal{N}(i))$ such that $i$ is the root of $t_1$ and a leaf of $t_2$. In other words, $t_1$ and $t_2$ are obtained by cuting the tree $t(\min\mathcal{N}(i))$ at the edge $i$. We define the two sums
\[\alpha_i:=\sum_{j \in V(t_1)}\omega_j \quad \text{and} \quad \beta_i:=\sum_{j \in V(t_2)}\omega_j \ .\] Multiplying these two numbers together for each edge of $t$, we obtain the following point
\[M(t, \mathcal{N},\omega) \coloneqq \big(\alpha_1\beta_1, \alpha_2\beta_2, \ldots, \alpha_{n-1}\beta_{n-1}\big)\in \ZZ^{n-1}\ . \]
\begin{remark} We will use the notations $\alpha_i$ and $\beta_i$ for brevity, even though these numbers depend on the tree $t$, the nesting $\mathcal{N}$ and the weight $\omega$. This dependence will be implicit but should be clear from the context.
\end{remark}
\begin{definition}[Loday realization of the operahedra] \label{def:Lodayreal}
For any $n\geq 1$, and for any tree $t \in \PT{n}$, the \emph{Loday realization of weight $\omega$ of the operahedron} is the  polytope
\[ P_{(t,\omega)} \coloneqq \conv \big\{M(t, \mathcal{N},\omega)\mid \mathcal{N} \in \mathcal{MN}(t) \big\}\subset \RR^{n-1}\ .\]
\end{definition}
The Loday realization of the operahedron associated to the standard weight $(1, \ldots, 1)$ is simply denoted by $P_t$. Some three-dimensional examples are shown in \cref{figure4}. For any corolla $t \in \PT{1}$, we adopt the following convention: the polytope $P_{(t,\omega)}$, with weight $\omega=(\omega_1)$ of length $1$, is made up of one point in $0$-dimensional space. 

\begin{figure}[h!]
    \centering
    \includegraphics[width=0.32\linewidth]{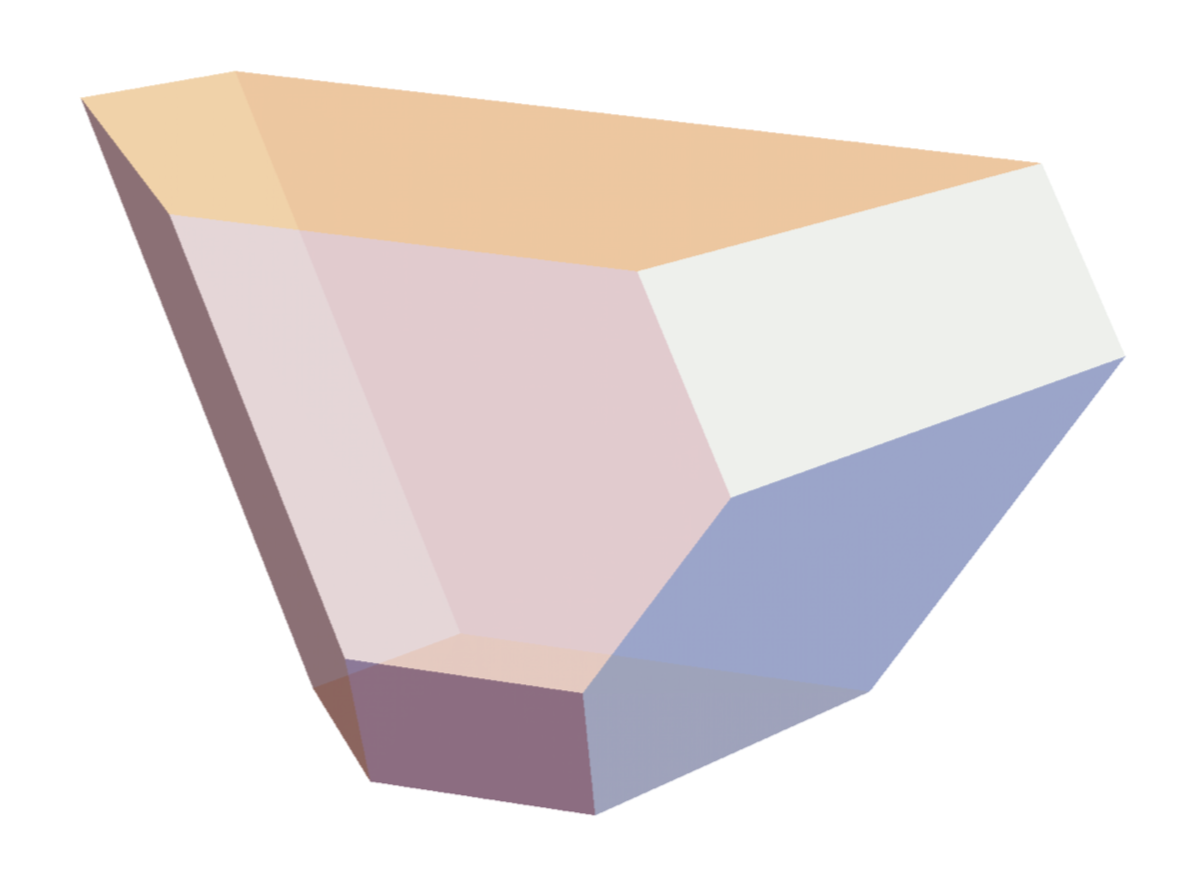} \ 
    \includegraphics[width=0.3\linewidth]{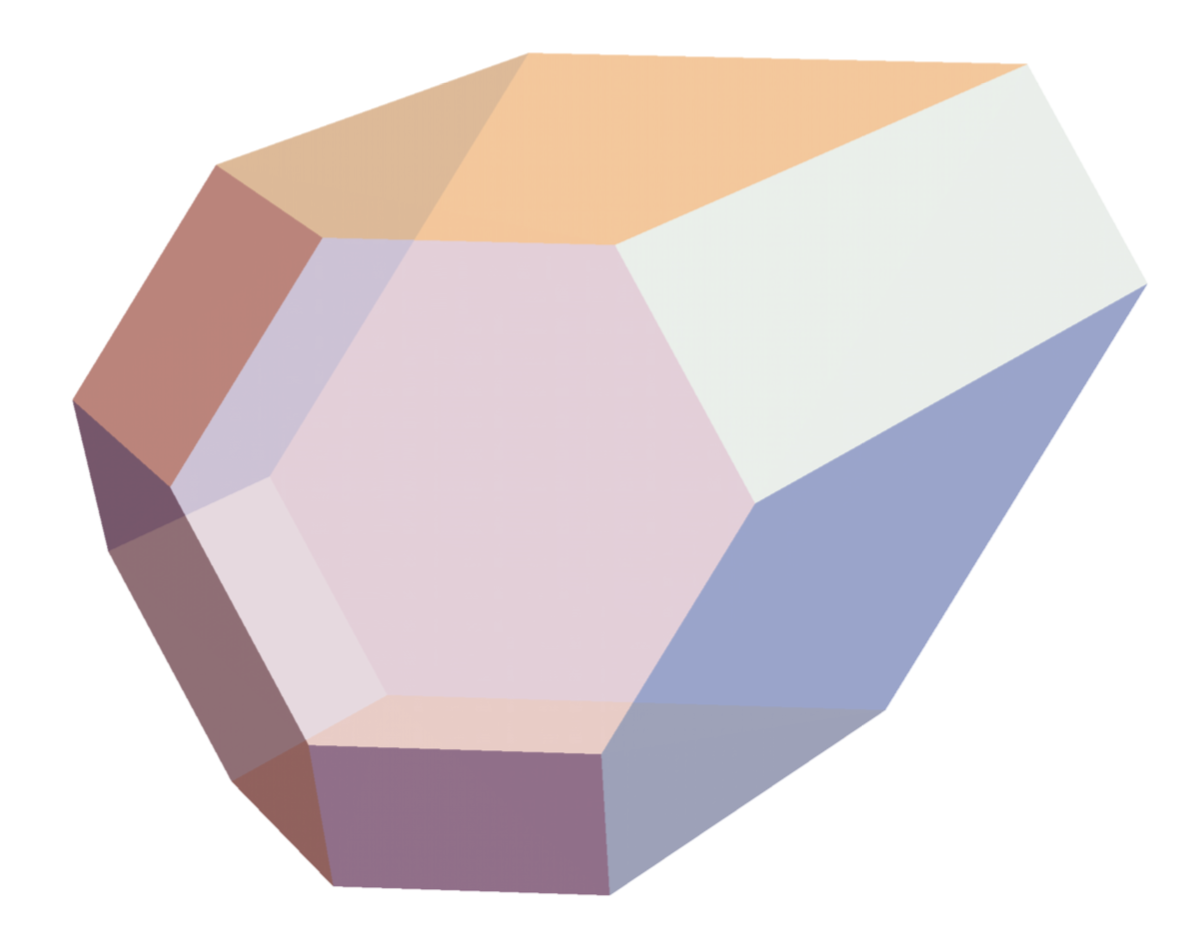} \ \ 
    \includegraphics[width=0.25\linewidth]{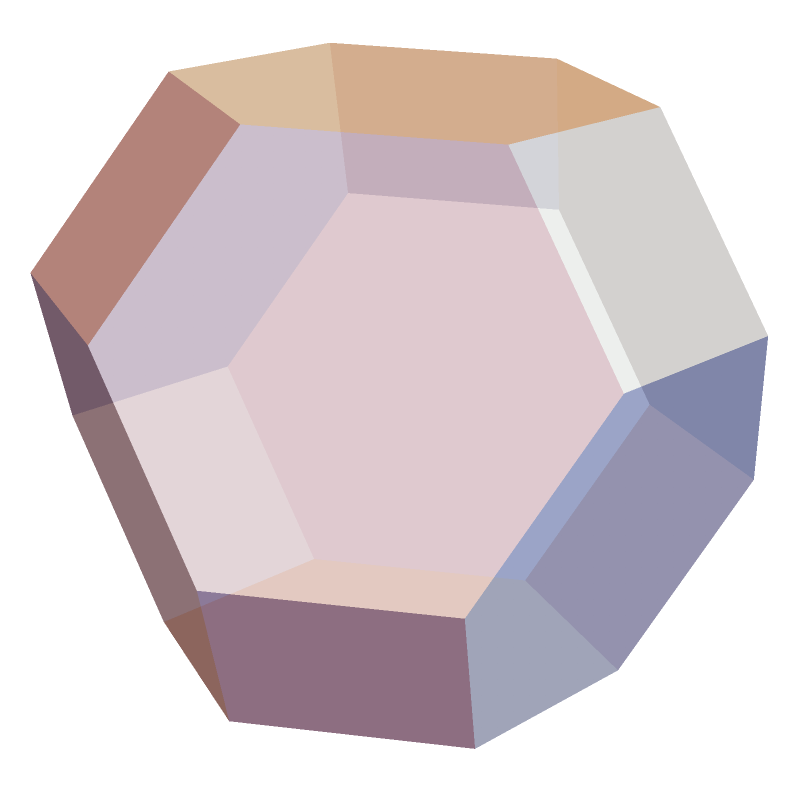}
\caption{Standard weight Loday realizations of some 3-dimensional operahedra, from the associahedron (left) to the permutahedron (right).}
\label{figure4}
\end{figure}

The following proposition summarizes the fundamental properties of Loday realizations of the operahedra and show in particular that they are indeed realizations of the operahedra. In the case of standard weight realizations, it should be compared with \cite[Theorem 56]{PilaudSignedTree13}.

\begin{proposition} \label{prop:PropertiesKLoday}
For any tree $t \in \PT{n}$ and for any weight $\omega$ of length $n$, the Loday realization of the operahedron $P_{(t,\omega)}$ satisfies the following properties. 
\begin{enumerate}
\item It is contained in the hyperplane with equation
\[
\sum_{i \in E(t)} x_i = \sum_{ \substack{k,\ell \in V(t) \\ k<\ell }} \omega_k \omega_l\ .
\]

\item Let $N$ be a non-trivial nest of $t$. For any maximal nesting $\mathcal{N}$, the point $M(t,\mathcal{N},\omega)$ is contained in the half-space defined by the inequality
\[
    \sum_{i \in E(t(N))} x_i \geq \sum_{ \substack{k,\ell \in V(t(N)) \\ k<\ell }} \omega_k \omega_l\ , 
\]
with equality if and only if $N \in \mathcal{N}$. 

\item The polytope $P_{(t,\omega)}$ is the intersection of the hyperplane of \emph{(1)} and the half-spaces of \emph{(2)}. 

\item The face lattice $(\La(P_{(t,\omega)}) \ , \subset)$ is isomorphic to the dual of the lattice of nestings $(\mathcal{N}(t) \ , \subset^{\emph{op}})$. 

\item Any face of a Loday realization of an operahedron is isomorphic to a product of Loday realizations of operahedra of lower dimension, via a permutation of coordinates. 
\end{enumerate}
\end{proposition}

\begin{proof} 

\leavevmode
\begin{enumerate}

\item  We show that every nest $N$ of a maximal nesting $\mathcal{N}$ satisfies the equation 
\[
\sum_{i \in E(t(N))} \alpha_i \beta_i = \sum_{ \substack{k,\ell \in V(t(N)) \\ k<\ell }} \omega_k \omega_l\ ,
\] 
by induction on the cardinality of $N$. The case when $|N|=1$ is clear. We suppose that every nest $N\in\mathcal{N}$ with $1\leq |N| \leq m-1$ satisfies the equation above. We consider now a nest $N$ with $|N|=m\geq 2$. We select $j \in N$ the unique edge such that $N=\min\mathcal{N}(j)$. Denoting by $t_1$ and $t_2$ the two subtrees of $t(N)$ having $j$ respectively as a root and a leaf, we have
\begin{eqnarray*}
    \sum_{i \in E(t(N))} \alpha_i \beta_i 
    &=& \alpha_j \beta_j + \sum_{i \in E(t_1)} \alpha_i \beta_i + \sum_{i \in E(t_2)} \alpha_i \beta_i \\
    &=& \left( \sum_{k \in V(t_1)}\omega_k \right) \left( \sum_{\ell \in V(t_2)}\omega_\ell \right) + \sum_{ \substack{k,\ell \in V(t_1) \\ k<\ell }} \omega_k \omega_l+ \sum_{ \substack{k,\ell \in V(t_2) \\ k<\ell }} \omega_k \omega_l \\
    &=&\sum_{ \substack{k,\ell \in V(t(N)) \\ k<\ell }} \omega_k \omega_l\ .
\end{eqnarray*}
Taking the trivial nest $N=E(t)$, which is contained in every maximal nesting, we obtain that every point $M(t,\mathcal{N},\omega)$ is contained in the hyperplane of (1). By convexity, the same is true for the entire polytope. 

\medskip

\item The proof of Point (1) shows that if the nest $N$ is in $\mathcal{N}$, then 
\[
\sum_{i \in E(t(N))} \alpha_i \beta_i = \sum_{ \substack{k,\ell \in V(t(N)) \\ k<\ell }} \omega_k \omega_l\ .
\]
Let us show that every nest $N \notin \mathcal{N}$ satisfies the strict inequality
\[
\sum_{i \in E(t(N))} \alpha_i \beta_i > \sum_{ \substack{k,\ell \in V(t(N)) \\ k<\ell }} \omega_k \omega_l\ ,
\]
by induction on the cardinality of $N$. The case when $|N|=1$ is clear. We suppose that every nest $N \notin \mathcal{N}$ with $1 \leq |N| \leq m-1$ satisfies the strict inequality above. We consider now a nest $N\notin \mathcal{N}$ with $|N|=m \geq 2$. We select $j$ the unique edge such that $\min\mathcal{N}(j)=\min\mathcal{N}(N)$. It is clear that this edge exists and is unique. We denote by $t_1$ and $t_2$ the two subtrees of $t(\min\mathcal{N}(j))$ having $j$ respectively as a root and a leaf.

If we suppose that $j \notin N$, then $N \subset E(t_1)$ or $N \subset E(t_2)$ which contradicts the assumption that $\min\mathcal{N}(j)=\min\mathcal{N}(N)$. So we have $j\in N$. We denote by $t_1'$ and $t_2'$ the two subtrees of $t(N)$ having $j$ respectively as a root and a leaf. At least one of the inclusions $E(t_1')\subset E(t_1)$ or $E(t_2') \subset E(t_2)$ has to be strict, otherwise we would have $N=\min\mathcal{N}(N)\in \mathcal{N}$. Thus we have 
\begin{eqnarray*}
    \sum_{i \in E(t(N))} \alpha_i \beta_i 
    &=& \alpha_j \beta_j + \sum_{i \in E(t_1')} \alpha_i \beta_i  + \sum_{i \in E(t_2')} \alpha_i \beta_i\\
    &\geq& \left( \sum_{k \in V(t_1)}\omega_k \right) \left( \sum_{\ell \in V(t_2)}\omega_\ell \right) 
    + \sum_{ \substack{k,\ell \in V(t_1') \\ k<\ell }} \omega_k \omega_l
    + \sum_{ \substack{k,\ell \in V(t_2') \\ k<\ell }} \omega_k \omega_l   \\
    &>&\left( \sum_{k \in V(t_1')}\omega_k \right) \left( \sum_{\ell \in V(t_2')}\omega_\ell \right) 
    + \sum_{ \substack{k,\ell \in V(t_1') \\ k<\ell }} \omega_k \omega_l
    + \sum_{ \substack{k,\ell \in V(t_2') \\ k<\ell }} \omega_k \omega_l   \\
    &=&  \sum_{ \substack{k,\ell \in V(t(N)) \\ k<\ell }} \omega_k \omega_l  \ .
\end{eqnarray*}

\medskip 

\item Let us denote by $P$ the polytope defined by the intersection of the hyperplane of (1) and the half-spaces of (2). We show that $P_{(t,\omega)}=P$.
The first inclusion ($\subset$) is obvious. For the reverse inclusion, we observe first that the equations of Point (2), with equality, define the facets of $P$. Let $x=(x_1,\ldots,x_{n-1})$ be a point in the intersection of two facets $F_1$ and $F_2$ of $P$. We claim that the associated nests $N_1$ and $N_2$ are compatible.
We suppose to the contrary that the nests $N_1$ and $N_2$ are not compatible. We are in one of the following two situations.
First, we suppose that $N_1 \cap N_2=\emptyset$. We have by the proof of Point (1) that
\begin{equation*}
\begin{aligned}[l]
\sum_{i \in E(t(N_1\cup N_2))} x_i 
&= \sum_{i \in E(t(N_1))} x_i + \sum_{i \in E(t(N_2))} x_i  \\
    &= \sum_{ \substack{k,\ell \in V(t(N_1)) \\ k<\ell }} \omega_k \omega_l + \sum_{ \substack{k,\ell \in V(t(N_2)) \\ k<\ell }} \omega_k \omega_l \\
    &< \sum_{ \substack{k,\ell \in V(t(N_1)) \\ k<\ell }} \omega_k \omega_l + \sum_{ \substack{k,\ell \in V(t(N_2)) \\ k<\ell }} \omega_k \omega_l + \sum_{ \substack{k \in V(t(N_1))\backslash V(t(N_2)) \\ \ell \in V(t(N_2))\backslash V(t(N_1))}} \omega_k \omega_l  \\
&= \sum_{ \substack{k,\ell \in V(t(N_1\cup N_2)) \\ k<\ell }} \omega_k \omega_l \ ,
\end{aligned}
\end{equation*}
which contradicts the inequality of Point (2) associated to the nest $N_1\cup N_2$.
Second, we suppose that $N_1 \cap N_2\neq \emptyset$. We have
\begin{equation*}
\begin{aligned}[l]
\sum_{i \in E(t(N_1\cap N_2))} x_i 
&= \sum_{i \in E(t(N_1))} x_i + \sum_{i \in E(t(N_2))} x_i  - \sum_{i \in E(t(N_1\cup N_2)} x_i \\
&\leq \sum_{ \substack{k,\ell \in V(t(N_1)) \\ k<\ell }} \omega_k \omega_l 
+ \sum_{ \substack{k,\ell \in V(t(N_2)) \\ k<\ell }} \omega_k \omega_l 
- \sum_{ \substack{k,\ell \in V(t(N_1\cup N_2)) \\ k<\ell }} \omega_k \omega_l\\
&= \sum_{ \substack{k,\ell \in V(t(N_1\cap N_2)) \\ k<\ell }} \omega_k \omega_l - \sum_{ \substack{k\in V(t(N_1))\backslash V(t(N_2)) \\ \ell \in V(t(N_2))\backslash V(t(N_1)) }} \omega_k \omega_l  
        \quad \\
&< \sum_{ \substack{k,\ell \in V(t(N_1\cap N_2)) \\ k<\ell }} \omega_k \omega_l \ , 
\end{aligned}
\end{equation*}
which contradicts the inequality of Point (2) associated to the nest $N_1\cap N_2$. So, $N_1$ and $N_2$ must be compatible. 

A vertex $M$ of $P$ is solution to a system of $n-1$ independent linear equations, one of type (1) and $n-2$ of type (2). By the preceding argument, the associated nests are compatible and assemble into a maximal nesting $\mathcal{N}$ of $t$. Also the point $M(t,\mathcal{N},\omega)$ is solution to this system of equations, in virtue of Point (1) and Point(2). Since the solution is unique, this implies $M=M(t,\mathcal{N},\omega)$ and therefore $P=P_{(t,\omega)}$.

\medskip

\item  Point (2) shows that the facets of $P_{(t,\omega)}$ correspond bijectively to nestings with only one non-trivial nest: the facet labeled by the non-trivial nest $N$ is the convex hull of the points $M(t,\mathcal{N},\omega)$ such that $N\in\mathcal{N}$. Any face of $P_{(t,\omega)}$ of codimension $k$, with $0\leq k \leq n-2$ is defined as the intersection of $k$ facets. The preceding description of facets gives that the set of faces of codimension $k$ is bijectively labeled by nestings with $k$ non-trivial nests: the face corresponding to such a nesting $\mathcal{N}$ is the convex hull of the points $M(t,\mathcal{N}',\omega)$ such that $\mathcal{N}'\subset \mathcal{N}$. With the top dimensional face labeled by the trivial nest, the statement is proved.

\medskip 

\item The proof of the preceding point shows that it is enough to treat the case of the facets. Let $\mathcal{N}$ be a nesting of $t$ with only one non-trivial nest $N$. We contract the nest $N$ to obtain a new tree $\overline{t}$. We define a weight $\overline{\omega}$ on $\overline{t}$ as follows. As a result of the contraction of $N$, the set $V(t(N))$ is reduced in $\overline{t}$ to a single vertex $j$. We assign to this vertex the sum of the weights of the vertices of $N$, that is, \[ \overline{\omega_j}:=\sum_{k \in V(t(N))}\omega_k \ .\] 
Each of the other vertices keeps its weight, only the label changes: for each $i\neq j$ in $V(\overline{t})$, we define $\overline{\omega_i}:=\omega_\ell$ for the corresponding vertex $\ell$ in $V(t)$.
We also define a weight $\widetilde{\omega}$ on $t(N)$, considered as an independent tree that we denote by $\widetilde{t}$. This weight is simply the restriction of $\omega$ to the vertices of $t(N)$ : for each $i \in V(\widetilde{t})$, we define $\widetilde{\omega}_i\coloneqq \omega_\ell$ for the corresponding vertex $\ell$ in $V(t)$. 

We write $|E(\overline{t})|=p$ and $|E(\widetilde{t})|=q$ and we renumber the edges of $\widetilde{t}$ from $p+1$ to $p+q$. We denote by $\sigma: E(\overline{t})\sqcup E(\widetilde{t})\to E(t)$ the permutation mapping each (just renumbered) edge of $\overline{t}$ and $\widetilde{t}$ to its label in $t$. We obtain a $(p,q)$-shuffle. We claim that the image of $P_{(\overline{t},\overline{\omega})}\times P_{(\widetilde{t},\widetilde{\omega})}\hookrightarrow P_{(t,\omega)}$ under the isomorphism 
\begin{align*}
\begin{array}{rccc}
\Theta\  : &  \RR^{p}\times \RR^{q} &\xrightarrow{\cong} &\RR^{n-1}\\
&(x_1, \ldots, x_{p})\times (x_{p+1}, \ldots, x_{p+q}) & \mapsto& 
(x_{\sigma^{-1}(1)}, \ldots, x_{\sigma^{-1}(n-1)})
\end{array}
\end{align*}
is equal to the facet defined by the weighted nested tree $(t,\mathcal{N},\omega)$. To see this, we recall that the two polytopes $P_{(\overline{t},\overline{\omega})}$ and $P_{(\widetilde{t},\widetilde{\omega})}$ are defined by the equations
\begin{eqnarray*}
    \sum_{i \in E(\overline{t})} x_i \overset{(a)}{=} \sum_{ \substack{k,\ell \in V(\overline{t}) \\ k<\ell }} \overline{\omega}_k \overline{\omega}_l
\quad \text{and} \quad 
    \sum_{i \in E(\widetilde{t})} x_i \overset{(b)}{=} \sum_{ \substack{k,\ell \in V(\widetilde{t}) \\ k<\ell }} \widetilde{\omega}_k \widetilde{\omega}_l\ ,
\end{eqnarray*}
respectively, and observe that the image under $\Theta$ of the pair of equations $(a)+(b)$ and $(b)$ consists exactly in the equations defining the facet labelled by $(t, \mathcal{N},\omega)$.
\end{enumerate}
\end{proof}

Restricting to linear trees, we recover the weighted Loday realizations of the associahedra of \cite[Proposition 1]{MTTV19}. Restricting to 2-leveled trees, we obtain weighted realizations of the permutahedron. To end this section, let us point out some geometric properties of the Loday realizations of the operahedra. They can be visualized on the examples of \cref{figure4}.

\begin{corollary} \label{coroll:geometricproperties} For any tree $t \in \PT{n}$ and for any weight $\omega$ of length $n$, the Loday realizations of the operahedron $P_t$ and $P_{(t,\omega)}$ satisfy the following geometric properties. 
    \begin{enumerate}
        \item The polytope $P_{(t,\omega)}$ is obtained by successive truncations of a simplex.
        \item The polytope $P_{(t,\omega)}$ is obtained from the classical permutahedron by parallel translation of its facets, i.e. it is a \emph{generalized permutahedron} in the sense of \cite{Postnikov09}.
        \item The polytope $P_t$ is obtained by deleting inequalities from the facet description of the classical permutahedron, i.e. it is a \emph{removahedron} in the sense of \cite{Pilaud14}. 
    \end{enumerate}
\end{corollary}
\begin{proof} 
    One can read off the normal fan of the operahedron $P=P_{(t,\omega)}$ in Points (1) and (2) of \cref{prop:PropertiesKLoday} as follows. 
    A face $F$ of codimension $k$ of $P$ is determined by a nesting $\mathcal{N}=\{N_j\}_{1\leq j \leq k+1}$ of $t$, where $N_{k+1}$ is the trivial nest. 
    For any nest $N_j \in\mathcal{N}$, we define its associated \emph{characteristic vector} $\vec N_j$ which has a 1 in position $i$ if $i \in N_j$ and 0 otherwise. 
    The normal cone of $F$ is then given by $\mathcal{N}_P(F)=\cone(-\vec N_1,\ldots,-\vec N_{k+1},\vec N_{k+1})$.
    If $t$ is a 2-leveled tree, all the subsets of edges define nests, and we have the normal fan of the permutahedron. 
    If $t$ is a tree which is not a 2-leveled tree, only some subsets of edges define nests. 
    Thus, we have Points (2) and (3) above. 
    For Point (1), we observe that the nests containing only one edge of $t$ define the normal fan of the standard simplex. 
\end{proof}

\section{The diagonal of the operahedra} \label{section:diagonal}

The main goal of this section is to compute the fundamental hyperplane arrangement of the permutahedron which, as we shall see, turns out to be a refinement of the braid arrangement. 
By the general theory of \cref{section:cellularappoximation}, any choice of a chamber in this arrangement then gives a cellular approximation of the diagonal of the permutahedron. 
Moreover, such a choice gives a cellular approximation of the diagonal for every operahedron (in fact, any generalized permutahedra), as well as an explicit combinatorial formula describing its cellular image. 
In contrast with the cases of the simplices, the cubes, and the associahedra, the combinatorics of the permutahedron are less constrained: many choices of chambers agree with the weak Bruhat order on the vertices, and the condition $\tp(F)\leq\bm(G)$ is no longer sufficient to characterize the image of the diagonal. 
We make a choice, motivated by the operadic structure that will appear in the next section. The formula thus obtained, which consists of complementary pairs of ordered partitions of $\{1,\ldots,n\}$, has interesting combinatorial properties.

\subsection{The fundamental hyperplane arrangement of the permutahedra}

Let us first recall from the proof of \cref{coroll:geometricproperties} above that a face $F$ of codimension $k$ of the operahedron $P=P_t$ is determined by a nesting $\mathcal{N}=\{N_j\}_{1\leq j \leq k+1}$ of $t$ where $N_{k+1}$ is the trivial nest. 
For any nest $N_j \in\mathcal{N}$, we define its associated \emph{characteristic vector} $\vec N_j$ which has a 1 in position $i$ if $i \in N_j$ and 0 otherwise.
The vectors $-\vec N_j$, $1 \leq j \leq k$ are outward pointing normal vectors for the facets defining $\mathcal{N}$, in the sense of \cref{def:normalpointing}.
Together with the vector $\vec N_{k+1}$, which forms a basis of the orthogonal complement of the affine hull of $P$, they define the normal cone 
\[ \mathcal{N}_P(F)=\cone\left(-\vec N_1,\ldots,-\vec N_k, -\vec N_{k+1}, \vec N_{k+1}\right) \ . \]

\begin{definition}[Trinary and boolean vectors] We say that a vector $\vec v  \in \mathbb{R}^{n-1}$ is a \emph{trinary vector} (resp. \emph{boolean}) if its coordinates are 0, 1 or -1 (resp. 0 or 1). 
\end{definition}

Let us recall that two nests $N_1$ and $N_2$ are said to be \emph{compatible} if they fulfill Conditions (2) and (3) of \cref{def:nesting}. Moreover, we say that they are \emph{linearly independent} if $\vec N_1$ and $\vec N_2$ are.

\begin{proposition} \label{prop:edges1}
Let $t \in \PT{n}$ and let us denote by $P=P_t$ the standard weight Loday realization of the operahedron. There is a surjection
\begin{eqnarray*}
    \begin{Bmatrix}
        \text{a set of } k \text{ compatible nests,} \\
        \text{a set of } l \text{ compatible nests,} \\
        \text{with } k+l=n-3 \text{ and }k,l\geq 0 \ , \\
        \text{mutually linearly independent and with the trivial nest}
    \end{Bmatrix}
    \twoheadrightarrow
    \begin{Bmatrix}
     \text{direction }\vec d \text{ of an edge} \\
     \text{of } P\cap\rho_z P, \text{ for some } z \in P 
    \end{Bmatrix}_{\big{/\sim}} \ ,
\end{eqnarray*}
where two directions in the target are identified if they are a scalar multiple of each other.
\end{proposition}

\begin{proof} This follows from a direct application of \cref{prop:edgesPP}.
\end{proof}

\begin{definition}[Support and length] The set of non-zero entries of a vector $\vec v \in \mathbb{R}^n$ is called its \emph{support} and the cardinality of this set is called its \emph{length}. 
\end{definition}

\begin{samepage}
\begin{proposition}[Direction of the edges of $P\cap\rho_z P$] \label{prop:charactedge}
Let $t\in\PT{n}$ and let $P=P_t$ be the standard weight Loday realization of the operahedron. Then, representatives for the equivalence classes of directions of the edges of $P\cap\rho_z P$, for all $z \in P$, are given by trinary vectors of $\mathbb{R}^{n-1}$ with the same number of $1$ and $-1$ and whose first non-zero coordinate is $1$.
\end{proposition}
\end{samepage}
    
\begin{proof} 
The space of solutions of the system of linear equations in the left hand side of the surjection in Proposition \ref{prop:edges1} is given by the kernel of the $(n-1) \times (n-1)$ boolean matrix 
\begin{eqnarray}  \label{system2}
\begin{pmatrix}
\text{---} & \vec N_1 & \text{---} \\
\text{---} & \vdots & \text{---} \\
\text{---} & \vec N_{k+1} & \text{---} \\
\text{---} & \vec N_1' & \text{---} \\
\text{---} & \vdots & \text{---} \\
\text{---} & \vec N_{l+1}' & \text{---} \\
\end{pmatrix} \ ,
\end{eqnarray}
where the vectors are written horizontally. The $k+1$ first (resp. $l+1$ last) lines are included in one another as elements of the boolean lattice $\{0,1 \}^{n-1}$. We can thus substract the line of minimal length to the $k$ (resp. $l$) others, then the line with second minimal length to the $k-1$ (resp. $l-1$) others, and so on until we obtain a family of $k+1$ (resp. $l+1$) lines with disjoint support, whose sum is $(1,\ldots, 1)$.

We claim that the system of linear equations obtained in this way is equivalent to a system where the length of each line is at most 2, that is, where each line has a 1 in at most two places. We proceed by induction on $n$. The case $n-1=2$ is clear. Let $n-1\geq 3$ and suppose that the result holds for every matrix of size $k\leq n-2$. Let $M$ be a matrix of size $n-1$ filling the hypothesis. 
\begin{enumerate}
\item Suppose that $M$ contains a line of length 1, that is a line $i$ with zeros everywhere except in place $j$. We can then reduce every non-zero element of the $j$th column to 0 and apply the induction hypothesis to the $(n-2) \times (n-2)$ matrix $M'$ obtained from $M$ by suppressing its $i$th line and $j$th column.
\item Suppose that no line of $M$ has length 1. 
\begin{enumerate}
\item Suppose that $k>l$. The length of the sum of the $k+1$ lines of the first group is at least $2k+2>k+l+2=n-1$, which is impossible. 
\item Suppose that $k=l$. The length of the sum of the $k+1$ lines of the first group is, as for the $l+1$ lines of the second group, exactly $2k+2$, which means that every line has length 2. This finishes the proof of the claim.
\end{enumerate}
\end{enumerate}
The kernel of (\ref{system2}) has dimension 1. Since the vector $(1,\ldots,1)$ is in the system, the coordinates of any non-zero vector in it sum to zero. By the preceding claim, it is a scalar multiple of a trinary vector with the same number of $1$ and $-1$, and whose first non-zero coordinate is $1$.
\end{proof}

\begin{corollary} \label{coroll:permuto}
Let $t\in\PT{n}$ be a 2-leveled tree, and let us denote by $P=P_t$ the standard weight Loday realization of the permutahedron. There is a bijection
\begin{eqnarray*}
    \begin{Bmatrix}
         \text{direction }\vec d \text{ of an edge} \\
         \text{of } P\cap\rho_z P, \text{ for some } z \in P 
    \end{Bmatrix}_{\big{/\sim}}
        \cong
    \begin{Bmatrix}
            \text{trinary vector of} \ \ \mathbb{R}^{n-1} \\
            \text{with the same number of }1\text{ and } -1 \\
            \text{whose first non-zero coordinate is }1
    \end{Bmatrix} \ ,
\end{eqnarray*}
where, in the first set, two linearly dependent directions are identified.
\end{corollary}
    
\begin{proof} We prove that every trinary vector on the right-hand side is a representative of some equivalence class of directions on the left-hand side. Let $\vec d \in \mathbb{R}^{n-1}$ be a vector having $p$ coordinates equal to 1, $p$ coordinates equal to -1 and $q$ coordinates equal to 0 with $p\geq 1, q \geq 0$ and $2p+q=n-1$. We construct a system of nested boolean vectors $\{\vec N_1, \ldots , \vec N_p, \vec N_1', \ldots , \vec N_{q+p}' \}$ that has $\vec d$ as solution. First label the pairs of $\{1,-1\}$ from left to right with $\{1,\ldots,p\}$ and the zeros, if there are any, from left to right with $\{1,\ldots,q\}$. Then,
\begin{enumerate}
    \item If $p=1$, go directly to Step (2). If $p\geq 2$, define $p-1$ boolean vectors $\{\vec N_1,\vec N_2,\ldots, \vec N_{p-1} \}$ by the following: the vector $\vec N_i$ has as support the columns of the $i$th first pairs of 1 and -1.
    \item[\refstepcounter{enumi}(2)] If $p=1$ and $q=0$, go directly to Step (3). Otherwise, define $q+p-1$ boolean vectors $\{\vec N_1',\vec N_2',\ldots, \vec N_{q+p-1}' \}$ by the following: the vector $\vec N_j'$ has as support the columns of the $j$th first zeros if $j\leq q$; otherwise for $j\geq q+1$ (that is, if $p\geq 2$) it has as support the columns of the $q$ zeros, the 1 of the first pair, the -1 of the $(j-q+1)$th pair, and if they exist (that is, if $p\geq 3$ and $j\geq q+2$) all the pairs from 2 to $j-q$.
    \item[\refstepcounter{enumi}(3)] Set $\vec N_p =\vec N_{q+p}' =(1,\ldots,1)$ and add the two vectors to the system.
\end{enumerate}
Choosing such vectors is possible since the permutahedron has as normal vectors to its facets every possible non-zero boolean vector, see Point (2) of \cref{prop:PropertiesKLoday}. It is clear from the construction that the vector $\vec d$ is a basis of the space of solutions to this system, see Figure \ref{fig:boolean}. 
\end{proof}

\begin{figure}[ht!]
\[   \vec d =  \begin{pmatrix}
    1 & 0 & -1 & -1 & 1 & 0 & 0 & -1 & 1 
\end{pmatrix} \]
\begin{eqnarray*}  
      \begin{pmatrix}
        \text{---} & \vec N_1 & \text{---} \rule{0pt}{10pt}\\
        \text{---} & \vec N_2 & \text{---} \rule{0pt}{10pt}\\
        \text{---} & \vec N_3 & \text{---} \rule{0pt}{10pt}\\
        \text{---} & \vec N_1' & \text{---} \rule{0pt}{10pt}\\
        \text{---} & \vec N_2' & \text{---} \rule{0pt}{10pt}\\
        \text{---} & \vec N_3' & \text{---} \rule{0pt}{10pt}\\
        \text{---} & \vec N_4' & \text{---} \rule{0pt}{10pt}\\
        \text{---} & \vec N_5' & \text{---} \rule{0pt}{10pt}\\
        \text{---} & \vec N_6' & \text{---} \rule{0pt}{10pt}
    \end{pmatrix} &=&
    \begin{pmatrix} 
        1 & 0 & 1 & 0 & 0 & 0 & 0 & 0 & 0 \rule{0pt}{10pt} \\
        1 & 0 & 1 & 1 & 1 & 0 & 0 & 0 & 0 \rule{0pt}{10pt}\\
        1 & 1 & 1 & 1 & 1 & 1 & 1 & 1 & 1 \rule{0pt}{10pt}\\
        0 & 1 & 0 & 0 & 0 & 0 & 0 & 0 & 0 \rule{0pt}{10pt}\\
        0 & 1 & 0 & 0 & 0 & 1 & 0 & 0 & 0 \rule{0pt}{10pt}\\
        0 & 1 & 0 & 0 & 0 & 1 & 1 & 0 & 0 \rule{0pt}{10pt}\\
        1 & 1 & 0 & 1 & 0 & 1 & 1 & 0 & 0 \rule{0pt}{10pt}\\
        1 & 1 & 0 & 1 & 1 & 1 & 1 & 1 & 0 \rule{0pt}{10pt}\\
        1 & 1 & 1 & 1 & 1 & 1 & 1 & 1 & 1 \rule{0pt}{10pt}
    \end{pmatrix}  
\end{eqnarray*}
\caption{The result of the procedure described in the proof of Corollary \ref{coroll:permuto} for a vector $\vec d$ with $p=3$ pairs of 1 and -1, and $q=3$ zeros.}
\label{fig:boolean}
\end{figure}

Among the fundamental hyperplane arrangements of the operahedra, the one associated to the permutahedron plays a special role.  

\begin{theorem}[Fundamental hyperplane arrangement of the permutahedron] 
\label{thm:permutohyperplane} 
Let $n\geq 1$, and let us write \[ D(n)\coloneqq \{(I,J) \ | \ I,J\subset\{1,\ldots,n\}, |I|=|J|, I\cap J=\emptyset, \min(I\cup J)\in I \}. \] 
The fundamental hyperplane arrangement of the $(n-1)$-dimensional permutahedron in $\RR^n$ is the set of hyperplanes defined by  
\[ \sum_{i\in I}x_i = \sum_{j \in J}x_j \quad \text{ for all } (I,J) \in D(n) \ . \]
\end{theorem}
\begin{proof} This follows immediately from \cref{coroll:permuto}. 
\end{proof}
    
This hyperplane arrangement is a refinement of the braid arrangement, see \cref{fig:braidrefinement}. Computations show that it is in general not a simplicial arrangement. 

\begin{figure}[ht!]
    \centering
    \includegraphics[width=0.4\linewidth]{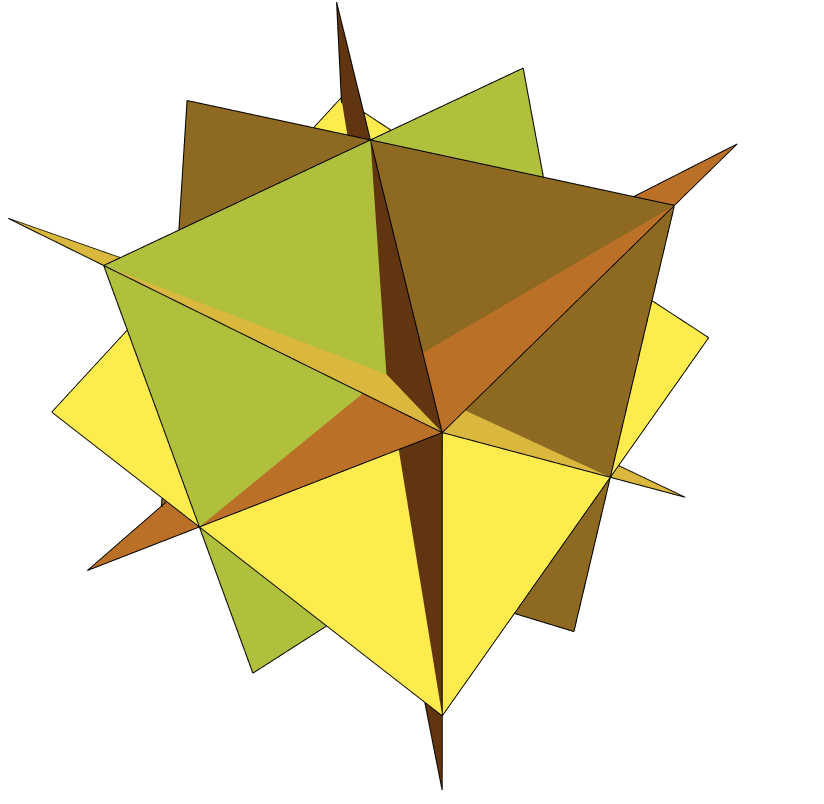}
    \includegraphics[width=0.4\linewidth]{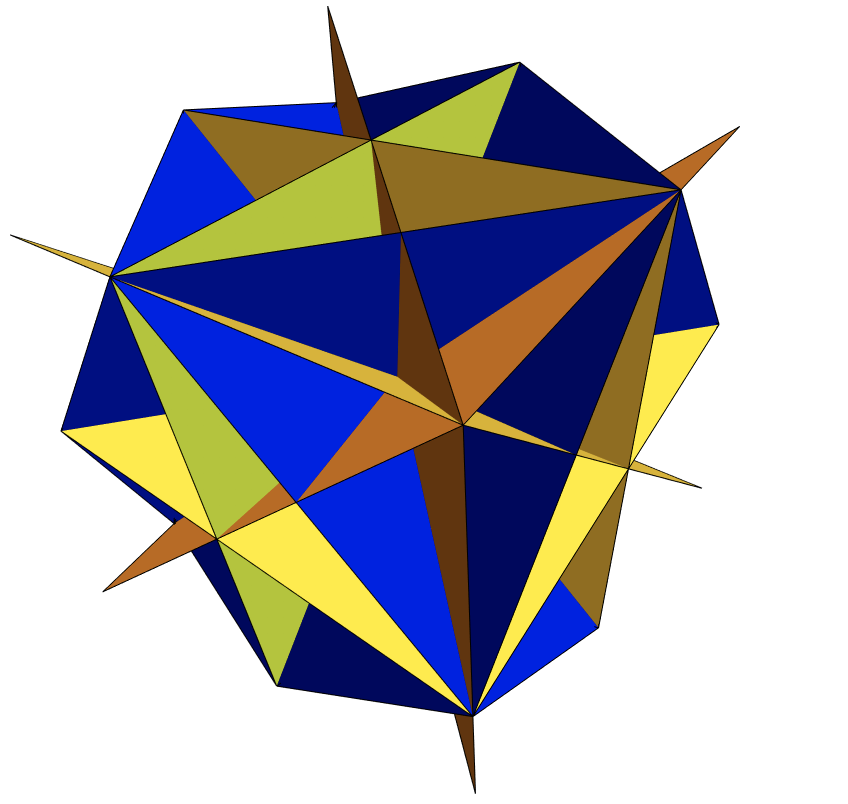}
\caption{The braid arrangement and fundamental hyperplane arrangement of the permutohedron in $\RR^4$, projected into $\RR^3$.}
\label{fig:braidrefinement}
\end{figure}

\begin{remark} 
    \label{rem:zonotope} 
    The fundamental hyperplane arrangement appears as the normal fan of a zonotope, which is itself a facet of the zonotope denoted $H_\infty(d,1)$ in \cite{DezaPourninRakotonarivo21}, the facet contained in the hyperplane $x_1+\cdots + x_d=0$. This last zonotope is related to matroid optimization \cite{DezaManoussakisOnn18} and generalizes L. Billera's \emph{White Whale} \cite{BilleraWhale}, which has been the subject of active research in the recent years. 
\end{remark}

For a tree $t$, we denote by $\mathcal{H}_t$ the fundamental hyperplane arrangement of $P_t$. 

\begin{proposition} \label{prop:inclusionhyperplanes} Let $t, t'\in \PT{n}$ such that $t'$ is a 2-leveled tree. We have $\mathcal{H}_t\subset\mathcal{H}_{t'}$, and if $P_{t'}$ is positively oriented by $\vec v$, then so is $P_t$. 
\end{proposition}
\begin{proof} This is an immediate consequence of \cref{prop:charactedge,coroll:permuto}. Alternatively, it is a special case of the general results \cref{prop:coarseninghyperplanes,coroll:coarseningoriented}.
\end{proof}

\begin{theorem} \label{thm:orientation1}
Let $t\in \PT{n}$ be a tree. Any vector $\vec v=(v_1, \ldots, v_{n-1})\in \RR^{n-1}$ satisfying the equations \begin{eqnarray*}  \sum_{i\in I}v_i \neq \sum_{j \in J}v_j \end{eqnarray*} for all $(I,J) \in D(n-1)$ defines a bot-top diagonal of $P_t$.
\end{theorem}

\begin{proof} This follows directly from \cref{prop:inclusionhyperplanes}.
\end{proof}

\subsection{Universal formula for the operahedra}

We restrict our attention to a certain class of orientation vectors. 

\begin{definition}[Well-oriented realization of the operahedron]
    Let $t\in\PT{n}$ be a tree. A \emph{well-oriented realization of the operahedron} is a positively oriented realization which also induces the poset of maximal nestings $(\mathcal{MN}(t),<)$ on the set of vertices.
\end{definition}

\begin{proposition}\label{prop:nestingoriented}
Let $t\in\PT{n}$ and let $\omega$ be a weight of length $n$. Any vector $\vec v\in \RR^{n-1}$ with strictly decreasing coordinates induces the poset of maximal nestings $(\mathcal{MN}(t),<)$ on the set of vertices of $P_{(t,\omega)}$.
\end{proposition}
\begin{proof} 
Let $\mathcal{N}$ and $\mathcal{N}'$ be two maximal nestings of $t$ corresponding to a covering relation $\mathcal{N} \prec \mathcal{N}'$. They differ only by a nest. We show that the corresponding edge in $P_{(t,\omega)}$ is of the form
\[\overrightarrow{M(t,\mathcal{N},\omega)M(t,\mathcal{N}',\omega)}=(0, \ldots, 0, x , 0, \ldots, 0, -x, 0, \ldots, 0) \ \ \ \ (*) \]
for some $x>0$. We denote by $N$ the unique nest of $\mathcal{N}\setminus\mathcal{N}'$ and by $N'$ the unique nest of $\mathcal{N}'\setminus\mathcal{N}$. Let $j$ and $j'$ be the two edges of $t$ such that $\min\mathcal{N}(j)=N \ \ \text{and} \ \ \min\mathcal{N}'(j')=N'$, see \cref{def:notationgeq}. Now, by the definition of the order on the edges of $t$ we have $j<j'$, see \cref{fig:treeandnesting} and \cref{def:order2}.
We denote by $\alpha_i'\beta_i'$ the $i$th coordinate of the point $M(t,\mathcal{N}',\omega$). The fact that $\min\mathcal{N}(i)=\min\mathcal{N}'(i)$ for all $i\neq j,j'$ implies that $\alpha_i\beta_i=\alpha_i'\beta_i'$ for all $i\neq j,j'$. We show that $\alpha_j\beta_j < \alpha_{j}'\beta_{j}'$. Since the nestings $\mathcal{N}$ and $\mathcal{N}'$ are maximal, we have $\min\mathcal{N}(j')=\min\mathcal{N}'(j)$ and \[ \alpha_j\beta_j = \left( \sum_{k \in V(t_1)}\omega_k\right)\left( \sum_{\ell \in V(t_2)}\omega_\ell\right) < \left( \sum_{k \in V(t_1)}\omega_k\right)\left( \sum_{\ell \in V(t_2)}\omega_\ell + \sum_{\ell \in V(t_3)}\omega_\ell\right) = \alpha_{j}'\beta_{j}' \ , \] where $t_1,t_2$ and $t_3$ are trees with possibly only one vertex, see Definition \ref{def:order2}. Moreover, since $\sum_{i \in E(t)} \alpha_i\beta_i = \sum_{i \in E(t)} \alpha_i'\beta_i'$ is constant, we must have $\alpha_j\beta_j + \alpha_{j'}\beta_{j'} = \alpha_j'\beta_j' + \alpha_{j'}'\beta_{j'}'$ . Defining $x:=\alpha_j\beta_j-\alpha_j'\beta_j'$ we obtain $\alpha_{j'}'\beta_{j'}'-\alpha_{j'}\beta_{j'}=-x$, which proves $(*)$. So $\Big\langle\overrightarrow{M(t,\mathcal{N},\omega)M(t,\mathcal{N}',\omega)}, \vec v\Big\rangle= x(v_j-v_{j'})>0$, and $\vec v$ induces the poset of maximal nestings on the set of vertices. 
\end{proof}

\begin{remark} This poset and an orientation vector inducing it were studied in more depth in \cite[Section 6]{PilaudSignedTree13}. In particular, it is shown that as soon as $t$ is not a linear tree, the poset $(\mathcal{MN}(t),<)$ is never a quotient of the weak order, see \cite[Proposition 86]{PilaudSignedTree13}. 
\end{remark}

Combining the results of \cref{thm:orientation1,prop:nestingoriented}, we obtain the following one. 

\begin{corollary} \label{prop:OrientationVector} Let $t\in\PT{n}$ be a tree and let $\omega$ be a weight of length $n$. Any vector $\vec v=(v_1, \ldots, v_{n-1})\in \RR^{n-1}$ satisfying $v_1>v_2>\cdots >v_{n-1} \quad \text{and} \quad \sum_{i\in I}v_i \neq \sum_{j \in J}v_j$, for all sets of indices $(I,J) \in D(n-1)$, induces a well-oriented realization of the operahedron $(P_{(t,\omega)}, \vec v)$.
\end{corollary}

\begin{remark} \label{rem:weakchambers}
Following \cref{prop:chamberinvariance}, one can wonder how many distinct well-oriented realizations of a given operahedron $P_t$ exist, i.e. how many chambers in $\mathcal{H}_t$ induce a well-oriented realization of $P_t$. For a linear tree, there is only one such chamber \cite[Proposition 3]{MTTV19}. In the case of the permutahedra of dimensions 2, 3 and 4 there are respectively 1, 2 and 12 such chambers. It would be interesting to count the number of chambers in higher dimensions. 
\end{remark}

Now, we make a coherent choice of cellular approximations of the diagonal of the operahedra. By the preceding results, this amounts to a coherent choice of chambers in the fundamental hyperplane arrangement of the permutahedra. We are motivated by the perspective of endowing the family of standard weight Loday realizations of the operahedra with a topological cellular operad structure, which will be done in the next section.

\begin{definition} \label{def:principalvector}
    A \emph{principal orientation vector} is a vector $\vec v \in \mathbb {R}^{n-1}$ such that $\sum_{i\in I}v_i > \sum_{j \in J}v_j$ for all $(I,J) \in D(n-1)$.
\end{definition}

\begin{samepage}
\begin{theorem}[Universal formula for the operahedra] \label{thm:formulaoperahedra} 
    Let $t$ be a tree with $n$ vertices, let $\omega$ be a weight of length $n$, and let $P=P_{(t,\omega)}$ denote the Loday realization of the operahedron. Let $\vec v$ be a principal orientation vector. For $F$, $G$ two faces of $P$ with associated nestings $\mathcal{N}$ and $\mathcal{N}'$, we have
\begin{eqnarray*}
    (F,G)\in \Ima\triangle_{(P,\vec v)} 
    & \iff & \forall (I,J) \in D(n-1), \exists N \in \mathcal{N}, |N\cap I|>|N\cap J| \text{ or } \\
    && \exists N' \in \mathcal{N}', |N'\cap I|<|N'\cap J|  \ .
\end{eqnarray*}
\end{theorem}
\end{samepage}

\begin{figure}[h!]
    \centering
    \includegraphics[width=0.32\linewidth]{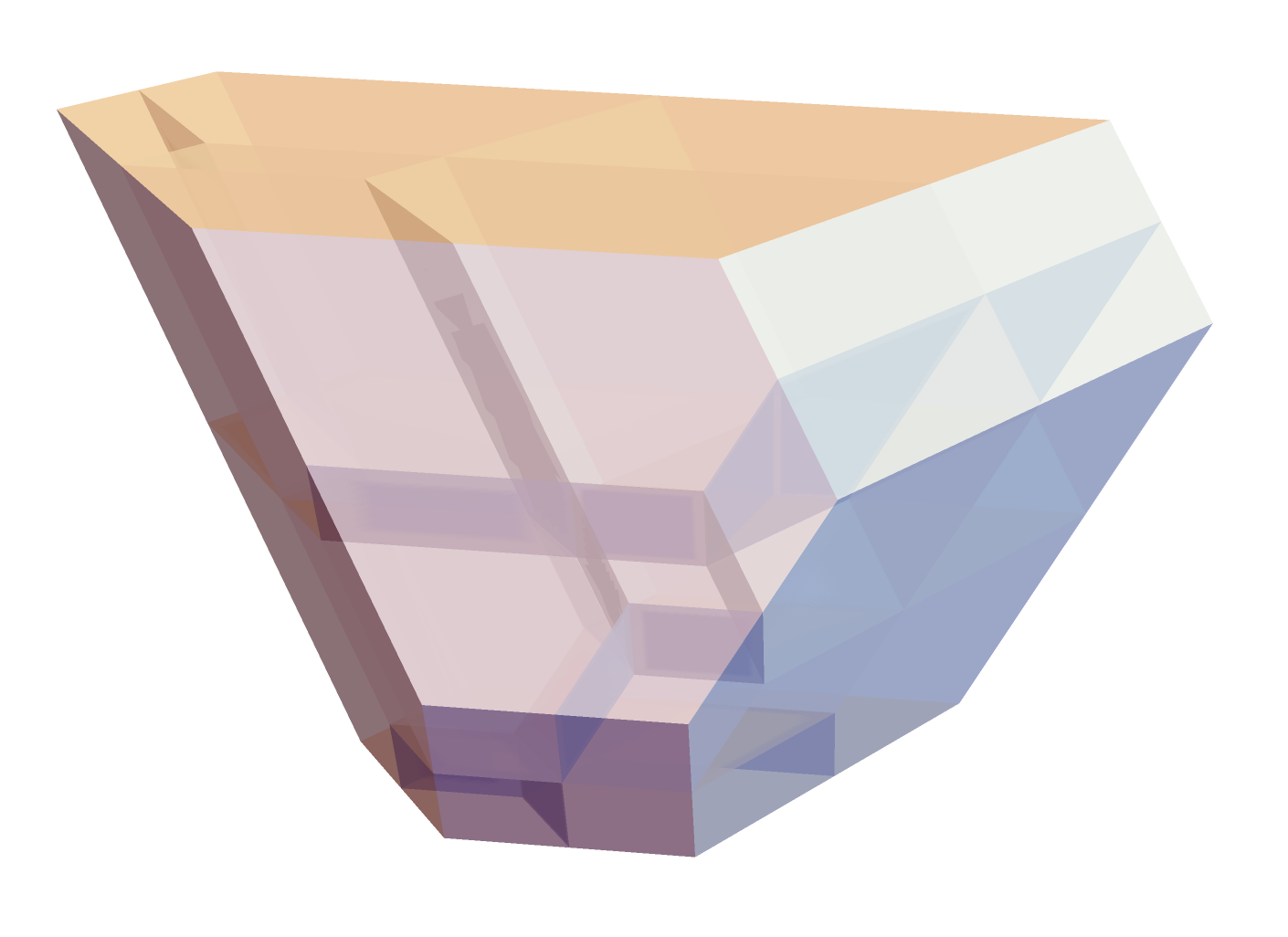} \ 
    \includegraphics[width=0.3\linewidth]{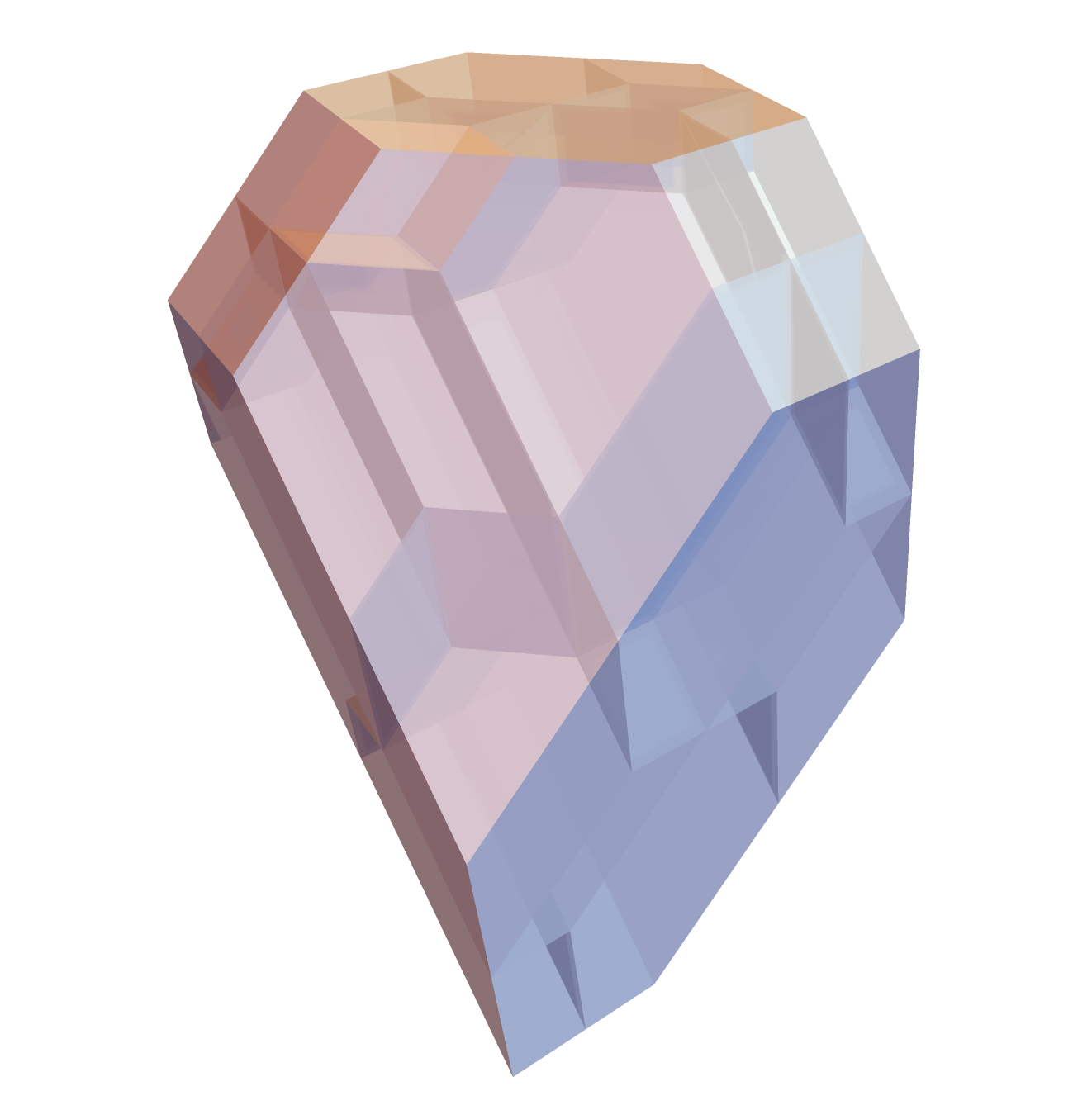} \ \ 
    \includegraphics[width=0.25\linewidth]{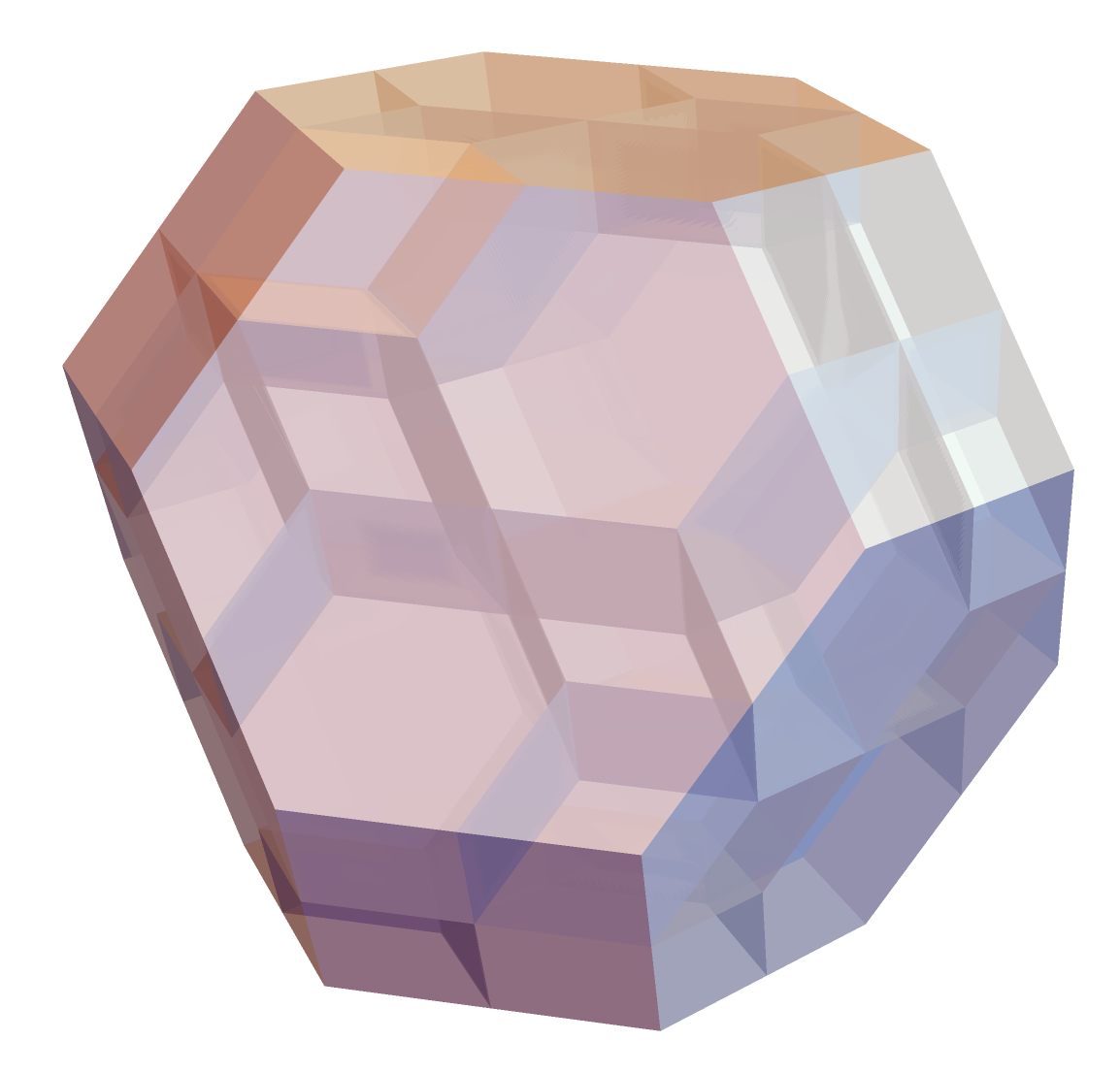}
\caption{Polytopal subdivisions of some 3-dimensional operahedra, from the associahedron (left) to the permutahedron (right), given by the universal formula.}
\label{fig:permutosub}
\end{figure}

\begin{proof} To any $(I,J)\in D(n-1)$, and thus to any hyperplane $H$ in the fundamental hyperplane arrangement of the permutohedron, we associate a normal vector $\vec d_H \in \RR^{n-1}$ by setting $(d_H)_i=1$ for $i\in I$, $(d_H)_j=-1$ for $j\in J$ and $(d_H)_k=0$ otherwise. Let us recall that the normal cone of a face $F$ of $P$, defined by a nesting $\mathcal{N}=\{N_i\}_{1\leq i \leq l+1}$, is given by $\mathcal{N}_P(F)=\cone(-\vec N_1,\ldots,-\vec N_l, -\vec N_{l+1},\vec N_{l+1})$, where $\vec N_{l+1}=(1,\ldots,1)$ is the basis of the orthogonal complement of the affine hull of $P$. Using \cref{prop:inclusionhyperplanes}, we can apply \cref{coroll:coarseninguniversal} and the result follows directly.
\end{proof}

Let us make explicit the image of the diagonal of the permutahedron in low dimensions. The bijection between maximal nestings of a 2-leveled tree and permutations pictured in \cref{fig:bijections} extends in a straightforward manner to all ordered partitions, and we use this more convenient labeling of the faces of the permutahedron to write the image of the diagonal. We also restrict ourselves to the pairs $(F,G)$ such that $\dim F + \dim G=\dim P$, since any other pair can be obtained from these by taking faces. In dimension 1, 2, and 3, we get the following formulas:
\begin{equation*}
    \begin{matrix}
        \triangle_{(P,\vec v)}(12) & = & \textcolor{blue}{1|2 \times 12} \cup \textcolor{blue}{12 \times 2|1}
    \end{matrix}
\end{equation*}

\begin{equation*}
    \renewcommand*{\arraystretch}{1.5}
\begin{matrix}
    \triangle_{(P,\vec v)}(123) 
    & = & \textcolor{blue}{1|2|3 \times 123} & \cup & \textcolor{blue}{123 \times 3|2|1} 
    & \cup & \textcolor{blue}{12|3 \times 2|13} & \cup &  13|2 \times 3|12 \\
    & \cup & \textcolor{blue}{2|13 \times 23|1} & \cup & 1|23 \times 13|2 
    & \cup & \textcolor{blue}{12|3 \times 23|1} & \cup & \textcolor{blue}{1|23 \times 3|12} \\
\end{matrix}
\end{equation*}

\begin{equation*}
    \renewcommand*{\arraystretch}{1.5}
\begin{matrix}
    \triangle_{(P,\vec v)}(1234) 
    & = & \textcolor{blue}{\textbf{1|2|3|4} \times \textbf{1234}} & \cup & \textcolor{blue}{\textbf{1234}\times \textbf{4|3|2|1}}
    & \cup & \textcolor{blue}{\textbf{12|3|4} \times \textbf{2|134}} & \cup & \textbf{134|2} \times \textbf{4|3|12} \\
    & \cup & \textcolor{blue}{12|3|4 \times 23|14} & \cup & \textbf{14|23} \times \textbf{4|3|12}
    & \cup & \textcolor{blue}{2|13|4 \times 23|14} & \cup & \textbf{14|23} \times \textbf{4|13|2} \\
    & \cup & \textbf{13|2|4} \times \textbf{3|124} & \cup & \textbf{124|3} \times \textbf{4|2|13}
    & \cup & \textcolor{blue}{\textbf{1|23|4} \times \textbf{3|124}} & \cup & 124|3 \times 4|23|1 \\
    & \cup & \textbf{1|2|34} \times \textbf{124|3} & \cup & \textcolor{blue}{\textbf{3|124} \times \textbf{34|2|1}}
    & \cup & \textbf{1|3|24} \times \textbf{134|2} & \cup & 2|134 \times 24|3|1 \\
    & \cup & \textbf{1|23|4} \times \textbf{134|2} & \cup & \textcolor{blue}{2|134 \times 4|23|1}
    & \cup & \textcolor{blue}{2|3|14 \times 234|1} & \cup & \textbf{1|234} \times \textbf{14|3|2} \\
    & \cup & \textcolor{blue}{2|13|4 \times 234|1} & \cup & \textbf{1|234} \times \textbf{4|13|2}
    & \cup & \textcolor{blue}{12|3|4 \times 234|1} & \cup & \textcolor{blue}{\textbf{1|234} \times \textbf{4|3|12}} \\
    & \cup & \textbf{1|24|3} \times \textbf{14|23} & \cup & \textcolor{blue}{23|14 \times 3|24|1}
    & \cup & \textbf{1|2|34} \times \textbf{14|23} & \cup & \textcolor{blue}{23|14 \times 34|2|1} \\
    & \cup & \textbf{1|23|4} \times \textbf{13|24} & \cup & 24|13 \times 4|23|1
    & \cup & \textbf{14|2|3} \times \textbf{4|123} & \cup & \textcolor{blue}{\textbf{123|4} \times \textbf{3|2|14}} \\
    & \cup & \textbf{1|24|3} \times \textbf{4|123} & \cup & \textcolor{blue}{123|4 \times 3|24|1}
    & \cup & \textcolor{blue}{\textbf{1|2|34} \times \textbf{4|123}} & \cup & \textcolor{blue}{\textbf{123|4} \times \textbf{34|2|1}} \\
    & \cup & \textbf{3|14|2} \times \textbf{34|12} & \cup & \textbf{12|34} \times \textbf{2|14|3}
    & \cup & \textcolor{blue}{\textbf{1|3|24} \times \textbf{34|12}} & \cup & 12|34 \times 24|3|1 \\
    & \cup & \textbf{13|4|2} \times \textbf{34|12} & \cup & \textcolor{blue}{\textbf{12|34} \times \textbf{2|4|13}}
    & \cup & \textcolor{blue}{\textbf{1|23|4} \times \textbf{34|12}}  & \cup & \textcolor{blue}{12|34 \times 4|23|1} \\
    & \cup & 2|14|3 \times 24|13 & \cup & \textbf{13|24} \times \textbf{3|14|2} 
    & \cup & 12|4|3 \times 24|13 & \cup & \textbf{13|24} \times \textbf{3|4|12} \\
    & \cup & 1|2|34 \times 24|13  & \cup & \textbf{13|24} \times \textbf{34|2|1}
\end{matrix}
\end{equation*}
The pairs in blue describe the image of the diagonal of the associahedron and the pairs in bold belong to the hemiassociahedron. The associated subdivisions of the three 3-dimensional polytopes are shown in \cref{fig:permutosub}. The number of pairs in the image of the diagonal of the permutahedra of dimensions 0 to 7 are given by the sequence 1, 2, 8, 50, 432, 4802, 65536, which coincides with the beginning of the integer sequence A007334 in \cite{OEIS}. 

\medskip

The condition $\tp(F)\leq\bm(G)$ of \cref{prop:topbot} characterizes completely all the pairs in dimension 3, except eight of them: 
$12|34\times 2|4|13$, $12|34\times 24|3|1$, $1|2|34\times 24|13$, $12|4|3\times 24|13$, $13|24\times 3|4|12$, $13|24\times 34|2|1$, $1|3|24\times 34|12$ and $13|4|2\times 34|12$. 
In fact, the condition $\tp(F)\leq\bm(G)$ is equivalent to the conditions in \cref{thm:formulaoperahedra} for the pairs $(I,J)$ of size $|I|=|J|=1$.

\begin{proposition}
\label{prop:inversionnumber}
For any pair of faces $F,G$ of an $(n-1)$-dimensional permutahedron $P\subset \RR^n$, we have that
\begin{eqnarray*}
    (*) \quad \quad \quad \tp(F)\leq\bm(G) 
    & \iff & \forall 1\leq i<j \leq n-1, \exists N \in \mathcal{N}, i \in N, j \notin N \text{ or }  \\
    && \exists N' \in \mathcal{N}', i \notin N', j \in N' \ . \nonumber
\end{eqnarray*}
\end{proposition}

\begin{proof}
We proceed in two steps.
\begin{enumerate}
\item Let us first prove the statement for $F$ and $G$ two vertices of $P$. 
They are associated to ordered partitions of $\{1,\ldots,n-1 \}$, from which one can extract the nestings by reading the connected unions of blocks containing the leftmost element. 
The condition on the right hand side of ($*$) is then exactly the condition that the set of inversions of $F$ is contained in the set of inversions of $G$, and thus that $F\leq G$ with respect to the weak Bruhat order.
\item Now let $F$ and $G$ be two faces of $P$ that are not necessarily vertices.
If they satisfy the condition on the right hand side of ($*$), then $\tp(F)$ and $\bm(G)$ certainly do, so by the preceding point we have $\tp(F) \leq \bm(G)$.
For the reverse implication, suppose that we have $\tp(F) \leq \bm(G)$.
Observe that $\tp(F)$ is obtained from the ordered partition defining $F$ by refining each block in the following way: putting the elements of the block in strictly decreasing order, and then making each of the elements into a new block. 
From this description, it is clear that if $i$ is to the left of $j$ in $\tp(F)$ and $i<j$, then the two must be in distinct blocks of $F$. 
Similarly, if $i$ is to the right of $j$ in $\bm(G)$ and $i<j$, then the two must be in distinct blocks of $G$.
Thus, $F$ and $G$ satisfy the condition on the right hand side of ($*$), which concludes the proof. 
\end{enumerate}
\end{proof}

\begin{remark} 
    \label{rem:combinatorialdescription} 
    When computing $\Ima\triangle_{(P,\vec v)}$, it appears that only a certain proportion of the pairs $(I,J)$ for $|I|=|J|\geq 2$ are necessary. 
    Is there a more "efficient" description of the diagonal? 
    In particular, is there a "purely combinatorial" description in terms of ordered partitions? 
\end{remark}

\begin{remark} 
    \label{rem:permutoSU} 
    The diagonal of the permutahedron considered here differs from that of \cite{SaneblidzeUmble04}, see Example 4 therein. 
    In dimension 3 the two diagonals correspond to the two chambers of the fundamental hyperplane arrangement respecting the weak order, see \cref{rem:weakchambers}. 
    It would be interesting to know if there is a choice of chambers in all dimensions that recovers the diagonal of \cite{SaneblidzeUmble04}. 
    According to \cite[Table 1]{Vejdemo07}, both diagonals have the same number of pairs in dimensions up to 7. 
\end{remark}

From the description of the diagonal of the permutohedron in terms of ordered partitions or nestings of a 2-leveled tree, one obtains the description of the diagonal of any operahedron by applying the coarsening projection of \cref{def:coarseningprojection}.

\begin{proposition}[Coarsening projection for the operahedra] Let $t,t'\in \PT{n}$ be such that $t'$ is a 2-leveled tree but $t$ is not. The coarsening projection $\theta : \mathcal{N}(t')\to\mathcal{N}(t)$ admits the following description. To each nest $N\subset E(t')$, we associate the minimal collection of disjoint nests $N_1,\ldots,N_r\subset E(t)$ such that $\cup_{1\leq i\leq r} N_i= N$. To obtains $\theta(\mathcal{N})$, for any nesting $\mathcal{N}$ of $t'$, we apply this procedure to every nest of $\mathcal{N}$ and then take the union of the resulting nests. 
\end{proposition}
\begin{proof} This follows from a direct translation of \cref{def:coarseningprojection} in terms of nestings. 
\end{proof}

\cref{prop:coarseningcommutes} ensures that the coarsening projection for the operahedra is surjective and commutes with the diagonal maps. We have $|\mathcal{N}|\geq |\theta(\mathcal{N})|$, so the dimension of a face stays the same or diminishes. To obtain the image of the diagonal of $P_t$, one has to apply $\theta$ to the pairs $(F,G)$ with $\dim F + \dim G = \dim P$ and keep only those for which $\dim \theta(F)=\dim F$ and $\dim\theta(G)=\dim G$. In the case where $t$ is a linear tree, one recovers A. Tonks' projection \cite{Tonks97}.

\begin{remark}
    Restricting to linear trees, we recover with a different combinatorial description the "magical formula" of \cite{MarklShnider06, MTTV19} and conjecturally \cite{SaneblidzeUmble04} for the associahedra.
\end{remark}

\section{Tensor product of homotopy operads} \label{section:operadicstructure}

In this section, we show that there exists a topological cellular colored operad structure on the Loday realizations of the operahedra compatible with (in fact, forced by) the above choices of diagonals. Applying the cellular chains functor, we obtain a "Hopf" operad in chain complexes (this operad is not quite a Hopf operad since the diagonal is not strictly coassociative) describing non-symmetric operads up to homotopy, that is non-symmetric operads where the parallel and sequential axioms are relaxed up to a coherent tower of homotopies. The formula for the image of the diagonal obtained in \cref{section:diagonal} allows us to define explicitly, for the first time, the tensor product of two non-symmetric operads up to homotopy.

\subsection{The colored operad encoding the space of non-symmetric operads} 
We work over a field $\mathbb{K}$ of characteristic $0$. 
Note that since the operads that we consider here come from set-theoretic ones, we could as well work over $\mathbb{Z}$.

\begin{definition}[Tree substitution] For any trees $t' \in \PT{k}$ and $t'' \in \PT{l}$, for any vertex $i \in V(t')$ having the same number of inputs as $t''$, we define the tree $t' \circ_i t'' \in \PT{k+l-1}$ obtained by replacing the induced subtree of the vertex $i$ in $t'$ by the tree $t''$. More precisely, the tree $t' \circ_i t''$ has vertices $(V(t')\setminus \{i\})\sqcup V(t'')$ and edges $E(t')\sqcup E(t'')$, see \cref{fig:subnested}.
\end{definition}

\begin{definition}[The colored operad $\mathcal{O}$] We denote by $\mathcal{O}$ the $\NN$-colored operad whose $\NN$-colored collection is defined by
$$\mathcal{O}({n_1,\ldots,n_k}; n)=
    \mathbb{K}\left\{
    \begin{array}{cl}
    \text{Planar tree } t \in \PT{k} \text{ with a bijection } \sigma : \{1,\ldots, k\}\to V(t) \\
    \text{such that } \sigma (i) \text{ has } n_i \text{ inputs for all } i   
    \end{array} 
    \right\} $$
if $n_1+\cdots+n_k-k+1=n$, and the trivial vector space otherwise. Let $t' \in \PT{k}$ and $t'' \in \PT{l}$ be trees with bijections $\sigma'$ and $\sigma''$, respectively. For any $i \in \{1,\ldots,k\}$ such that $\sigma'(i)\in V(t')$ has the same number of inputs as $t''$, we define a partial composition map via tree substitution \[ (t',\sigma')\circ_i (t'',\sigma'') \coloneqq (t' \circ_{\sigma'(i)} t'',\sigma'\circ_i \sigma'') \ , \] where the permutation $\sigma'\circ_i \sigma''$ is defined by 
\[\sigma'\circ_i\sigma''(j)\coloneqq \left\{
    \begin{array}{cll}
    \sigma'(j) & \text{if}\  j<i \\
    \sigma''(j-i+1) & \text{if}\ i\leq j \leq i+l-1 \\
    \sigma'(j-l+1)  & \text{if}\ i+l\leq j \leq k+l-1 \ .
    \end{array} 
    \right. \]
The symmetric groups action is given by precomposition and the corollas $\PT{1}=\{ \mathcal{O}(n;n) \ | \ n \in \NN\}$ give the unit elements.  
\end{definition}

The basis elements of $\mathcal{O}$ are called \emph{operadic trees}. We represent an operadic tree $(t, \sigma)$ by labeling every vertex $j \in V(t)$ with the number $\sigma^{-1}(j)$. If this labeling coincides with the canonical labeling defined in \cref{subsec:whatis}, we say that $t$ is a \emph{left-recursive} operadic tree. 

When we restrict $\mathcal{O}$ to the linear trees where each vertex has only one input (that is, we restrict the set of colors from $\NN$ to $1$), we obtain the symmetric operad $\mathrm{Ass}$ encoding associative algebras. 

When we restrict $\mathcal{O}$ to the two-leveled trees where each vertex of the second level has only one input, and identify all the trees with the same number of vertices and vertices labels, we obtain the permutad $\mathrm{permAs}^{sh}$ encoding associative permutadic algebras \cite[Section 7.6]{LodayRonco11}. Here, substitution is restricted to the vertex of the first level only, in such a way that we stay in 2-leveled trees.  

\begin{proposition}
    Algebras over the colored operad $\mathcal{O}$ are non-unital non-symmetric operads.
\end{proposition}
\begin{proof} We refer to \cite[Section 4]{VanDerLaan03}, \cite[Section 1]{DehlingVallette15} or \cite[Section 2]{Obradovic19} for details. 
\end{proof}

The operation of tree substitution naturally extends to nested trees. For $n\geq 2$, let us denote by $\mathrm{N}\PT{n}$ the set of nested trees with $n$ vertices. By convention, we define $\mathrm{N}\PT{1}:=\PT{1}$. 

\begin{definition}[Nested tree substitution] For any nested trees $(t',\mathcal{N}') \in \mathrm{N}\PT{k}$ and $(t'',\mathcal{N}'') \in \mathrm{N}\PT{l}$, for any $i \in V(t')$ having the same number of inputs as $t''$, we define the nested tree \[ (t',\mathcal{N}') \circ_i (t'',\mathcal{N}'')\coloneqq (t' \circ_i t'', \mathcal{N}'\circ_i \mathcal{N}'') \in \mathrm{N}\PT{k+l-1}\ , \] where $\mathcal{N}'\circ_i \mathcal{N}''=\{(N'\setminus \{i\})\sqcup V(t'') \ | \ N' \in \mathcal{N}'\}\sqcup \mathcal{N}''$. 
\end{definition}

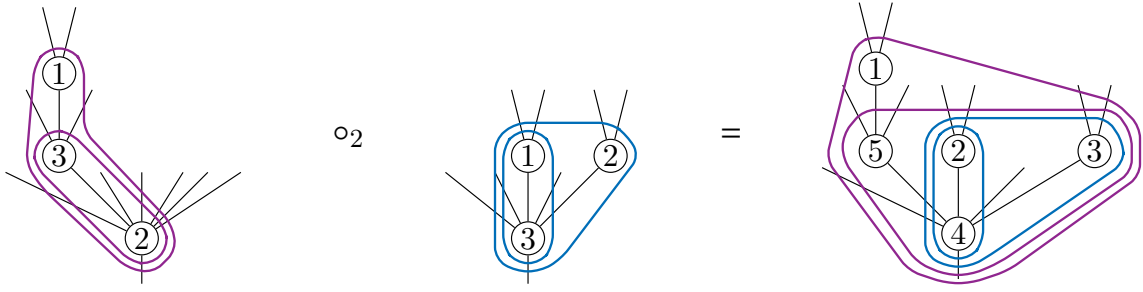
\begin{figure}[h!]
\centering
\resizebox{0.95\linewidth}{!}{
\begin{tikzpicture}
\node (E)[circle,draw=black,minimum size=4mm,inner sep=0.1mm] at (0,0) {\small $2$};
\node (F) [circle,draw=black,minimum size=4mm,inner sep=0.1mm] at (-1,1) {\small $3$};
\node (A) [circle,draw=black,minimum size=4mm,inner sep=0.1mm] at (-1,2) {\small $1$};
\draw[-] (0.8,0.8) -- (E)--(-1.65,0.8); 
\draw[-] (0.5,0.8) -- (E)--(-0.5,0.8); 
\draw[-] (-1.2,2.8) -- (A)--(-0.8,2.8); 
\draw[-] (E)--(0,-0.55); 
\draw[-] (-1.4,1.8) -- (F)--(-0.6,1.8);   
\draw[-] (E)--(F) node {};
\draw[-] (E)--(0,0.8) node  {};
\draw[-] (E)--(1.2,0.8) node {};
\draw[-] (F)--(A) node {};
\draw [Plum,rounded corners,thick] (0.11,-0.32) -- (-0.14,-0.28) -- (-1.28,0.86) -- (-1.32,1.1) --  (-1.1,1.32) -- (-0.86,1.28) -- (0.28,0.14) -- (0.32,-0.11) -- cycle;
\draw [Plum,rounded corners,thick] (0.14,-0.42) -- (-0.18,-0.36) -- (-1.2,0.6) -- (-1.4,0.9) -- (-1.3,2.1) -- (-1.15,2.3) -- (-0.85,2.3) -- (-0.7,2.1) --  (-0.7,1.3) -- (0.36,0.18) -- (0.42,-0.14) -- cycle;
\end{tikzpicture}   \quad \quad \raisebox{4.1em}{$\circ_2$} \quad \quad  
\begin{tikzpicture}
\node (E)[circle,draw=black,minimum size=4mm,inner sep=0.1mm] at (0,0) {\small $3$};
\node (F) [circle,draw=black,minimum size=4mm,inner sep=0.1mm] at (0,1) {\small $1$};
\node (A) [circle,draw=black,minimum size=4mm,inner sep=0.1mm] at (1,1) {\small $2$};
\draw[-] (-0.2,1.8) -- (F)--(0.2,1.8);   
\draw[-] (0.8,1.8) -- (A)--(1.2,1.8);   
\draw[-] (-0.4,0.8)--(E)--(0,-0.55); 
\draw[-] (0.4,0.8)--(E)--(-1,0.8); 
\draw[-] (E)--(F) node {};
\draw[-] (E)--(A) node  {};
\draw [NavyBlue,rounded corners,thick] (-0.15,-0.3) -- (-0.3,-0.1) -- (-0.3,1.1) -- (-0.15,1.3) -- (0.15,1.3) -- (0.3,1.1) -- (0.3,-.1) -- (0.15,-.3) -- cycle;
\draw [NavyBlue,rounded corners,thick] (-0.15,-0.4) -- (-0.4,-0.25) -- (-0.4,1.2) -- (-0.15,1.4) -- (1,1.4) -- (1.2,1.25) -- (1.35,1) -- (0.4,-0.25) -- (0.15,-0.4) -- cycle;
\end{tikzpicture}  \quad \quad \raisebox{4.1em}{$=$} \quad \quad 
\begin{tikzpicture}
\node (E)[circle,draw=black,minimum size=4mm,inner sep=0.1mm] at (0,0) {\small $4$};
\node (F) [circle,draw=black,minimum size=4mm,inner sep=0.1mm] at (-1,1) {\small $5$};
\node (A) [circle,draw=black,minimum size=4mm,inner sep=0.1mm] at (-1,2) {\small $1$};
\node (q) [circle,draw=black,minimum size=4mm,inner sep=0.1mm] at (0,1) {\small $2$};
\node (r) [circle,draw=black,minimum size=4mm,inner sep=0.1mm] at (1.65,1) {\small $3$};
\draw[-] (0.8,0.8) -- (E)--(-1.65,0.8); 
\draw[-] (-1.2,2.8) -- (A)--(-0.8,2.8); 
\draw[-] (-0.2,1.8) -- (q)--(0.2,1.8);   
\draw[-] (1.85,1.8) -- (r)--(1.45,1.8); 
\draw[-] (E)--(0,-0.55); 
\draw[-] (-1.4,1.8) -- (F)--(-0.6,1.8);   
\draw[-] (E)--(F) node {};
\draw[-] (E)--(q) node  {};
\draw[-] (E)--(r) node {};
\draw[-] (F)--(A) node {};
\draw [NavyBlue,rounded corners,thick] (-0.15,-0.3) -- (-0.3,-0.1) -- (-0.3,1.1) -- (-0.15,1.3) -- (0.15,1.3) -- (0.3,1.1) -- (0.3,-.1) -- (0.15,-.3) -- cycle;
\draw [NavyBlue,rounded corners,thick] (-0.15,-0.4) -- (-0.4,-0.25) -- (-0.4,1.2) -- (-0.15,1.4) -- (1.8,1.4) -- (2,1.1) -- (2,0.85) -- (0.4,-0.25) -- (0.15,-0.4) -- cycle;
\draw [Plum,rounded corners,thick] (-0.15,-0.5) -- (-0.5,-0.35) -- (-1.4,0.9) -- (-1.4,1.2) -- (-1.15,1.5) -- (1.9,1.5) -- (2.1,1.2) -- (2.1,0.75) -- (0.5,-0.35) -- (0.15,-0.5) -- cycle;
\draw [Plum,rounded corners,thick] (-0.15,-0.6) -- (-0.6,-0.45) -- (-1.6,0.9) -- (-1.25,2.3) -- (-1,2.4) -- (1.9,1.6) -- (2.2,1.3) -- (2.2,0.7) -- (0.6,-0.45) -- (0.15,-0.6) -- cycle;
\end{tikzpicture} }
\caption{Substitution of nested operadic trees.}
\label{fig:subnested}
\end{figure} 

We note that any nested tree can be obtained from a family of trivially nested trees by successive substitutions. In general, these substitutions can be performed in different orders without changing the resulting nested tree. 

\begin{definition}[Increasing order on nestings] \label{def:nestorder} For a nested tree $(t,\mathcal{N})$, we order the nests of $\mathcal{N}$ by decreasing order of cardinality, and further order nests of same cardinality in increasing order according to their minimal elements. 
We obtain a total order on the set $\mathcal{N}$ and a corresponding unique sequence of substitution of trivially nested trees $(\cdots(t_1\circ_{i_1}t_2)\circ_{i_2} t_3) \cdots \circ_{i_k} t_{k+1})=(t,\mathcal{N})$.
\end{definition}

Let us recall the permutation introduced in Point (5) of \cref{prop:PropertiesKLoday}. 

\begin{definition} \label{def:compositionshuffle}
Let $t$ be a tree, and let $\mathcal{N}$ be a nesting of $t$ with only one non-trivial nest $N$. We contract the nest $N$ to obtain a new tree $t'$. We write $t'' =t(N)$ for the subtree induced by $N$, considered as an independent tree. We write $|E(t')|=p$ and $|E(t'')|=q$ and we renumber the edges of $t''$ from $p+1$ to $p+q$. We denote by $\sigma_N: E(t')\sqcup E(t'')\to E(t)$ the $(p,q)$-shuffle mapping each just renumbered edge of $t'$ and $t''$ to its label in $t$.
\end{definition}

\begin{definition}[The graded colored operad $\mathcal{O}_\infty$] \label{def:NOT} We denote by $\mathcal{O}_\infty$ the graded $\mathbb{N}$-colored operad whose $\NN$-colored collection is given by
$$ \mathcal{O}_\infty(n; {n_1,\ldots,n_k})=
\mathbb{K}\left\{
\begin{array}{cl}
\text{Planar nested tree } (t,\mathcal{N}) \in \mathrm{N}\PT{k} \text{ with a bijection } \sigma : \{1,\ldots, k\}\to V(t) \\
\text{such that } \sigma (i) \text{ has } n_i \text{ inputs for all } i   
\end{array} 
\right\} $$
if $n_1+\cdots+n_k-k+1=n$, and the trivial vector space otherwise. The homological degree of a basis element $(t,\mathcal{N},\sigma)$ is given by $|E(t)|-|\mathcal{N}|$. Let $(t',\mathcal{N}') \in \mathrm{N}\PT{k}$ and $(t'',\mathcal{N}'') \in \mathrm{N}\PT{l}$ be two nested trees with bijections $\sigma'$ and $\sigma''$. For any $i \in \{1,\ldots,k\}$ such that $\sigma'(i)\in V(t')$ has the same number of inputs as $t''$, partial composition is defined via substitution of nested trees \[ (t',\mathcal{N}',\sigma')\circ_i (t'',\mathcal{N}'',\sigma'') \coloneqq \pm ((t',\mathcal{N}') \circ_{\sigma'(i)} (t'',\mathcal{N}''),\sigma'\circ_i\sigma'')  \] where the permutation $\sigma'\circ_i \sigma''$, the symmetric groups action and the units are defined exactly as in $\mathcal{O}$, and the sign is induced by the choice of increasing order on nestings. 
\end{definition}

An example of partial composition in $\mathcal{O}_\infty$ is pictured in \cref{fig:subnested}. The degree 0 part of $\mathcal{O}_\infty$ forms a suboperad made up of fully nested trees.

\begin{proposition} The $\mathbb{N}$-colored operad $\mathcal{O}_\infty$ is free on operadic trees.
\end{proposition}
\begin{proof} Substituting operadic trees produces nested trees. There is a unique way to write a nested tree by iterating this process, up to the parallel and sequential axioms. So, this operad is free.
\end{proof}

Now we turn $\mathcal{O}_\infty$ into a differential graded colored operad. The differential is the unique derivation extending the map $\partial$, defined on trivially nested operadic trees by \[ \partial(t,\mathcal{N},\sigma)\coloneqq -\sum_{N\in\mathcal{N}(t)}(-1)^{|E(t)\setminus N|}\textrm{sgn}(\sigma_N)(t,\mathcal{N}\cup \{N\},\sigma) \ , \] where the sum runs over all nests of $t$.

\begin{proposition} \label{prop:Oinfiniminimal} The dg $\mathbb{N}$-colored operad $\mathcal{O}_\infty$ is the minimal model $\Omega\mathcal{O}^{\as}$ of $\mathcal{O}$.
\end{proposition}
\begin{proof} One can compute the operad $\Omega\mathcal{O}^{\as}$ by using the binary quadratic presentation of $\mathcal{O}$ \cite[Definition 5]{DehlingVallette15}. Using the fact that the colored operad $\mathcal{O}$ is Koszul self-dual \cite[Theorem 4.3]{VanDerLaan03} and the bijection between composite and operadic trees \cite[Section 1.3]{DehlingVallette15}, one obtains $\mathcal{O}_\infty$ as defined above. The sign in the differential comes from the choice of the left-levelwise order on composite trees and application of the Koszul sign rule thereafter. The term $\textrm{sgn}(\sigma_N)$ comes from the decomposition map of $\mathcal{O}^{\as}$ and the term $(-1)^{|E(t)\setminus N|}$ comes from the desuspension in the definition of the differential in the cobar construction. 
\end{proof}

The part of $\mathcal{O}_\infty$ made up of the linear trees where each vertex has only one input gives the minimal model $\mathrm{Ass}_\infty$ of the operad $\mathrm{Ass}$. The part made up of the equivalences classes two-leveled trees with restricted substitution gives the minimal model $\mathrm{permAs}^{sh}_\infty$ of the permutad $\mathrm{permAs}^{sh}$ \cite{LodayRonco11}. Considering the associated non-symmetric operad and permutad, one recovers the minimal models $A_\infty$ \cite[Section 9.2.4]{LodayVallette12} and $\mathrm{permAs}_\infty$ \cite[Section 5.2]{LodayRonco11} -see also \cite{Markl19}, of $\mathrm{As}$ and $\mathrm{permAs}$, respectively.

\begin{remark} \label{rem:ForceyRonco} Under the bijection between nested trees and their tubed line graphs mentioned in \cref{rem:linegraph}, we recover the operation of substitution of tubings defined by S. Forcey and M. Ronco in \cite{ForceyRonco19}, and we observe that the family of clawfree block graphs is stable under this operation.  
\end{remark}

Algebras over the operad $\mathcal{O}_\infty$ are non-symmetric operads up to homotopy, as introduced by P. Van der Laan in \cite{VanDerLaan03}.

\begin{definition}[Non-symmetric operad up to homotopy]
\label{def:homotopyoperad}
A \emph{non-symmetric non-unital operad up to homotopy} is a family of graded vector spaces $\mathcal{P}=\{\mathcal{P}(n)\}_{n\geq 1}$ together with operations \[\mu_t : \mathcal{P}(n_1)\otimes\cdots\otimes \mathcal{P}(n_k)\to \mathcal{P}(n_1+\cdots+n_k-k+1) \] of degree $|E(t)|-1$ for each $t\in \PT{k}$ and all $k\geq 1$, where $n_1,\ldots,n_k$ are the number of inputs of the vertices of $t$, which satisfy the relations \[\sum_{t'\circ_i t''=t} (-1)^{|E(t)\setminus E(t'')|}\emph{sgn}(\sigma_N) \  \mu_{t'}\circ_i \mu_{t''}=0 \ , \] where the sum runs over all the subtrees $t''$ of $t$. 
\end{definition}
The operations $\mu_t$ for the corollas $t\in \PT{1}$ satisfy the relations $\mu_t\circ_1 \mu_t=0$, so they make the spaces $\{\mathcal{P}(n)\}_{n\geq 1}$ into chain complexes. The operations for trees with 2 vertices correspond to partial composition operations $\circ_i$ as in the definition of a non-symmetric operad. The presence of the operations $\mu_t$ for trees $t\in \PT{3}$ indicates that these partial compositions verify the parallel and sequential axioms only up to homotopy. The operations $\mu_t$ for trees $t\in \PT{4}$ are homotopies between these homotopies, and so on. 

\begin{example} 
As proved by P. Van der Laan in \cite[Theorem 5.7]{VanDerLaan03}, the singular $\mathbb{Q}$-chains on configuration spaces of points in the plane form an operad up to homotopy quasi-isomorphic to the operad of singular $\mathbb{Q}$-chains on the little discs operad. 
\end{example}

\subsection{Topological colored operad structure on the operahedra} \label{operadstructure}
In order to contemplate polytopal $\mathbb{N}$-colored operads, we need a suitable symmetric monoidal category of polytopes. We consider the following category, which is a slight modification of the symmetric monoidal category defined in \cite[Section 2.1]{MTTV19}. 
\begin{definition}[The category $\PolySub$]   $   $
\begin{enumerate}
    \item The objects are the disjoint unions $\coprod_{i=1,\ldots,r}P_i$ of non-necessarily distinct polytopes.
    \item Morphisms are disjoint union $\coprod_{i=1,\ldots,r}f_i$ of continuous maps $f_i: P_i\to Q_i$ where for each $i$, $f_i$ sends $P_i$ homeomorphically to the underlying set $|\mathcal{D}_i|$ of a polytopal subcomplex $\mathcal{D}_i\subset \La(Q_i)$ of $Q_i$ 
    such that $f_i^{-1}(\mathcal{D}_i)$ defines a polytopal subdivision of $P_i$.
\end{enumerate}
\end{definition}
The results of \cite{MTTV19} extend in a straightforward manner to this new category. The symmetric monoidal structure is given by the cartesian product of polytopes, and the unit is the trivial polytope made up of one point in $\RR^0$.

\medskip

We want to endow the Loday realizations of the operahedra of standard weight with a colored operad structure in the category $\PolySub$. The underlying set-theoretic operad structure is given on the set of face lattices by substitution of trees. The geometric avatar of this operation is the isomorphism $\Theta: \RR^{k-1} \times \RR^{l-1} \cong \RR^{n-1}$ introduced in the proof of Point~(5) of \cref{prop:PropertiesKLoday}.

\begin{problem} For each operahedron $P_t$, make a choice of an orientation vector $\vec v$ such that the family of diagonal maps $\triangle_{(t,\vec v)}$ commutes with the maps $\Theta$.
\end{problem}

Suppose that we have made a choice of an orientation vector for each operahedron. We fix $t\in \PT{n}$ with chosen orientation vector $\vec v$. We let $t' \in \PT{k}$ and $t'' \in \PT{l}$ be two trees, with chosen orientation vectors $\vec v'$ and $\vec v''$ respectively, such that $t' \circ_i t'' = t$ for some $i \in V(t')$. We denote by $\omega$ the weight $(1,\ldots,1,l,1,\ldots, 1)$ of length $k$, where $l$ is in position $i$. We want the following diagram to commute 
\[
\vcenter{\hbox{
\begin{tikzcd}[column sep=2.2cm, row sep=1.3cm]
P_{(t',\omega)}\times P_{t''}
\arrow[r,  "\Theta"] 
\arrow[d,  "\triangle_{(t',\vec v')}\times\triangle_{(t'',\vec v'')}"]
& P_{t} \arrow[d,  "\triangle_{(t,\vec v)}"] \\
P_{(t',\omega)}\times P_{(t',\omega)}
\times P_{t''} \times P_{t''} 
\arrow[r,  "(\Theta\times\Theta)(\id \times \sigma_2 \times \id)"]
& P_{t} \times P_{t}\ ,
\end{tikzcd}
}}\]
where $\sigma_2$ is the permutation of the two middle blocks of coordinates. The preimage of $\vec v$ under $\Theta$ determines two orientation vectors $\vec w' \in \RR^{k-1}$ and $\vec w'' \in \RR^{l-1}$ of $P_{(t',\omega)}$ and $P_{t''}$ explicitly given by \[ \vec w'= (v_{\sigma (1)},\ldots , v_{\sigma (k-1)}) \quad \text{and} \quad \vec w''= (v_{\sigma (k)},\ldots , v_{\sigma (k+l-1)}) \ , \] where $\sigma$ is a $(k-1,l-1)$-shuffle.  

\begin{proposition} \label{prop:samechambers} Suppose that for each map $\Theta$, the two orientation vectors $\vec w'$ and $\vec w''$ in the preimage $\Theta^{-1}(\vec v)$ are in the same chambers of $\mathcal{H}_{t'}$ and $\mathcal{H}_{t''}$ as $\vec v'$ and $\vec v''$ respectively. Then, the family of diagonal maps $\triangle_{(t,\vec v)}$ commutes with the maps $\Theta$.
\end{proposition}
\begin{proof} \cref{prop:chamberinvariance} shows that $\triangle_{(t',\vec w')}=\triangle_{(t',\vec v')}$ and $\triangle_{(t'',\vec w'')}=\triangle_{(t'',\vec v'')}$. The fact that the above diagram commutes is then straightforward to verify, using the pointwise definition of $\triangle_{(t,\vec v)}$. 
\end{proof}

Recall from \cref{def:principalvector} that a \emph{principal orientation vector} $\vec v \in \mathbb {R}^{n-1}$ is such that $\sum_{i\in I}v_i > \sum_{j \in J}v_j$ for all $(I,J) \in D(n-1)$. 

\begin{proposition} \label{prop:thetacommutes} For any choice of principal orientation vector $\vec v$ for every Loday realization of operahedron $P_t$ of standard weight, the family of diagonal maps $\triangle_{(t,\vec v)}$ commutes with the maps $\Theta$. 
\end{proposition}
\begin{proof} Since $\sigma$ is a $(k-1,l-1)$-shuffle, the two vectors $\vec w'$ and $\vec w''$ are again principal orientation vectors. We conclude with \cref{prop:samechambers}.
\end{proof}

\begin{proposition}[Transition map {\cite[Proposition 7]{MTTV19}}]\label{prop:Transition}
Let $(P, \vec v)$ and $(Q,\vec w)$ be two positively oriented polytopes, with a combinatorial equivalence  $\Phi: \La(P)\xrightarrow{\cong} \La(Q)$.
Suppose that tight coherent subdivisions $\sF_{(P, \vec v)}$ and $\sF_{(Q, \vec w)}$ are combinatorially equivalent under $\Phi\times \Phi$. 
\begin{enumerate}
\item There exists a unique continuous map \[\tr=\tr_P^Q : P\to Q\ ,\]
which extends the restriction of $\Phi$ to the set of vertices and which commutes with the respective diagonal maps. 
\item The map $\tr$ is an isomorphism in the category $\PolySub$, whose correspondence of faces agrees with $\Phi$.
\end{enumerate}
\end{proposition}
The map $\tr$ constructed explicitly in \cite{MTTV19} and has a strong "fractal" character. We fix a tree $t \in \PT{n}$, a weight $\omega$ of the form $(1,\ldots,1,l,1,\ldots,1)$ with $l\geq 1$ and a principal orientation vector $\vec v \in \RR^{n-1}$. We apply \cref{prop:Transition} to the polytopes $P_t$ and $P_{(t,\omega)}$ and obtain a map $\tr : P_t  \longrightarrow P_{(t,\omega)}$.

\begin{definition}[Partial composition and symmetric group action]
We consider the $\mathbb{N}$-colored collection \[O_{\infty}(n_1,\ldots,n_k;n)\coloneqq \coprod_{(t,\sigma) \in \mathcal{O}(n_1,\ldots,n_k;n)}P_t \ . \]
Let $(t',\sigma')$ and $(t'', \sigma'')$ be two composable operadic trees with $k$ and $l$ vertices, respectively. We denote by $(t,\sigma)\coloneqq (t',\sigma') \circ_i (t'',\sigma'')$ their composition at vertex $\sigma'(i) \in V(t')$ in $\mathcal{O}$. We denote by $\omega$ the weight $(1,\ldots,1,l,1,\ldots, 1)$ of length $k$, where $l$ is in position $i$. We define the \emph{partial composition map} by 
\[
\vcenter{\hbox{
\begin{tikzcd}[column sep=1cm]
\circ_i\ : \ P_{t'}\times P_{t''}
\arrow[r,  "\tr\times \id"]
& P_{(t',\omega)}\times P_{t''}
 \arrow[r,hookrightarrow, "\Theta"]
&
P_t \ .
\end{tikzcd}
}}  \]
For $\kappa \in \mathbb{S}_n$, we define the \emph{symmetric group action} on the polytopes associated to $(t,\sigma)$ and $(t,\kappa\circ\sigma)$ by the identity map $P_t\to P_t$. 
\end{definition}

\begin{samepage}
\begin{theorem}\label{thm:MainOperad}\leavevmode

\begin{enumerate}
\item The $\mathbb{N}$-colored collection $\{O_\infty(n_1,\ldots,n_k;n) \ | \ n_1,\ldots,n_k,n \in \mathbb{N}\}$, together with the partial composition maps $\circ_i$, the symmmetric group actions and the 0-dimensional unit elements $\{O_\infty(n;n) \ | \ n \in \mathbb{N} \}$, forms a symmetric colored operad in the category $\PolySub$. 

\item This colored operad structure extends the topological operad structure on the vertices of the operahedra, that is, the fully nested trees.

\item The maps $\{\triangle_{(t,\vec v)} : P_t \to P_t\times P_t\}_{(t,\sigma) \in \mathcal{O}}$, where $\vec v$ are principal orientation vectors, form a morphism of symmetric colored operads in the category $\PolySub$.
\end{enumerate}
\end{theorem}
\end{samepage}

\begin{proof} Once we have in hand \cref{prop:thetacommutes} asserting that the diagonal maps commute with the maps $\Theta$, we can apply the proof of \cite[Theorem 1]{MTTV19} \emph{mutatis mutandis}. The additional facts involving the symmetric groups action and units are straightforward to verify. 
\end{proof}

\begin{remark} \label{rem:automorphisms} The proof of \cref{thm:MainOperad} shows that any family of orientation vectors satisfying \cref{prop:samechambers} induces a colored operad structure on the operahedra. There is more than one such family: consider for instance the vectors $\vec v$ with strictly descreasing coordinates which satisfy $\sum_{i\in I}v_i>\sum_{j\in J}v_j$ for all $I,J\subset \{1,\ldots,n\}$ such that $I\cap J=\emptyset$, $|I|=|J|\geq 2$ and $\max (I\cup J) \in I$. It would be interesting to know how many such families exist, and how they are related to each other.
\end{remark}

\subsection{Tensor product of operads up to homotopy}

We consider the set of all ordered bases of a finite-dimensional vector space $V$. 
We declare two bases \emph{equivalent} if the unique linear endomorphism of $V$ sending one basis to the other has positive determinant. 
In this way, we obtain two equivalence classes of ordered bases. 

\begin{definition} An \emph{orientation} of $V$ is a bijection between the equivalence classes of ordered bases and the set $\{+1,-1\}$. Any basis in the first equivalence class is called a \emph{positively oriented basis}.
\end{definition}

So there are exactly two distinct orientations of $V$. 

\begin{definition}[Cellular orientation of a polytope] Let $P\subset\RR^n$ be a polytope, and let $F$ be a face of $P$. A \emph{cellular orientation of $F$} is a choice of orientation of its linear span. A \emph{cellular orientation of $P$} is a choice of cellular orientation for each face $F$ of $P$. 
\end{definition}

An orientation vector of $P$, in the sense of \cref{def:orientedpolytope}, induces a cellular orientation of the 1-skeleton of $P$. 

\begin{proposition} \label{prop:CWcomplexstructure} A cellular orientation of a polytope $P$ makes it into a regular CW complex. Moreover, the choice of a cellular orientation for every operahedron promotes the colored operad $O_\infty$ to an operad in CW complexes and \cref{thm:MainOperad} holds in this category. 
\end{proposition}
\begin{proof} The choice of a cellular orientation of a face $F$ corresponds to the choice of a generator of the top homology group of $F$. Thus, it makes sense to choose a degree one attaching map from the boundary of the $\dim(F)$-sphere to the boundary of $F$. We endow $P$ with the regular CW structure given by a family of such attaching maps. Now it is clear that the morphisms in the category $\PolySub$ define cellular maps, and that the proof of \cref{thm:MainOperad} can be performed \emph{mutatis mutandis} in the category of CW complexes.
\end{proof}

One can thus apply the cellular chains functor to $O_\infty$ and obtain a colored operad in chain complexes. 

\begin{theorem} \label{thm:choiceorientation} There is a choice of cellular orientation that yields an isomorphism of differential graded symmetric colored operads $C_{\bullet}^{\textrm{cell}}(O_\infty)\cong\mathcal{O}_\infty$.
\end{theorem}
\begin{proof} By definition, the operadic structure of $O_\infty$ coincides cellularly with the operadic structure of $\mathcal{O}_\infty$, and the boundary map of $C_{\bullet}^{\textrm{cell}}(O_\infty)$ coincides up to sign with the differential of $\mathcal{O}_\infty$. We make an explicit choice of orientations and prove that we recover the signs of $\mathcal{O}_\infty$. We build on the work of T. Mazuir who recovered this way in \cite[I,Section 4]{Mazuir21} the signs of the operad $A_\infty$. For a left-recursive operadic tree $(t,\sigma)$, we choose as orientation of the top dimensional cell of $P_t$ the positively oriented basis \[ e_j=(1,0,\ldots,0,-1_{j+1},0,\cdots,0) \ , \] where $-1$ is in position $j+1$ for $j=1,\ldots,n-2$. For any operadic tree $(t,\kappa\circ \sigma)$ obtained from $(t,\sigma)$ by the action of an element $\kappa$ of the symmetric group, we set the orientation of the top-dimensional cell of $P_{(t,\kappa\circ\sigma)}$ to be the orientation of $P_{(t,\sigma)}$ multiplied by $\text{sgn}(\kappa)$. Then, we choose the orientation of any other cell $(t,\mathcal{N})$ of $P_t$ to be the one induced by operadic composition as follows. We consider the unique sequence of substitution of trivially nested trees $(t,\mathcal{N})=(\cdots(t_1\circ_{i_1}t_2)\circ_{i_2} t_3) \cdots \circ_{i_k} t_{k+1})$ arising from the increasing order on $\mathcal{N}$ (\cref{def:nestorder}), and we set the orientation of $(t,\mathcal{N})$ to be the image of the positively oriented basis of the top cells of the polytopes $P_{t_i}$ under this sequence of operations. Computing the signs amounts to comparing bases where the vectors have been permuted. Since we have chosen the increasing order on the nests, we recover precisely the signs involved in the composition of $\mathcal{O}_\infty$. 

We claim that this choice of orientations recovers the signs in the differential of $\mathcal{O}_\infty$. It is enough to consider the boundary map of the top cell of $P_t$. Let $(t,\mathcal{N})$ be a facet of $P_t$ and let $t'$ and $t''$ be two composable operadic trees with trivial nestings such that $t'\circ_i t''=(t,\mathcal{N})$. Let $N$ denote the unique non-trivial nest of $\mathcal{N}$. We denote by $(e_j')_{1\leq j\leq k-1}$ and $(e_j'')_{1\leq j\leq l-1}$, the positively oriented basis associated to $t'$ and $t''$, respectively. We recall the application $\Theta$ and its associated permutation $\sigma_N : E(t')\sqcup E(t'')\to E(t)$ from Point (5) of \cref{prop:PropertiesKLoday}. We choose an outward pointing normal vector $\nu$ to the facet $(t,\mathcal{N})$. The sign associated to this facet in the sum $\partial(t)$ is given by comparing the orientation induced by the operad structure and the orientation of $P_t$. This amounts to computing the determinant $\det{(\nu, \Theta(e_j'), \Theta(e_j''))}$ in the basis $e_j$. We distinguish two cases. 
\begin{enumerate}
    \item If $N$ contains 1, i.e. if $\sigma_N(1)\neq 1$, an outward pointing normal vector $\nu$ is given by forgetting the first coordinate of the vector $\vec N-(1,\ldots,1)$. We have in this case 
    \begin{eqnarray*}
        \Theta(e_j')&=&-e_{\sigma_N(1)-1}+e_{\sigma_N(j+1)-1} \ , \quad 1\leq j\leq k-1 \\
        \Theta(e_j'')&=&e_{\sigma_N(j+1+k)-1} \ , \quad \quad 1\leq j\leq l-1 
    \end{eqnarray*}
    and the value of the determinant is \[\det{(\nu, \Theta(e_j'), \Theta(e_j''))}=-|E(t)\setminus N|(-1)^{|E(t)\setminus N|}\textrm{sgn}(\sigma_N) \ .\] 

    \item If $N$ does not contain 1, i.e. if $\sigma_N(1)= 1$, an outward pointing normal vector $\nu$ is given by forgetting the first coordinate of the vector $\vec N$. We have in this case 
    \begin{eqnarray*}
        \Theta(e_j')&=&e_{\sigma_N(j+1)-1} \ , \quad \quad 1\leq j\leq k-1 \\
        \Theta(e_j'')&=&-e_{\sigma_N(1+k)-1}+e_{\sigma_N(j+1+k)-1} \ ,  \quad 1\leq j\leq l-1 
    \end{eqnarray*}
    and the value of the determinant is \[\det{(\nu, \Theta(e_j'), \Theta(e_j''))}=-|N|(-1)^{|E(t)\setminus N|}\textrm{sgn}(\sigma_N) \ .\] 
\end{enumerate}
We thus recover in both cases the sign of the differential of $\mathcal{O}_\infty$. 
\end{proof}

\begin{corollary} \label{coroll:functorialtensor} The image of the diagonal maps under the cellular chains functor define a morphism of operads in chain complexes $\mathcal{O}_\infty\to\mathcal{O}_\infty\otimes \mathcal{O}_\infty$, and thus a functorial tensor product of non-symmetic operads up to homotopy. 
\end{corollary}
\begin{proof} The cellular chains functor is strong symmetric monoidal and sends the operad $O_\infty$ to the operad $\mathcal{O}_\infty$. For $\mathcal{P}$ and $\mathcal{Q}$ two homotopy operads defined by morphisms of $\mathbb{N}$-colored operads $f:\mathcal{O}_\infty\to \textrm{End}_\mathcal{P}$ and $g:\mathcal{O}_\infty\to \textrm{End}_\mathcal{Q}$, the composite of morphisms \[\mathcal{O}_\infty \xrightarrow{C_\bullet^{\emph{cell}}(\triangle)} \mathcal{O}_\infty\otimes \mathcal{O}_\infty \xrightarrow{f\otimes g}\textrm{End}_\mathcal{P}\otimes \textrm{End}_\mathcal{Q}\to \textrm{End}_{\mathcal{P}\otimes \mathcal{Q}} \ , \] where the last arrow is given by permutation of the factors, defines the structure of an operad up to homotopy on the tensor product of $\mathcal{P}$ and $\mathcal{Q}$. 
\end{proof}

\begin{remark} 
    \label{rem:coassociativity} 
    The diagonal $\triangle_{(t,\vec v)}$ is neither pointwise nor cellular coassociative and the induced diagonal of the dg colored operad $\mathcal{O}_\infty$ is not coassociative either. 
    M. Markl and S. Schnider have actually shown in \cite[Section 6]{MarklShnider06} that such a diagonal \emph{cannot} exist. 
    So the newly defined tensor product cannot make the category of homotopy non-symmetric operads (with strict morphisms) into a symmetric monoidal category. 
\end{remark}
 
We end by computing the signs associated to our choice of cellular orientation for the approximation of the diagonal $\triangle_{(t,\vec v)}$. 

\begin{definition}
\label{def:admissible} 
Let $t$ be a tree and let $\mathcal{N},\mathcal{N}'$ be a pair of nestings such that $|\mathcal{N}|+|\mathcal{N}'|=|V(t)|$. 
An edge $i\in E(t)$ is said to be \emph{admissible} in $\mathcal{N}$ if $i\neq \min(\min \mathcal{N}(i))=:\inf_i(\mathcal{N})$. The set of admissible edges of $\mathcal{N}$ is denoted by $Ad(\mathcal{N})$. 

We give the set $Ad(\mathcal{N})\sqcup Ad(\mathcal{N}')$ a total order by using the increasing order on the nestings (\cref{def:nestorder}) and within a nest by following the numbering of the edges in increasing order. 
Then, the function $\sigma_{\mathcal{N}\mathcal{N}'} : Ad(F)\sqcup Ad(G)\to (1,\ldots, |Ad(\mathcal{N})\sqcup Ad(\mathcal{N}')|)$ defined for $i\in Ad(\mathcal{N})$ by 
\begin{equation*}
\sigma_{\mathcal{N}\mathcal{N}'}(i)=
\begin{cases} 
    \inf_i(\mathcal{N})-1 & \text{ if } i \in Ad(\mathcal{N}) \cap Ad(\mathcal{N}') \text{ and } 1 \neq \inf_i(\mathcal{N}) < \inf_i(\mathcal{N}') \\
    i-1 & \text{ otherwise }
\end{cases}
\end{equation*}
and similarly on $i\in Ad(\mathcal{N}')$ by reversing the roles of $\mathcal{N}$ and $\mathcal{N}'$, induces a permutation of the set $\{1,\ldots, |Ad(\mathcal{N}) \sqcup Ad(\mathcal{N}')|\}$ that we still denote by $\sigma_{\mathcal{N}\mathcal{N}'}$. 
\end{definition}

For convenience, let us recall that  \[ D(n)\coloneqq \{(I,J) \ | \ I,J\subset\{1,\ldots,n\}, |I|=|J|, I\cap J=\emptyset, \min(I\cup J)\in I \} \ . \]

\begin{proposition}[Tensor product of operads up to homotopy] 
\label{cor:tensorproduct} 
Given two non-symmetric non-unital operads up to homotopy $(\mathcal{P},\{\mu_t\})$ and $(\mathcal{Q}, \{\nu_t\})$, their \emph{tensor product} $(\mathcal{P}\otimes\mathcal{Q},\{\rho_t\})$ is given by the Hadamard tensor product of spaces $(\mathcal{P}\otimes\mathcal{Q})(n)\coloneqq \mathcal{P}(n)\otimes\mathcal{Q}(n)$ and the operations \[\rho_t\coloneqq\sum_{\substack{
    \mathcal{N},\mathcal{N}' \in \mathcal{N}(t) \\ |\mathcal{N}|+|\mathcal{N}'|=|V(t)| \\ \forall (I,J) \in D(|E(t)|), \exists N \in \mathcal{N}, |N\cap I|>|N\cap J| \\ \text{ or } \exists N' \in \mathcal{N}', |N'\cap I|<|N'\cap J|  }} (-1)^{|Ad(\mathcal{N})\cap Ad(\mathcal{N}')|} \emph{sgn}(\sigma_{\mathcal{N}\mathcal{N}'}) \ \mathcal{N}(\mu_{t}) \otimes \mathcal{N}'(\nu_t) \ \sigma_t \]
where $\mathcal{N}(\mu_t)$ and $\mathcal{N}'(\nu_t)$ denote the composition of the operations corresponding to the nests of $\mathcal{N}$ and $\mathcal{N}'$ in the increasing orders and where $\sigma_t$ is the isomorphism $\mathcal{P}(n_1)\otimes \mathcal{Q}(n_1)\otimes \cdots \otimes \mathcal{P}(n_k)\otimes \mathcal{Q}(n_k) \cong \mathcal{P}(n_1)\otimes \cdots \otimes \mathcal{P}(n_k) \otimes \mathcal{Q}(n_1)\otimes \cdots \otimes \mathcal{Q}(n_k)$. 
\end{proposition}

Before proceeding to the proof, let us be more precise about what we mean by "the operations corresponding to the nests of $\mathcal{N}$". 
Let $(\mathcal{P},\{\mu_t\})$ be an operad up to homotopy, as in \cref{def:homotopyoperad}, and consider the nested tree $(t,\mathcal{N})$ represented in \cref{fig:treeandnesting}. 
The increasing order on nestings gives the following total order on the nests: $N_1\coloneqq\{1,2,3,4\}< N_2\coloneqq \{1,2,3\}< N_3\coloneqq \{1,2\}< N_4 \coloneqq \{1\}$.
Consider now the tree $t_1$ obtained by contracting inside $t=t(N_1)$ the subtree $t(N_2)$, the tree $t_2$ obtained by contracting inside $t(N_2)$ the subtree $t(N_3)$, the tree $t_3$ obtained by contracting inside $t(N_3)$ the subtree $t(N_4)$, and the tree $t_4=t(N_4)$. 
We have $t=(((t_1 \circ_1 t_2) \circ_1 t_3) \circ_1 t_4)$. 
The composition of operations corresponding to the nesting $\mathcal{N}$ is then $(((\mu_{t_1}\circ_1 \mu_{t_2})\circ_1 \mu_{t_3})\circ_1 \mu_{t_4})$.

\begin{proof}[Proof of {\cref{cor:tensorproduct}}] This is just unravelling the definition of tensor product arising from \cref{coroll:functorialtensor}. For a pair of faces $(F,G)\in \Ima\triangle_{(t,\vec v)}$, the sign comes from the comparison of our choice of orientation on $F\times G$, which is just the product of the orientations of $F$ and $G$, with the orientation induced by the diagonal $\triangle_{(t,\vec v)}$ when restricted to $(\mathring F + \mathring G )/ 2$. Let $(e_j)$ denote as before the positively oriented basis of the top cell of $P_t$. We need to compute the sign of the determinant of the vectors $\triangle_{(t,\vec v)}(e_j)$ expressed in the basis $\{e_j^F \times 0\} \cup \{0 \times e_j^G\}$ corresponding to the orientation of $F\times G$. By the very definition of $\triangle_{(t,\vec v)}$ (\cref{def:Diag}), this is the same as computing the sign of the determinant of the $e_j^F, e_j^G$ expressed in the basis $(e_j)$, which gives the expression appearing above. 
\end{proof}

\bibliographystyle{amsalpha}
\bibliography{bib}

\end{document}